\documentclass[journal,onecolumn]{IEEEtran} 
\ifCLASSINFOpdf
\else
\fi
\usepackage{amsmath,amssymb}
\usepackage{epstopdf}
\usepackage{times}
\usepackage{amssymb}
\usepackage{mathrsfs}
\usepackage{graphicx}
\usepackage{rotating,subfigure}
\usepackage[flushleft]{threeparttable}
\usepackage{multirow}
\usepackage{subfigure}
\usepackage{caption}
\usepackage{epstopdf}

\usepackage{star}

\usepackage[colorlinks,
linkcolor=blue,
anchorcolor=blue,
citecolor=blue
]{hyperref}

\newcommand{\sgn}{\mathop{\mathrm{sign}}}

\numberwithin{equation}{section}

\newcommand{\beq}{\begin{equation}}
\newcommand{\bEq}{\end{equation}}

\newcommand{\al}{\alpha}

\newcommand{\ee}{\end{equation}}

\newcommand{\e}{{\varepsilon}}

\newcommand{\fa}{{\mathfrak a}}

\renewcommand{\cal}{\mathcal}

\newcommand{\wh}{\widehat}
\newcommand{\wt}{\widetilde}

\newcommand{\ii}{\mathrm{i}} 
\newcommand{\dd}{\mathrm{d}}

\newcommand{\id}{\mspace{2mu}\mathrm{i}\mspace{-0.6mu}\mathrm{d}} 
\renewcommand{\epsilon}{\varepsilon}
\renewcommand{\leq}{\leqslant}
\renewcommand{\geq}{\geqslant}

\renewcommand{\le}{\leq}
\renewcommand{\ge}{\geq}

\newcommand{\E}{\mathbb{E}}
\newcommand{\R}{\mathbb{R}}
\newcommand{\C}{\mathbb{C}}
\newcommand{\N}{\mathbb{N}}

\DeclareMathOperator{\supp}{supp}

\DeclareMathOperator{\re}{Re}
\DeclareMathOperator{\im}{Im}

\DeclareMathOperator{\OO}{O}
\DeclareMathOperator{\oo}{o}

\DeclareMathOperator{\bbN}{\mathbb{N}}
\DeclareMathOperator{\bbP}{\mathbb{P}}

\begin{document}
%
\title{Tracy-Widom distribution for heterogeneous Gram matrices with applications in signal detection}


%


\author{Xiucai~Ding and Fan~Yang
	\thanks{X. Ding is with the Department of Statistics, University of California, Davis (e-mail: xcading@ucdavis.edu). F. Yang is with the Department of Statistics and Data Science, University of Pennsylvania (e-mail: fyang75@wharton.upenn.edu).}
}

\maketitle

\begin{abstract}
Detection of the number of signals corrupted by high-dimensional noise is a fundamental problem in signal processing and statistics. This paper focuses on a general setting where the high-dimensional noise has an unknown complicated heterogeneous variance structure. We propose a sequential test which utilizes the edge singular values (i.e., the largest few singular values) of the data matrix. It also naturally leads to a consistent sequential testing estimate of the number of signals. We describe the asymptotic distribution of the test statistic in terms of the Tracy–Widom distribution. 
The test is shown to be accurate and have full power against the alternative, both theoretically and numerically. The theoretical analysis relies on establishing the Tracy-Widom law for a large class of Gram type random matrices with non-zero means and completely arbitrary variance profiles, which can be of independent interest.         
\end{abstract}

\section{Introduction}

Detection of unknown noisy signals is a fundamental task in many signal processing and wireless
communication applications \cite{AN17,kay1998fundamentals,NE1,onatski2014}. Consider the following generic signal-plus-noise model 
\begin{equation}\label{eq_signalmodel}
\bm{y}= \bm{s}+ \bm{z},
\end{equation}
where $\bm{s}$ and $\bm{z}$ are independent $p$-dimensional centered signal and noise vectors, respectively. 
In many applications, $\bm{s}$ is usually generated from a low-dimensional MIMO filter such that $\bm{s}=\Gamma \bm{\nu}$ \cite{kay1998fundamentals}, where $\Gamma$ is a $p \times r$ deterministic matrix, $\bm{\nu}$ is an $r$-dimensional centered random vector and $r$ is some unknown fixed integer that does not depend on $p$. The value of $r$ is one of the most important inputs for many computationally demanding parametric procedures such as direction of arrival estimation, blind source
deconvolution, and so on. In the literature of statistical signal processing, the most common approaches to determine the value of $r$ are perhaps the information theoretic criteria, including the minimum description length (MDL), Bayesian information criterion (BIC) and Akaike information criterion (AIC) and their variants. For a detailed review of this aspect, we refer the reader to \cite{1311138}. All these methods assume that the dimension $p$ is fixed and the sample size $n$, i.e., the number of observations, goes to infinity. Consequently, none of these estimators is applicable to large arrays where the number of sensors is comparable to or even larger than the sample size \cite{KN}.

To address the issue of high dimensionality, many methods and statistics have been proposed to infer the value of $r$ under various settings. Many methods have been proposed to test ${\mathbf{H}}_0: r=0$ against ${\mathbf{H}}_a: r\ge 1,$ 
which is equivalent to testing the existence of the signals. When $\bm{z}$ is a white noise, a non-parametric method was proposed in \cite{KN}, the generalized maximum likelihood test was studied \cite{BDMN} and a sample eigenvalue based method was proposed in \cite{NE1}. When $\bm{z}$ is a colored noise, i.e., $\bm{z}=\Sigma^{1/2}\bm{x}$ for a positive definite covariance matrix $\Sigma$ and a white noise $\bm{x},$ the same testing problem has been considered in \cite{bao2015,9103128,5447639, 6588599} under different moment assumptions on the entries of $\bm{x}.$ However, all the aforementioned methods assume explicitly that the noise vectors $\bm{z}_1,\cdots, \bm{z}_n$ are generated independently from the same distribution. 
If the noise vectors are correlated or generated from possibly different distributions, none of these methods works or has been justified rigorously. One such example is the doubly heteroscedastic noise, whose matrix of noise vectors $(\bm{z}_1, \cdots, \bm{z}_n)$ take the form $A^{1/2}\cal N B^{1/2}$ \cite{2019arXiv190209474L}, where $\cal N$ is a $p\times n$ white noise matrix, and $A$ and $B$ are two positive definite symmetric matrices representing the spatial and temporal covariances, respectively. Many previous works also depend crucially on the null hypothesis $r=0$, and cannot be applied to the more general setting with null hypothesis $r=r_0$ for a fixed $r_0\ge 0$. 

\subsection{Problem setup and test statistics}\label{sec_setup}

In this paper, we present a more general setting for the statistical analysis of the detection of the number of signals. On the one hand, we propose some statistics to study the following hypothesis testing problem 
\begin{equation}\label{eq_test1}
\mathbf{H}_0: r=r_0 \quad \text{vs} \quad \mathbf{H}_a: r>r_0,
\end{equation}  
where $r_0$ is some pre-given integer representing our belief of the true value of $r.$ (\ref{eq_test1}) generalizes the previous works, which mainly focus on the $r_0=0$ case, i.e., the testing of the existence of signals. On the other hand, we consider more general covariance structures of the noise, which include the doubly-heteroscedastic noise, sparse noise and noise with banded structures as special cases. We refer the readers to Examples \ref{example_separable} and \ref{exam_sparserandomgram} and the simulation settings in Section \ref{sec_stat} for more details. We emphasize that through (\ref{eq_test1}), a natural   consistent sequential testing estimate of $r$ can be generated, that is, 
\begin{equation}\label{eq_seqestimate}
\widehat{r}:=\inf\{r_0 \geq 0: \ \mathbf{H}_0 \ \text{is accepted} \}.
\end{equation} 
We refer the readers to (\ref{eq_defn_consistentesmtiator}) and Corollary \ref{coro_onatskiresult} for more rigorous arguments on this aspect. 

In order to test (\ref{eq_test1}), we propose some data-adaptive statistics utilizing the edge eigenvalues of the data matrix. Suppose we observe $n$ data samples 
and stack them into the data matrix  
\begin{equation}\label{eq_matrixdenoisingtest}
Y=R+Z,
\end{equation} 
where $Y=(\bm{y}_1,\cdots, \by_n) \in \mathbb{R}^{p \times n}$ collects the noisy observations, $R=(\bm{s}_1,\cdots, \bm{s}_n)$ is the signal matrix of rank $r$, and $Z=(\bm{z}_1,\cdots, \bm{z}_n)$ is the noise matrix. The matrix (\ref{eq_matrixdenoisingtest}) is commonly referred to as the \emph{signal-plus-noise matrix} in the literature, which is also closely related to the problem of low-rank matrix denoising \cite{BDW, ding2020, DD1995, RRN14, TS93,YMB}.  In the current paper, we consider the high-dimensional regime where $p$ and $n$ are comparably large so that 
\begin{equation*}
\tau \leq  {p}/{n} \leq \tau^{-1},  
\end{equation*}
for a small constant $0<\tau<1$. 
We assume that the entries of $Y$ are independent random variables satisfying that 
\begin{equation}\label{eq_gramgeneralassumption}
\mathbb{E} y_{ij}=r_{ij},\quad \Var( y_{ij}) =s_{ij}.
\end{equation}
Correspondingly, we will also call $R=(r_{ij})$ the \emph{mean matrix}, while the \emph{variance matrix} $S=(s_{ij})$ describes a heterogeneous variance profile for the noise. In this paper, we refer to $YY^\top$ as a \emph{random Gram matrix}.  
We mention that the detection of the number of signals has been studied rigorously in the literature only when $S$ is of sample covariance type, that is, $s_{ij}=a_i$ for some $a_i>0$. Even for the doubly-heteroscedastic noise with $s_{ij}=a_i b_j$ for some $a_i, b_j>0,$ the aforementioned testing methods in the literature will lose their validity.

There exists a vast literature on conducting high-dimensional statistical inference using the largest eigenvalues of $YY^\top$ when $S$ is of sample covariance type. 
For instance, they have been employed to test the existence and number of spikes for the spiked covariance matrix model \cite{johnstone2020,onatski2014}, test the number of factors in factor model \cite{OAea}, detect the signals in a signal-plus-noise model \cite{AN17, bao2015,BDMN,6588599}, test the structure of covariance matrices \cite{elkaroui2007,han2016}, and perform the multivariate analysis of variance (MANOVA) \cite{ZJ2017, han2016}. 
In most of these applications, on the one hand, researchers aim to test between the null hypothesis of a non-spiked sample covariance matrix and the alternative of a spiked sample covariance matrix. Under the null hypothesis, the largest few eigenvalues have been proved to satisfy the Tracy-Widom law asymptotically under a proper scaling \cite{bao2015, DY, elkaroui2007,IJ2, Anisotropic, LS, Regularity4,pillai2014}. More precisely, there exist parameters $\lambda_+$ and $\varpi$ such that $\varpi p^{2/3}(\lambda_{1}-\lambda_+)$ converges in law to the type-1 Tracy-Widom distribution \cite{TW1,TW}, where $\lambda_1$ is the largest eigenvalue of $YY^\top$. Then it is natural to choose $\varpi p^{2/3}(\lambda_{1}-\lambda_+)$ as the test statistic. 
On the other hand, especially in the setting of factor models in economics, researchers are interested in inferring the number of factors. Under the null hypothesis that there are $r$ large factors, the $(r+1)$-th eigenvalue $\lambda_{r+1}$ obeys the Tracy-Widom distribution asymptotically \cite{OAea}.

Based on the above observations, if we can show that $\lambda_{r+1}$ obeys the Tracy-Widom law in our setting 
(\ref{eq_gramgeneralassumption}),
we can naturally choose $\varpi p^{2/3}(\lambda_{r+1}-\lambda_+)$ as the test statistic for the testing problem (\ref{eq_test1}). However, in practice, the two parameters $\varpi$ and $\lambda_+$ depend on the usually unknown variance matrix $S$. To resolve this issue, we can follow \cite{OAea} to use the statistic 
\begin{equation}\label{eq_ona intro}
 \mathbb{T} \equiv \mathbb{T}(r_0):=\max_{r_0< i \le r_*}\frac{\lambda_{i}-\lambda_{i+1}}{\lambda_{i+1}-\lambda_{i+2}},
\end{equation} 
where $\lambda_1 \geq \lambda_2 \geq \cdots \ge \lambda_p$ are the eigenvalues of $YY^\top$ arranged in descending order,  and $r_*$ is a pre-chosen integer that is interpreted as the maximum possible number of signals the model can have. { We will also see in Section \ref{sec_statisticsproperty}  that (\ref{eq_ona intro}) can be used to count the number of outlier eigenvalues that correspond to signals through a sequential testing procedure.}
Onatski \cite{OAea} observed that in the setting of sample covariance matrices, $\mathbb T$ is independent of $\varpi$ and $\lambda_+$ under the null hypothesis, and hence is asymptotically pivotal. Moreover, its asymptotic distribution is determined by the Tracy-Widom law of the edge eigenvalues. Consequently, we can approximate the distribution of $\mathbb{T}$ using Monte Carlo simulations of Wishart matrices. 

We point out that in many literature and scientific applications \cite{BDW,2020arXiv200909177J,Marcotte751,RRN14,6588599}, it is reasonable to assume that the signals are distinct. Under this assumption, we also propose the following statistic
\begin{equation}\label{eq_ona intro1}
\mathbb{T}_{r_0}:=\frac{\lambda_{r_0+1}-\lambda_{r_0+2}}{\lambda_{r_*+1}-\lambda_{r_*+2}}.
\end{equation}
Compared to (\ref{eq_ona intro}),  the  statistic (\ref{eq_ona intro1}) relies on fewer (actually, only three or four) sample eigenvalues. Moreover, 
{ for commonly used alternatives with low-rank signals, we expect that the statistic (\ref{eq_ona intro1}) has better performance in terms of power (i.e., it is sensitive to a wider class of alternatives and has higher power for some fixed alternative). Our expectation, although without full theoretical justification, is partly due to the fact that $\mathbb T_{r_0}$ has smaller critical values compared to $\mathbb T$ as illustrated in Table \ref{table_criticalvalue}, which is reasonable because taking maximum over a sequence of random variables increases critical values.} Empirically, our simulations in Section \ref{sec_stat} will show that  \eqref{eq_ona intro1} indeed has better finite-sample performance than \eqref{eq_ona intro} in terms of power. In fact, we believe that the statistic \eqref{eq_ona intro1} will also work when the signals are degenerate, because the corresponding sample eigenvalues will be separated. We refer the reader to Remark \ref{rem_degen} for more details.

The statistics (\ref{eq_ona intro}) and (\ref{eq_ona intro1}) are applicable to statistical inference only if the Tracy-Widom law has been established for the associated random Gram matrix $YY^\top$. However, to the best of our knowledge, this has only been proved rigorously for sample covariance type random Gram matrices in the literature. 
Therefore, for hypothesis testing problems involving random Gram matrices with general mean and variance profiles, 
we need to prove the Tracy-Widom fluctuation rigorously before validating the use of $\mathbb{T}$ and $\mathbb{T}_{r_0}.$ This motivates us to study the limiting distributions of the edge eigenvalues in the general setup (\ref{eq_gramgeneralassumption}). 
Here the notion ``edge eigenvalues" refers to the largest few eigenvalues near the right edge of the bulk eigenvalue spectrum, excluding the outliers of $YY^\top$ caused by the signals. 


\subsection{Tracy-Widom distribution for random Gram matrices}

The Tracy-Widom law for the edge eigenvalues of non-spiked sample covariance matrices has been proved in a series of papers. For Wishart matrices, 
it was first proved in \cite{IJ2} that the largest eigenvalue satisfies the Tracy-Widom law asymptotically. This result was later extended to more general sample covariance matrices with generally distributed entries (assuming only certain moment assumptions) and variance profiles $s_{ij}=a_i$ (assuming certain regularity conditions on the sequence $\{a_i:1\le i \le p\}$) in a series of papers under various settings; see e.g. \cite{bao2015, DY, elkaroui2007, Anisotropic, LS, Regularity4,pillai2014}. 
However, when the mean and variance profiles of the random Gram matrix become more complicated, much less is known about the limiting distribution of the edge eigenvalues. 

In this paper, motivated by the applications in signal detection as discussed in Section \ref{sec_setup}, we establish the Tracy-Widom asymptotics for the edge eigenvalues of a general class of random Gram matrices. The informal statement is given in Theorem \ref{informal1}. Following the conventions in the random matrix theory literature, we shall rescale the matrix $Y$ properly so that the limiting ESD of $ YY^\top$ is compactly supported as $n\to \infty$. 
Moreover, recall that GOE (Gaussian orthogonal ensemble) refers to symmetric random matrices of the form $H:=(X+X^\top)/\sqrt{2},$ where $X$ is a $p \times p$ matrix with i.i.d. real Gaussian entries of mean zero and variance $p^{-1}$. In this paper, we will consistently denote the eigenvalues of $H$ by 
\begin{equation}\label{eq GOE} 
\mu^{\rm GOE}_1 \geq \mu^{\rm GOE}_2 \geq \cdots \ge \mu^{\rm GOE}_p.
\end{equation}
 
\begin{theorem}[Informal statement of Theorem \ref{thm_twgram}]\label{informal1}
For $Y$ satisfying (\ref{eq_gramgeneralassumption}), we denote the eigenvalues of $\cal Q:=YY^\top$ by $\lambda_1\ge \lambda_2 \ge \cdots \ge \lambda_p$. Let $\lambda_+$ be the rightmost edge of the limiting bulk eigenvalue spectrum, and $a\in \N$ be the index of the largest edge eigenvalue. Then, there exists a deterministic sequence of numbers $\varpi\equiv \varpi (R,S)$ depending on $R$ and $S$, such that for any fixed $k\in \N$, the first $k$ rescaled edge eigenvalues, $\{\varpi p^{2/3}(\lambda_{a+i}-\lambda_+):0\le i \le k -1 \}$, have the same asymptotic joint distribution as the first $k$ rescaled eigenvalues of GOE, $\{p^{2/3}(\mu^{\rm GOE}_{i}-2):1\le i \le k \}$, as $p\to \infty$. 
\end{theorem}

It is well-known that $p^{2/3}(\mu^{\rm GOE}_{1}-2)$ converges to the type-1 Tracy-Widom distribution \cite{TW1,TW}. Furthermore, for any fixed $k\in \N$, the joint distribution of the largest $k$ eigenvalues of GOE can be written in terms of the Airy kernel \cite{Forr}. Hence Theorem \ref{informal1} gives a complete description of the finite-dimensional correlation functions of the edge eigenvalues of $\mathcal Q$. Once Theorem \ref{informal1} is established, we can determine the asymptotic distributions of the statistics \eqref{eq_ona intro} and \eqref{eq_ona intro1}, and apply them to the hypothesis testing problem (\ref{eq_test1}).   

Our proof of Theorem \ref{informal1} is based on the following result on the edge eigenvalues of a general class of Gaussian divisible random Gram matrices. 

\begin{theorem}[Informal statement of Theorem \ref{thm_regularbm}]\label{informal2}
For a parameter $t>0$, we denote $\cal Q_t:=(Y+\sqrt{t}X)(Y+\sqrt{t}X)^\top$, where $X$ is a $p\times n$ random matrix independent of $Y$ and has i.i.d. Gaussian entries of mean zero and variance $n^{-1}$. Denote the eigenvalues of $\cal Q_t$ by $\lambda_1(t)\ge \lambda_2(t) \ge \cdots \ge \lambda_p(t)$. Let $ \eta_* > 0$ be a scale parameter depending on $n$. Suppose the empirical spectral distribution of $\cal Q=YY^\top$ has a regular square root behavior near the right edge $\lambda_+$ on any scale larger than $\eta_*$ (in the sense of Definition \ref{assumption_edgebehavior} below). 
Let $a\in \N$ be the index of the largest edge eigenvalue.  Then for any $t\gg \sqrt{\eta_*}$ and fixed $k\in \N$, there exist deterministic sequences of numbers $\varpi_t $ and $\lambda_{+,t} $ such that the first $k$ rescaled edge eigenvalues of $\cal Q_t$, $\{\varpi_t p^{2/3}(\lambda_{a+i}(t)-\lambda_{+,t}):0\le i \le k -1 \}$, have the same asymptotic joint distribution as the first $k$ rescaled eigenvalues of GOE, $\{p^{2/3}(\mu^{\rm GOE}_{i}-2):1\le i \le k \}$, as $p\to \infty$. 
\end{theorem}

On one hand, Theorem \ref{informal2} covers more general matrices than the random Gram matrices proposed in \eqref{eq_gramgeneralassumption}, because it only requires a regular square root behavior of the ESD near the right edge without assuming any independence between matrix entries of $Y$. We remark that the square root behavior of the ESD is generally believed to be a necessary condition for the appearance of the Tracy-Widom law in the asymptotic limit. For example, if the ESD has a cubic root behavior, then the corresponding cusp universality is different from the Tracy-Widom law \cite{Cusp2,Cusp1}. On the other hand, Theorem \ref{informal2} gives the Tracy-Widom law for the edge eigenvalues of a different matrix $\cal Q_t$ other than $\cal Q$. To obtain the Tracy-Widom law for the original matrix $\cal Q$, we still need to show that the edge eigenvalues of $\cal Q_t$ have the same joint distribution as those of $\cal Q$ asymptotically, which, however, is not always true. In fact, if $t$ is too large, then the edge statistics of $\cal Q_t$ can be very different from those of $\cal Q$. For example, if $Y$ is a rectangular matrix whose singular values are all the same, then $\cal Q$ trivially has a square root behavior on any scale larger than $\eta_* =1$ in the sense of Definition \ref{assumption_edgebehavior}. But in the setting of Theorem \ref{informal2}, for $t\gg 1$, the edge statistics of $\cal Q_t$ is dominated by a Wishart matrix $ t X X^\top$.   
 
From the above discussions, we see that in order to prove the Tracy-Widom law for the edge eigenvalues of $\cal Q$ using Theorem \ref{informal2}, we need to establish the following two results:
\begin{itemize}
\item the ESD of $\cal Q$ has a regular square root behavior near $\lambda_+$ on a sufficiently fine scale $\eta_*\ll 1$;
\item for some $\sqrt{\eta_*}\ll t\ll 1$, the edge statistics of $\cal Q_t$ match those of $\cal Q$ asymptotically. 
\end{itemize}
In random matrix theory, there is a general way to accomplish this by using some sharp estimates, called \emph{local laws}, on the resolvent of $\cal Q$, defined as $G(z):=(\cal Q-z)^{-1}$ for $z\in  \C.$ Such local laws for the model \eqref{eq_gramgeneralassumption} have been proved in \cite{alt20172, alt20171} under quite general conditions. 
Combining these local laws with Theorem \ref{informal2}, we can conclude Theorem \ref{informal1} using some standard resolvent comparison arguments developed in e.g. \cite{EYY,Anisotropic,LY,yang20190}. 


We remark that there exists another method in the literature \cite{ZJ2017,LS2015,LS,ZP2020} to prove the Tracy-Widom law for sample covariance type matrices, that is, a so-called \emph{resolvent flow argument}. While we expect this method to be also applicable to our setting, the techniques seem to be much harder, and we do not pursue this direction in this paper.

\medskip

The rest of this paper is organized as follows. In Section \ref{sec_assumptionandexample}, we give the precise assumptions on the signal matrix $R$ and the variance matrix $S$. We also provide some concrete examples with complicated heterogeneous variance profiles $S$, which have not been studied rigorously in the literature. In Section \ref{sec_mainresults}, we state our main results. The Tracy-Widom distribution for general random Gram matrices is presented in \ref{sec_twdistribution}, while the theoretical properties of the testing statistics (\ref{eq_ona intro}) and (\ref{eq_ona intro1}) are analyzed in Section \ref{sec_statisticsproperty}. In Section \ref{sec_stat}, we conduct numerical simulations to verify the accuracy and power of the proposed statistics for the testing problem (\ref{eq_test1}) under various noise settings that have not been considered in the literature. In Section \ref{sec_proofstrategy}, we sketch the strategy for proving the Tracy-Widom distribution. The technical proofs are put into Appendices \ref{sec_prooftwgram}--\ref{sec_calculationDBM}. 

\section{The model assumptions and examples}\label{sec_assumptionandexample}

In this section, we impose some general assumptions on the signal matrix $R$ and the variance matrix $S.$ We also provide some important examples that have been used in the literature. Note that $YY^\top$ and $Y^\top Y$ have the same non-zero eigenvalues. Hence without loss of generality, we only need to consider the high-dimensional setting where the aspect ratio $c_n:=p/n$ satisfies that
\begin{equation}\label{eq_defnc}
 \tau\le c_n \le 1, 
\end{equation}
for a small constant $\tau>0$.
For the signal matrix $R$, we assume that 
\begin{equation}\label{eq_rankdefinition}
\operatorname{rank}(R)=r,
\end{equation} 
for a fixed $r\in \N$ that is independent of $p$ and $n$. {Note that when $r=0,$ $Y$ is a centered random Gram matrix.} 
Following \cite{alt20172,alt20171}, we impose the following regularity assumptions on the heterogeneous variance profile.

\begin{assumption}\label{assum_gram} Suppose $S$ satisfies the following regularity conditions. 
\begin{itemize}
\item[\bf (A1)] The dimensions of $S$ are comparable, that is, (\ref{eq_defnc}) holds. 

\item[\bf (A2)]  The variances are bounded in the sense that there exist constants $s_* , \e_*>0$ such that
\begin{equation}\label{eq flatS}
\max_{i,j}s_{ij} \leq \frac{s_*}{n},\quad \min_{i,j}s_{ij}\ge \frac{n^{-1/3+\e_*}}{n}.
\end{equation} 
\item[\bf (A3)] The matrices $S$ and $S^\top$ are irreducible in the sense that there exist $L_1, L_2 \in \N$ and a small constant $\tau >0$ such that 
\begin{equation*}
\min_{i,j}[(SS^{\top})^{L_1}]_{ij} \geq \frac{\tau}{n}, \quad \min_{i,j}[(S^{\top}S)^{L_2}]_{ij} \geq \frac{\tau}{n}.
\end{equation*}  

\item[\bf (A4)] The rows and columns of $S$ are sufficiently close to each other in the following sense. There is a continuous 
monotonically decreasing ($n$-independent) function $\Gamma:(0,1] \rightarrow (0,\infty)$  such that $\lim_{\epsilon \downarrow 0} \Gamma(\epsilon)=\infty$, and for all $\epsilon \in (0,1],$ we have 
\begin{equation}\label{assmA4}
\Gamma(\epsilon) \leq \min \left\{ \inf_{1\le i\le p} \frac{1}{p} \sum_{l} \frac{1}{\epsilon+n\norm{S_{i}-S_l}_2^2}, \inf_{1\le j\le n} \frac{1}{n} \sum_{l} \frac{1}{\epsilon+n\norm{(S^{\top})_{j}-(S^{\top})_l}_2^2} \right\},
\end{equation}  
where $S_i$ and $(S^{\top})_{j}$ denote the $i$-th row of $S$ and $j$-th row of $S^{\top}$, respectively.
\end{itemize}
\end{assumption}

\begin{remark}\label{rem_support}
The upper bound in \eqref{eq flatS} is chosen such that the limiting ESD of {$(Y-R)(Y-R)^\top$} is compactly supported as $n\to \infty$. More precisely, we equip $\{1,\cdots, p\}$ and $\{1,\cdots, n\}$ with the $\ell^\infty$-norm and denote the induced operator norms by $\|\cdot \|_{\ell^\infty(n)\to \ell^\infty(p)}$ and $\|\cdot \|_{\ell^\infty(p)\to \ell^\infty(n)}$. Then, Proposition 2.1 of \cite{alt20172} shows that the rightmost edge $\lambda_+$ of the limiting ESD of $(Y-R)(Y-R)^\top$ satisfies 
\begin{equation}\label{bound_center}
\lambda_+\le 4\mathfrak M,\quad \text{with}\quad \mathfrak M:= \max\left\{ \|S\|_{\ell^\infty(n)\to \ell^\infty(p)},\|S^\top\|_{\ell^\infty(p)\to \ell^\infty(n)}\right\}.
\end{equation}
By \eqref{eq flatS}, it is easy to see that $\mathfrak M\le s_*$.
\end{remark}

\begin{remark}
As explained in (2.22) of \cite{Quadratic2015}, assumption (A4) aims to rule out possible spikes from $S$. In \cite[Remark 2.4]{alt20172}, an easier to check sufficient (but not necessary) condition for {\bf (A4)} was also proposed. 
\vspace{5pt}

{\bf \noindent{(A4-s)}:} There are two finite partitions $(I_{\alpha})_{\alpha \in \cal A}$ and $(J_{\beta})_{\beta \in \cal B}$ of $\{1,\cdots, p\}$ and $\{1,\cdots, n\}$, respectively, such that for any $\alpha \in \cal A$  and $\beta \in \cal B,$ we have $|I_\al|\ge \tau p$, $|I_\beta|\ge \tau n$,  and
\begin{equation}\label{assmA4s}
\begin{split}
& \norm{S_{i_1}-S_{i_2}}_2 \leq \tau^{-1}n^{-1}  |i_1-i_2|^{1/2} \    \text{for} \ i_1,i_2 \in I_\alpha, \ \ \text{and}\ \ \norm{(S^{\top})_{j_1}-(S^{\top})_{j_2}}_2 \leq \tau^{-1} n^{-1} |j_1-j_2|^{1/2}  \  \text{for} \ j_1,j_2 \in J_\beta,
 \end{split}
\end{equation}
for a small constant $\tau>0$. The condition (A4) follows easily from (A4-s) using an integral approximation. 
\end{remark}
In addition to \eqref{eq_rankdefinition}, we introduce the following assumption on the signal strengths, i.e. the singular values of $R$. 
\begin{assumption}\label{assum_meanmatrix} We assume that (\ref{eq_rankdefinition}) holds. When $r \geq 1,$ denote by $\sigma_{r} (R)$ the smallest non-trivial singular value of $R$. We assume that 
\begin{equation}\label{eq_supercritical}
\sigma_{r}(R)>(4+\tau)\sqrt{\mathfrak{M}} , 
\end{equation}
for a small constant $\tau>0$, where $\mathfrak M$ is defined in \eqref{bound_center}. 
\end{assumption} 
\begin{remark}\label{rem_why_outlier}
(\ref{eq_supercritical}) is commonly referred to as the supercritical condition, and has appeared in lots of literature in random matrix theory and statistics \cite{BDW,BDMN,RRN14,5447639}. It is a sufficient condition for the mean matrix $R$ to give rise to $r$ outliers of $YY^\top$ that are detached from the bulk spectrum. By Lemma \ref{rigid_lem} below, we have that the largest eigenvalue of $(Y-R)(Y-R)^\top$ is at most $ \lambda_+ +\oo(1)$ with high probability. Combining it with \eqref{eq_supercritical} and applying Weyl's inequality, it is easy to check that $YY^\top$ has $r$ eigenvalues that are larger than $(2+\tau-\oo(1))^2 \mathfrak{M}$. On the other hand, by the Cauchy interlacing, the limiting bulk eigenvalue spectrum of $YY^\top$ is supported on $[0,\lambda_+]$. Hence, under \eqref{eq_supercritical}, there are $r$ outliers that are away from the spectrum edge $\lambda_+$. 
	
{However, we remark that $4\sqrt{\mathfrak M}$ is quite likely not the exact threshold for BBP transition \cite{BBP}. To guarantee the Tracy-Widom law of the edge eigenvalues, it is necessary that all spikes of $R$ are away from (i.e., either above or below) the BBP threshold. If there are critical spikes (i.e., spikes that are exactly equal to the BBP transition threshold), then the Tracy-Widom law of the edge eigenvalues can fail; see Theorem 1.1 in \cite{BBP}. Here we have chosen \eqref{eq_supercritical} simply to ensure that all spikes are supercritical. To determine the exact BBP threshold and to include settings with subcritical spikes, we need to perform a more detailed study of spiked random Gram matrices.} We postpone it to future works, since it is not the focus of the current paper.  
 
%
 \end{remark}

In what follows, we give two concrete examples which satisfy the above assumptions and have not been studied rigorously in the literature. 

\begin{example}[Doubly-heteroscedastic noise, \cite{2019arXiv190209474L}]\label{example_separable}
Consider the following doubly-heteroscedastic noise matrix  
\begin{equation}\label{eqn defnsep}
Y:=A^{1/2}\cal N B^{1/2},
\end{equation}
where $A$ and $B$ are deterministic positive definite symmetric matrices, and $\cal N =(\cal N_{ij})$ is a $p\times n$ random matrix with i.i.d. entries of mean zero and variance $n^{-1}$.
Suppose $A$ and $B$ are diagonal matrices
\begin{equation}\label{eigen}A=\text{diag}(a_1, \ldots, a_p), \quad  B=\text{diag}(b_1, \ldots, b_n),
\end{equation}
with $a_1 \geq a_2 \geq \cdots \geq a_p>0$ and $b_1 \geq b_2 \geq \cdots \geq b_n>0$. 
Then $\mathcal{Q}=YY^\top$ is a random Gram matrix as in Theorem \ref{informal1} with 
variance matrix $S=(({a_ib_j})/n)$ and mean matrix $R=0$. It is easy to see that (A2) and (A3) of Assumption \ref{assum_gram} hold if $a_i$'s  and $b_j$'s are all of order 1. Furthermore, assumption (\ref{assmA4}) is reduced to 
\begin{equation}\label{assmA4 rem} 
\Gamma(\epsilon) \leq \min \left\{ \inf_{1\le i\le p} \frac{1}{p} \sum_{l} \frac{1}{\epsilon+|a_{i}-a_l|^2}, \inf_{1\le j\le n} \frac{1}{n} \sum_{l} \frac{1}{\epsilon+|b_{j}-b_l|^2} \right\},
\end{equation}  
and condition (\ref{assmA4s}) is reduced to 
\begin{equation} \label{assmA4s rem} 
\begin{split}
 |a_{i_1}-a_{i_2}| \leq \tau^{-1}n^{-1/2}  |i_1-i_2|^{1/2}  \quad \text{for} \ i_1,i_2 \in I_\alpha,  \ \ \text{and} \ \
 |b_{j_1}-b_{j_2}|  \leq \tau^{-1}n^{-1/2} |j_1-j_2|^{1/2} \quad \text{for} \ j_1,j_2 \in J_\beta.
\end{split}
\end{equation}
In fact, if we have $a_i=f(i/p)$ and $b_j=g(j/n)$ for some piecewise $1/2$-H{\"o}lder continuous functions $f$ and $g$, then \eqref{assmA4s rem} holds true. One special case is that $f$ and $g$ are piecewise constant functions, which happens when the eigenvalues of $A$ and $B$ take at most $\OO(1)$ many different values. If \eqref{assmA4 rem} or \eqref{assmA4s rem}  holds, as we will see in Section \ref{sec_twdistribution}, Theorem \ref{informal1} applies to \eqref{eqn defnsep} with $r=0$.  

We remark that the diagonal assumption (\ref{eigen}) is not necessary for the Tracy-Widom asymptotics. When the matrices $A$ and $B$ are non-diagonal, we get a model that extends the setting in \eqref{eq_gramgeneralassumption} because the entries of $Y=A^{1/2}\cal N B^{1/2}$ can be correlated.  Finally, we remark that (A4) of Assumption \ref{assum_gram} can be violated by allowing for some large $a_i$'s and $b_j$'s. Then we get a spiked separable covariance matrix, which has been studied in detail in \cite{dingyang2}.  Our Theorem \ref{informal1} also applies to this case.
\end{example}

\begin{example}[Sparse noise, \cite{hwang2019,lounici2014}]\label{exam_sparserandomgram} In this example, we consider the sparse noise matrix $Z$ as proposed in \cite{hwang2019}. The sparse random Gram matrices can be used as a natural model to study high-dimensional data with randomly missing observations. 
For instance, given a probability $\mathsf{p},$ we set $z_{ij}=h_{ij}w_{ij},$ where $w_{ij}$ are random variables independent of $\{h_{ij}\}$, and $h_{ij}$ are i.i.d. (rescaled) Bernoulli random variables with $\mathbb{P}(h_{ij}=(n \mathsf{p})^{-1/2})=\mathsf p$ and $\mathbb{P}(h_{ij}=0)=1-\mathsf p$. More generally, 
we say that $\mathcal{Q}=YY^\top$ is a sparse random Gram matrix if $Y$ satisfies the following properties: the entries $y_{ij}$, $1 \leq i \leq p, 1 \leq j \leq n$, are independent random variables satisfying 
\begin{equation}\label{eq_sy}
\mathbb{E}\left|\frac{y_{ij}-\E y_{ij}}{\sqrt{ns_{ij}}}\right|^k \leq  \frac{(Ck)^{Ck}}{nq^{k-2}} , \quad k \geq 3,
\end{equation} 
for a large constant $C>0$ and sparsity parameter $q$ with $1\ll  q \leq \sqrt{n}$. In the above setting with randomly missing observations, we have that $q=\sqrt{n\mathsf{p}}$.

\end{example}

\section{Main results}\label{sec_mainresults}

In this section, we state the main results of this paper. The Tracy-Widom distribution of the edge eigenvalues for a general class of random Gram matrices, i.e., the formal statement of Theorem \ref{informal1}, will be presented in Section \ref{sec_twdistribution}. The theoretical properties of the test statistics (\ref{eq_ona intro})--(\ref{eq_ona intro1}) and the associated sequential estimator (\ref{eq_seqestimate}) will be given in Section \ref{sec_statisticsproperty}.

\subsection{Tracy-Widom distribution for random Gram matrices}\label{sec_twdistribution}

In this subsection, we provide the formal statement for Theorem \ref{informal1}. Before stating our main result, we first introduce the necessary notations. 
If \eqref{eq flatS} holds, then there exists a unique vector of holomorphic functions 
$$\mb(z)=(m_1(z),  \cdots, m_p(z)): \mathbb{C}_+  \rightarrow \mathbb{C}^p,\quad \mathbb{C}_+:=\{z\in \C: \im z>0\},$$ 
satisfying the so-called \emph{vector Dyson equation} 
\begin{equation}\label{selfm1}
  \frac{1}{\mb}=-z \mathbf{1}+S\frac{1}{\mathbf{1}+S^\top \mb},
\end{equation}
such that $\im m_k (z)>0,$ $k=1,\cdots, p,$  for any $z \in \mathbb{C}_+$ \cite{alt20172,alt20171,hachem2007}. In the above equation, $\mathbf 1$ denotes the vector whose entries are all equal to $1$, and both ${1}/{\mb}$ and ${1}/({\mathbf{1}+S^\top \mb})$ mean the entrywise reciprocals. Moreover, for each $k=1,\cdots,p,$ there exists a unique probability measure $\nu_k$ that has support contained in $[0, 4\mathfrak M]$ and is absolutely continuous with respect to the Lebesgue measure, such that $m_k$ is the Stieltjes transform of $\nu_k$:
\begin{equation*}
m_k(z)=\int \frac{\nu_k(\dd x)}{x-z} , \quad z \in \mathbb{C}_+. 
\end{equation*}   
(If we consider the case $p>n$, then $\nu_k$ will also have a point mass at zero, but we do not have to worry about this issue under \eqref{eq_defnc}.) Let  $\rho_k$ be the density function associated with $\nu_k$. Then the asymptotic ESD of $(Y-R)(Y-R)^\top$ is given by $\nu:= p^{-1}\sum_k \nu_k$, with the following density $\rho$ and Stieltjes transform $m$,
\begin{equation}\label{defnmz}
\rho:=\frac1p\sum_{k}\rho_k,\quad m(z):=\frac{1}{p} \sum_k m_k(z).
\end{equation} 
We summarize the basic properties of the density functions $\rho$ and $\rho_k$, $1 \leq k \leq p$. 
 \begin{lemma}[Theorem 2.3 of \cite{alt20172}] \label{lem_edgebehavior}
 Under Assumption \ref{assum_gram}, 
 for any $1 \leq k \leq p,$ there exists a sequence of positive numbers  $a_1 > a_2 >\cdots > a_{2\mathfrak q} \ge 0$ such that 
 \begin{equation*}
 \supp  \rho_k =\supp \rho  =\bigcup_{i=1}^{\mathfrak{q}} [a_{2i}, a_{2i-1}],
 \end{equation*}
 where $\mathfrak{q} \in \mathbb{N}$ depends only on $S$. Moreover, $\rho$ has the following square root behavior near  $a_{1}$:
 \begin{equation}\label{sqrtrho}
 \rho(a_1-x)={\pi}^{-1} {\varpi}\sqrt{x}+\OO(x), \quad x \downarrow 0,
\end{equation}  
{where $\varpi\equiv \varpi(S)$ is an order 1 positive value determined by $S$. }  
 \end{lemma}
 

In what follows, we shall call $a_{k}$ the spectral edges. In particular, we will focus on the right-most edge $a_1$ and denote it by  $\lambda_+\equiv a_1$ following the convention in the random matrix theory literature. {We remark that as discussed in \cite{alt20172}, it is possible that the density $\rho$ has some cusp singularities when two edges are close to each other or when $\rho$ touches zero. In the current paper, since we are mainly interested in the edge eigenvalue statistics around $a_1$, we only need assumptions to ensure \eqref{sqrtrho}. However, to show the Tracy-Widom law at other edges, we need extra edge regularity and edge separation conditions to avoid cusp singularities as in \cite{ZJ2017,Anisotropic}. We will pursue this direction in future works.} Now, we are ready to state 
the Tracy-Widom law of the largest edge eigenvalues for a general class of random Gram matrices with variance  and mean matrices satisfying Assumptions \ref{assum_gram} and \ref{assum_meanmatrix}.

\begin{theorem}\label{thm_twgram} 
Let $Y=(y_{ij})$ be a $p\times n$ random matrix such that $\wt y_{ij}:= (y_{ij}-r_{ij})/\sqrt{s_{ij}}$ are real i.i.d. random variables. Suppose $\wt y_{11}$ follows a probability distribution that does not depend on $n$, and satisfies $\E\wt y_{11}=0$, $\E\wt y^2_{11}=1$ and
\begin{equation}\label{sharp_moment}
\lim_{x \rightarrow \infty} x^4 \mathbb{P}\left(\left|\wt y_{11}\right| \geq x\right)=0.
\end{equation}
Suppose the variance matrix $S=(s_{ij})$ satisfies Assumption \ref{assum_gram} and the mean matrix $R=(r_{ij})$ satisfies Assumption \ref{assum_meanmatrix}.  Denote the eigenvalues of $\mathcal{Q}=YY^\top$ by $\lambda_1\ge \lambda_2 \ge \cdots \ge \lambda_p$.
Then we have that 
\begin{equation}\label{eq_twgram}
\lim_{n \rightarrow \infty} \mathbb{P}\left(\varpi^{2/3} p^{2/3} (\lambda_{r+1}-\lambda_+) \leq x\right)=F_1(x), \quad \text{for all} \ \ x \in \mathbb{R},
\end{equation}
where $\varpi$ is the value defined in (\ref{sqrtrho}), and $ F_1$ is the type-1 Tracy-Widom cumulative distribution function. More generally, for any fixed $k\in \N$, we have that {
\begin{equation}\label{SUFFICIENT2}
\begin{split}
   \lim_{p \to \infty}\mathbb{P}\left[ \left(\varpi^{2/3}p^{{2}/{3}}(\lambda_{i+r} - \lambda_+) \leq x_i\right)_{1\le i \le k} \right] 
=   \lim_{p \to \infty} \mathbb{P}\left[\left(p^{{2}/{3}}(\mu_i^{\rm GOE} - 2) \leq x_i\right)_{1\le i \le k} \right] , 
\end{split}
\end{equation}}
for all $(x_1 , x_2, \ldots, x_k) \in \mathbb R^k$, where we recall that $\mu_i^{\rm GOE}$ are the eigenvalues of GOE as given by \eqref{eq GOE}.

Furthermore, the condition \eqref{sharp_moment} is necessary in the following sense: if $r=0$ and \eqref{sharp_moment} does not hold, then we have that for any fixed $x>\lambda_+$,
\begin{equation}\label{eq_twgramnece}
\limsup_{n \rightarrow \infty} \mathbb{P}( \lambda_{1} > x)>0.
\end{equation}
Hence $\varpi^{2/3}p^{{2}/{3}}(\lambda_{1} - \lambda_+)$ does not converge to $F_1$ in distribution. 
\end{theorem}

For the reader's convenience, we state the Tracy-Widom distributions for the models in Examples \ref{example_separable} and \ref{exam_sparserandomgram} as corollaries of Theorem \ref{thm_twgram}.  

\begin{corollary} \label{cor_separable}
Assume that \eqref{eq_defnc} holds. Consider the doubly-heteroscedastic matrix in \eqref{eqn defnsep}, where $\cal N$ is a $p\times n$ random matrix with \smash{$\cal N_{ij}= n^{-1/2}\wt y_{ij}$} for a sequence of i.i.d. random variables $\wt y_{ij}$. Suppose $\wt y_{11}$ follows a probability distribution that does not depend on $n$, and satisfies $\E \wt y_{11}=0$, $\E \wt y^2_{11}=1$ and \eqref{sharp_moment}. In addition, assume that 
\begin{equation} \label{assm_3rdmoment}
	\mathbb{E}\left(\wt y_{11}^3\right)=0.
\end{equation}  
Let $A$ and $B$ be $p\times p$ and $n\times n$ deterministic positive definite symmetric matrices, whose eigenvalues satisfy that 
\begin{equation}\label{eq_seperablecheckable}
 {\tau} \leq a_p \le a_1 \leq  {\tau^{-1}},\quad {\tau} \leq b_n\le  b_1 \leq  {\tau^{-1}},   
\end{equation}
for a small constant $\tau>0$, and satisfy the condition \eqref{assmA4 rem} for a continuous monotonically decreasing function $\Gamma:(0,1] \rightarrow (0,\infty)$ such that $\lim_{\epsilon \downarrow 0} \Gamma(\epsilon)=\infty$. 
Then, for any fixed $k\in \N$, we have that
\begin{equation}\nonumber 
\begin{split}
 & \lim_{n\to \infty}\mathbb{P}\left[ \left(\varpi^{2/3}p^{{2}/{3}}(\lambda_{i} - \lambda_+) \leq x_i\right)_{1\le i \le k} \right]= \lim_{n\to \infty} \mathbb{P}\left[\left(p^{{2}/{3}}(\mu_i^{\rm GOE} -2) \leq x_i\right)_{1\le i \le k} \right] , 
\end{split}
\end{equation}
for all $(x_1 , x_2, \ldots, x_k) \in \mathbb R^k$, where $\lambda_+$ and $\varpi$ are defined for the variance matrix $S=(({a_ib_j})/n)$. Finally, the condition \eqref{assm_3rdmoment} is not necessary if either $A$ or $B$ is diagonal.
\end{corollary}

\begin{corollary}\label{thm_twgraph_sparse}
Suppose $\mathcal{Q}=YY^\top $ is a sparse random Gram matrix, where the entries of $Y$ satisfy (\ref{eq_sy}) with $q\ge n^{1/3+c_\phi}$ for a small constant $c_\phi>0$. Suppose the variance matrix $S=(s_{ij})$ satisfies Assumption \ref{assum_gram} and the mean matrix $R=(r_{ij})$ satisfies Assumption \ref{assum_meanmatrix}. 
Then for any fixed $k\in \N$, we have that
\begin{equation}\nonumber 
\begin{split}
 & \lim_{n\to \infty}\mathbb{P}\left[ \left(\varpi^{2/3}p^{{2}/{3}}(\lambda_{i+r} - \lambda_+) \leq x_i\right)_{1\le i \le k} \right]= \lim_{n\to \infty} \mathbb{P}\left[\left(p^{{2}/{3}}(\mu_i^{\rm GOE} - 2) \leq x_i\right)_{1\le i \le k} \right] , 
\end{split}
\end{equation}
for all $(x_1 , x_2, \ldots, x_k) \in \mathbb R^k$. 
\end{corollary}

The proofs of Theorem \ref{thm_twgram}, Corollary \ref{cor_separable} and Corollary \ref{thm_twgraph_sparse} will be given in Appendix \ref{sec_prooftwgram}. We remark that the settings of Corollaries \ref{cor_separable} and \ref{thm_twgraph_sparse} are actually beyond the one in Theorem \ref{thm_twgram}: in Corollary \ref{cor_separable}, the entries of $Y$ can be correlated because we did not assume that $A$ and $B$ are diagonal, while in Corollary \ref{thm_twgraph_sparse}, the distribution of $\wt y_{11}= (y_{11}-r_{11})/\sqrt{s_{11}}$ may depend on $n$ under the condition (\ref{eq_sy}). Hence, they are not trivial corollaries of Theorem \ref{thm_twgram}. But in the proof, we can reduce their settings to ones that are compatible with Theorem \ref{thm_twgram}. {For example, for doubly-heteroscedastic matrices, under the setting of Corollary \ref{cor_separable},  \cite{yang20190} has proved the edge universality---the limiting distribution of the edge eigenvalues is the same as that in the Gaussian case with i.i.d. Gaussian $\wt y_{ij}$. On the other hand, by the rotational invariance of Gaussian $\cal N$, we can reduce the model to one with diagonal $A$ and $B$ so that Theorem \ref{thm_twgram} applies. Combining these two results finishes the proof of Corollary \ref{cor_separable}.} 

{We also mention that the condition (\ref{assm_3rdmoment}) in Corollary \ref{cor_separable} and the condition $q \geq n^{1/3+c_{\phi}}$ in Corollary \ref{thm_twgraph_sparse} are mainly technical. The edge universality in \cite{yang20190} was proved under the vanishing third moment condition. Hence, we have kept (\ref{assm_3rdmoment}) in Corollary \ref{cor_separable}, but we believe it can be removed with further theoretical development. We also believe that $q \geq n^{1/3+c_{\phi}}$ can be weakened to $q \geq n^{1/6+c_{\phi}}$, while Corollary \ref{thm_twgraph_sparse} may fail when $q \le n^{1/6}$. Since these problems are not the main focus of this paper, we will pursue them in future works. We also refer the readers to Remark \ref{rmk_techinical} for more details. }  

\subsection{Theoretical properties of the test statistics}\label{sec_statisticsproperty}

With Theorem \ref{thm_twgram}, we can readily obtain the asymptotic distributions of the statistics {$ \mathbb{T}(r_0)$} in (\ref{eq_ona intro}) and  $\mathbb{T}_{r_0}$ in (\ref{eq_ona intro1}) under the null hypothesis in (\ref{eq_test1}), and analyze the statistical power of them under the alternatives. Corresponding to {$\mathbb{T}(r_0)$} and  $\mathbb{T}_{r_0}$, we define the following two sequential testing estimators  
\begin{equation}\label{eq_defn_consistentesmtiator}
 \wh r_1:=\inf\{r_0\ge 0:   \mathbb{T}(r_0)  < \delta_n^{(1)} \},\quad \wh r_2:=\inf\{r_0\ge 0: \mathbb T_{r_0} < \delta_n^{(2)} \}.
\end{equation}  
We will show that $\widehat{r}_1$ and $\wh r_2$ are consistent estimators of $r$ as long as we choose the critical values $\delta_n^{(1)}$ and $\delta_n^{(2)}$ properly. Let $\cal W \sim W_p(I_p ,n)$ be a standard Wishart matrix.
We define the following statistics $\mathbb{G}_1$ and $\mathbb{G}_2$ in terms of the eigenvalues of $\cal W$,
\begin{equation*}
\mathbb{G}_1:=\max_{1 \leq i \le r_*-r_0} \frac{\lambda_i(\cal W)-\lambda_{i+1}(\cal W)}{\lambda_{i+1}(\cal W)-\lambda_{i+2}(\cal W)}, \quad 
\mathbb{G}_2:= \frac{\lambda_1(\cal W)-\lambda_2(\cal W)}{\lambda_{r_* - r_0 +1}(\cal W)-\lambda_{r_* - r_0 +2}(\cal W)}. 
\end{equation*}   

\begin{corollary}\label{coro_onatskiresult}
Suppose the assumptions of Theorem \ref{thm_twgram} hold and $r_* > r$. Under the null hypothesis $\mathbf{H}_0$ in (\ref{eq_test1}), we have that  
\begin{equation}
\lim_{n \rightarrow \infty} \mathbb{P}(\mathbb{T} \leq x)=\lim_{n \rightarrow \infty} \mathbb{P} (\mathbb{G}_1 \leq x),
\quad \text{and}\quad
\lim_{n \rightarrow \infty} \mathbb{P}(\mathbb{T}_{r_0} \leq x)=\lim_{n \rightarrow \infty} \mathbb{P} (\mathbb{G}_2 \leq x), \label{first_ona}
\end{equation} 
for all $x \in \mathbb{R}$. On the other hand, 
if $\delta_n^{(1)} p^{-2/3}\to 0$, then 
\begin{equation}\label{second_ona}
\lim_{n \rightarrow \infty} \mathbb{P} (\mathbb{T}>\delta_n^{(1)})=1,\quad \text{under \ \ \ $\mathbf{H}_a$};
\end{equation} 
if $\delta_n^{(2)} p^{-2/3}/\left(\lambda_{r_0+1}-\lambda_{r_0+2}\right)\to 0$, then 
\begin{equation}\label{third_ona}
\lim_{n \rightarrow \infty} \mathbb{P} (\mathbb{T}_{r_0}>\delta_n^{(2)})=1,\quad \text{under \ \ \ $\mathbf{H}_a$.}
\end{equation}
Consequently, if $\delta_n^{(1)}\to \infty$ and $\delta_n^{(1)} p^{-2/3}\to 0$, then 
\begin{equation}\label{consis1}
\lim_{n \rightarrow \infty} \mathbb{P}(\widehat{r}_1=r)=1;
\end{equation}
if $\delta_n^{(2)}\to \infty$ and $\delta_n^{(2)} p^{-2/3}/\left(\lambda_{r_0+1}-\lambda_{r_0+2}\right)\to 0$, then 
\begin{equation}\label{consis2}
\lim_{n \rightarrow \infty} \mathbb{P}(\widehat{r}_2=r)=1.
\end{equation}
\end{corollary}
\begin{proof}
\eqref{first_ona} follows directly from \eqref{SUFFICIENT2}. 
On the other hand, under $\mathbf{H}_a$ and the assumption $r_*>r$, we have that
\begin{equation*}
\mathbb{T} \geq \frac{\lambda_{r}-\lambda_{r+1}}{\lambda_{r+1}-\lambda_{r+2}}. 
\end{equation*}
By Theorem \ref{thm_twgram}, we have that
\begin{equation}
\lambda_{r+1}-\lambda_+=\OO(p^{-2/3}),\quad \lambda_{r+1}-\lambda_{r+2}=\OO(p^{-2/3}),\quad \lambda_{r_*+1}-\lambda_{r_*+2} =\OO(p^{-2/3}), \label{eq_smallgap}
\end{equation} 
with probability $1-\oo(1)$. Furthermore, under Assumption \ref{assum_meanmatrix}, as discussed in Remark \ref{rem_why_outlier} we have that
 $ |\lambda_r - \lambda_+| \geq c_\tau$ for a small constant $c_\tau>0$. Hence we get that with probability $1-\oo(1)$,
\begin{equation*}
\mathbb{T} \geq \frac{\lambda_{r}-\lambda_{r+1}}{\lambda_{r+1}-\lambda_{r+2}}\gtrsim p^{2/3}, 
\end{equation*}
which concludes \eqref{second_ona} and \eqref{consis1}. 
Finally, using \eqref{eq_smallgap}, we immediately conclude \eqref{third_ona} and \eqref{consis2}. 
\end{proof}

\begin{remark}\label{rem_degen}
{ We make a few remarks here. First, the conditions $\delta_n^{(1)} \rightarrow \infty$ and $\delta_n^{(2)} \rightarrow \infty$ are necessary and sufficient to guarantee that $\mathbb{T}$ and $\mathbb{T}_{r_0}$ have asymptotic zero type I errors. For any fixed $r_*-r_0$, the joint distribution of $\{ \lambda_i(\mathcal{W})\}_{1 \leq i \le r_*-r_0+2}$ can be expressed in terms of the Airy kernel \cite{Forr}. 
Although it is hard to get explicit expressions of the limiting distributions of $\mathbb{G}_1$ and $\mathbb{G}_2$, it is easy to check that 
both the distributions are supported on the whole positive real line. Consequently, it is necessary to let $\delta_n^{(1)}$ and $\delta_n^{(2)}$ diverge.}
Second, in order to choose a non-trivial $\delta^{(2)}_n$ satisfying $\delta_n^{(2)}\to \infty$ and $\delta_n^{(2)} p^{-2/3}/\left(\lambda_{r_0+1}-\lambda_{r_0+2}\right)\to 0$, we need the following estimate:	
\begin{equation}\label{non-degenerate eq} p^{2/3}\left(\lambda_{r_0+1}-\lambda_{r_0+2}\right)\to \infty  \quad \text{in probability}.\end{equation} 
The condition \eqref{non-degenerate eq} can be guaranteed if $\mathbf H_a$ holds and the $(r_0+1)$-th and $(r_0+2)$-th singular values of $R$ are non-degenerate.  However, we believe that even in the degenerate case, the condition \eqref{non-degenerate eq} still holds. In fact, following \cite{BBP, principal}, we conjecture that the degenerate $(r_0+1)$-th and $(r_0+2)$-th spikes of $R$ will give rise to outliers satisfying that $\lambda_{r_0+1} - \lambda_{r_0+2}\gtrsim p^{-1/2}$ with probability $1-\oo(1)$. To prove this fact, we need to establish the limiting distributions of the outliers of spiked random Gram matrices, and we postpone the study to a future work. 
\end{remark} 

In Table \ref{table_criticalvalue}, we report some simulated finite sample critical values of $\mathbb{G}_1$ and $\mathbb{G}_2$ corresponding to type I error rate $\alpha=0.1$ for different choices of $r_*-r_0$, $n\in \{200,\, 500\}$ and $c_n=p/n\in \{0.5,\, 1,\, 2\}$ based on $5,000$ Monte Carlo simulations. 
All the simulations in Section \ref{sec_stat} will be based on these critical values.

\begin{table}[H]
\def\arraystretch{1.3}
\begin{center}
\begin{tabular}{cccccccc}

$r_*-r_0/(p,n)$ & $(100,200)$  &  $(250,500)$  & $(200,200)$ & $(500,500)$ & $(400, 200)$ & $(1000, 500)$    \\ \hline
1 & 4.77 & 4.68 & 4.71 & 4.53 & 4.51 & 4.51\\ 
2 & 5.68 (4.98) & 5.6 (4.86) & 5.68 (5.02) & 5.62 (4.96) & 5.59 (4.95) & 5.62 (4.87)\\  
3 & 6.37 (5.15) & 6.42 (4.95) & 6.51 (5.23) & 6.41 (5.48) & 6.63 (5.23) & 6.38 (5.19)\\  
4 & 6.94 (5.41) & 7.12 (5.28) & 6.98 (5.63) & 6.96 (5.52) & 7.07 (5.34) & 7.93 (5.48)\\  
5 & 7.86 (5.94) & 8.12 (5.87) & 8.23 (6.03) & 7.89 (5.94) & 7.91 (5.82) & 7.78 (5.79)\\  \hline
\end{tabular}
\caption{Critical values for $\mathbb{G}_1$ and $\mathbb{G}_2$ (inside the parentheses) for different combinations of $p$, $n$ and $r_*-r_0$ under the nominal significance level $0.1.$ 
When $r_*-r_0=1,$ we have $\mathbb{G}_1=\mathbb{G}_2$, so they share the same critical values. Note that $\mathbb{G}_2$ always has smaller critical values than $\mathbb G_1$. 
}\label{table_criticalvalue}
\end{center}
\end{table}


\section{Numerical simulations}\label{sec_stat}

In this section, we design Monte-Carlo simulations to demonstrate the accuracy and power of our proposed statistics for the hypothesis testing problem (\ref{eq_test1}) under some general noise structures. By Corollary \ref{coro_onatskiresult}, we will use the statistics $\mathbb{T}$ and $\mathbb{T}_{r_0}$ and reject the null hypothesis \smash{$\mathbf{H}_0$} of (\ref{eq_test1}) if they are larger than the critical values in Table \ref{table_criticalvalue}. 
%
For the simulations, we always consider the following scenario: $R$ is of rank $r\le 5$, and all the singular values of $R$ are non-degenerate. 
%
In the above scenario, we consider the following three noise structures, whose impact on the signal detection is still unknown rigorously in the literature.
\begin{enumerate}
\item[(I)] $Z$ is a doubly-heteroscedastic noise matrix. Specifically, we take {$Z= A^{1/2} \cal N B^{1/2},$ where $\cal N$ is a $p \times n$ white noise matrix with i.i.d. entries of mean zero and variance $n^{-1}$,} {and $A$ and $B$ are two positive definite matrices generated as follows: $A$ and $B$ have spectral decompositions $A=U_A \Sigma_A U_A^\top$ and $B=U_B \Sigma_B U_B^\top,$ where 
\begin{equation*}
\Sigma_A=\operatorname{diag}(\underbrace{1,\cdots,1}_{p/2 \ \text{times}},\  \underbrace{2, \cdots,2}_{p/2 \ \text{times}} ),\quad 
\Sigma_B=\operatorname{diag}(\underbrace{3,\cdots,3}_{p/4 \ \text{times}},\  \underbrace{4, \cdots,4}_{p/4 \ \text{times}}, \underbrace{5, \cdots,5}_{p/2 \ \text{times}} ),
\end{equation*}
and $U_A$ and $U_B$ are two orthogonal matrices generated from the \texttt{R} package \texttt{pracma}. 
}
\item[(II)] 
$Z$ is a sparse noise matrix. Specifically, we take $z_{ij}=h_{ij}w_{ij}$, where $h_{ij}$ are i.i.d. (rescaled) Bernoulli random variables 
satisfying $\PP(h_{ij}=(n \mathsf {p})^{-1/2})=\mathsf  p$ and $\PP(h_{ij}=0)=1-\mathsf  p$, and $w_{ij}$ are independent $\mathcal{N}(0,s_{ij})$ random variables. In the simulations, we take $\mathsf {p}=n^{-1/4}$ and $s_{ij}=\alpha_i \beta_j$ with $\alpha_i$ being i.i.d. random variables uniformly distributed on $[1,2]$ and  $\beta_j$ being i.i.d. random variables uniformly distributed on $[3,4]$. 

\item[(III)] $Z=(z_{ij})$ is a noise matrix whose variance matrix $S$ has a banded latent structure. Specifically, we assume that $z_{ij} \sim \mathcal{N}(0, s_{ij})$ with 
\begin{equation*}
s_{ij}= \left( 1+\nu_{ij} \mathbf{1}_{|i-j|\leq 5} \right)/n,
\end{equation*}   
where 
$\nu_{ij}$ are i.i.d. random variables uniformly distributed on $[1,2].$ 

\end{enumerate}
In the simulations, we always take $r_*=5$ and $c_n\in \{0.5, \, 1,\,  2\}$. 


First, under the null hypothesis $\mathbf{H}_0$ in \eqref{eq_test1}, we check the accuracy of the statistics under the nominal significance level $0.1.$ 
We consider the above settings (I)--(III) under the null hypothesis $r_0=3$, with signal matrix $R=18\mathbf{e}_{1p} \mathbf{e}_{1n}^\top+16 \mathbf{e}_{2p} \mathbf{e}_{2n}^\top+14 \mathbf{e}_{3p} \mathbf{e}_{3n}^\top$.
Here, $\mathbf{e}_{ip}$ and $\mathbf{e}_{in}$ denote the unit vectors along the $i$-th coordinate axis in $\R^p$ and $\R^n$, respectively. In Figure \ref{fig_test1typeis1}, we report the simulated type I error rates for both the statistics (\ref{eq_ona intro}) and (\ref{eq_ona intro1}) in the settings (I)--(III) for the noise matrices. We find that both statistics combined with the critical values in Table \ref{table_criticalvalue} can attain reasonable accuracy even when $n=200$. 




\begin{figure}[!ht]
\subfigure[Accuracy of  $\mathbb{T}$ in (\ref{eq_ona intro}).]{\label{fig:a}\includegraphics[width=8cm,height=5cm]{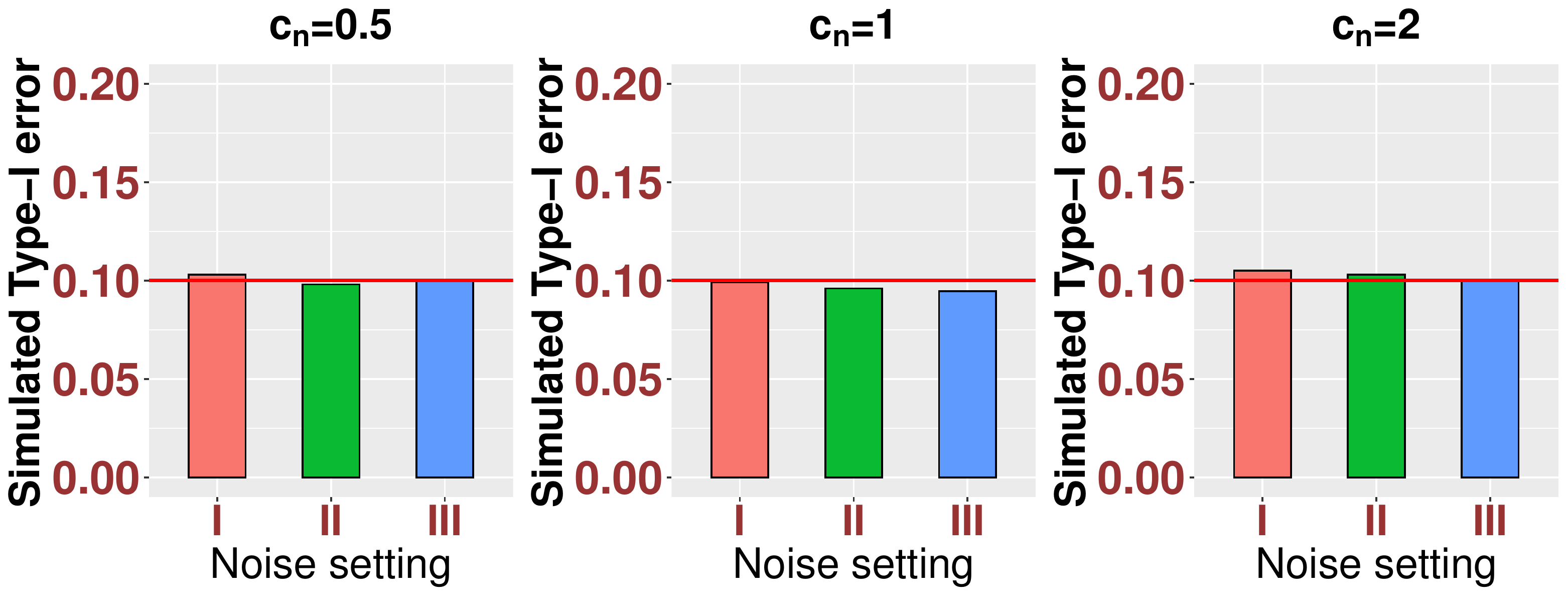}}
\hspace*{2cm}
\subfigure[Accuracy of  $\mathbb{T}_{r_0}$ in (\ref{eq_ona intro1}).]{\label{fig:b}\includegraphics[width=8cm,height=5cm]{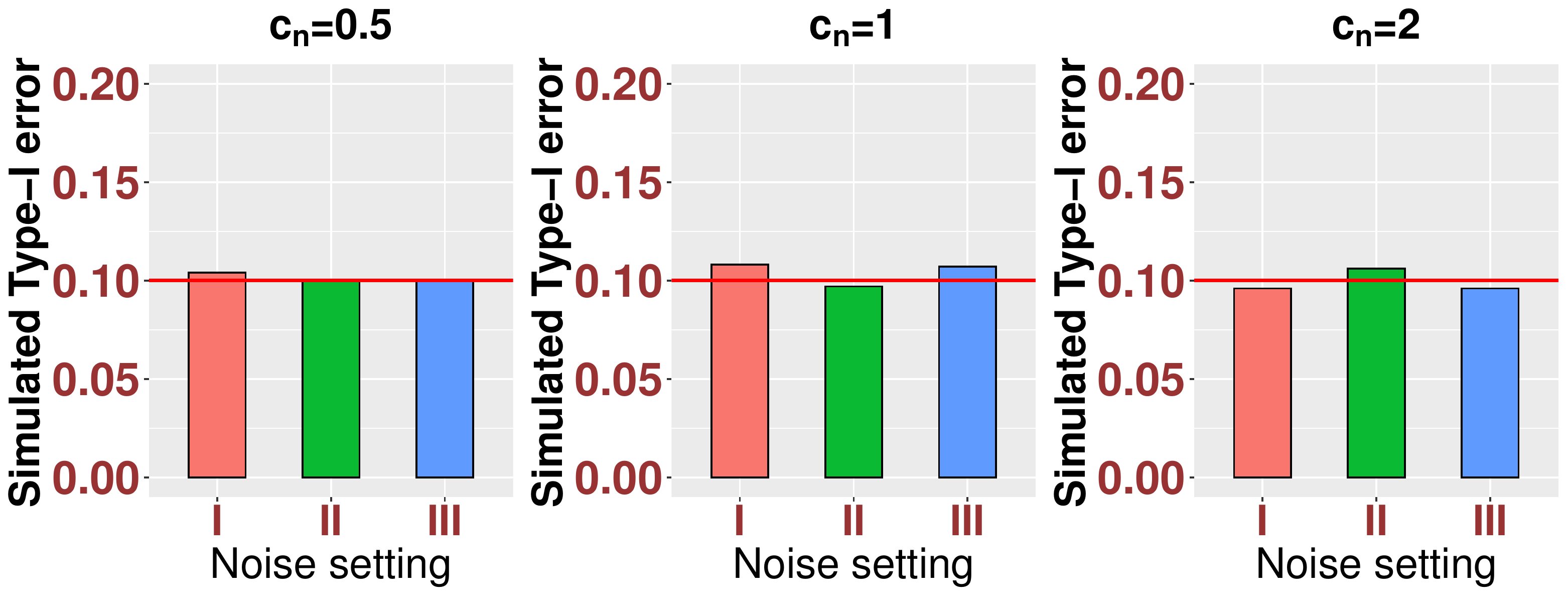}}
\caption{Simulated type I error rates under the nominal level 0.1 for $\mathbb{T}$ and $\mathbb{T}_{r_0}$. We take $n=200$ and report the results based on 2,000 Monte-Carlo simulations and the critical values from Table \ref{table_criticalvalue}.}
\label{fig_test1typeis1}
\end{figure}  

Second, we examine the power of the statistics under the nominal level $0.1$ when $r_0=0$ in (\ref{eq_test1}).  We set the alternative as 
\begin{equation}\label{eq_alternativet1}
\mathbf{H}_a: \ R=d \mathbf{e}_{1p} \mathbf{e}_{1n}^\top,\quad { \text{for some fixed value $d> 0$}}.
\end{equation}  
In Figure \ref{fig_test1powers1}, we report the simulated power for both the statistics (\ref{eq_ona intro}) and (\ref{eq_ona intro1}) as $d$ increases, where we take $c_n=2$ and the settings (I)--(III) for the noise matrices. We see that both statistics have high power even for a not so large $n$, $n=200$, as long as $d$ is above some threshold. Furthermore, 
when $d$ is in a certain range, we find that the statistic  $\mathbb{T}_{r_0}$ in (\ref{eq_ona intro1}) has better performance in terms of power than the statistic $\mathbb{T}$ in (\ref{eq_ona intro}). {Finally, the statistic $\mathbb{T}_{r_0}$ starts to have non-zero power for smaller values of $d$ compared to $\mathbb{T}.$ This enables us to study a wider range of alternatives in terms of the $d$ value.} We expect that this is due to the fact that the statistic $\mathbb{T}$  needs a larger critical value to reject \smash{$\mathbf{H}_0$} as illustrated in Table \ref{table_criticalvalue}.



\begin{figure}[!ht]
\centering
\includegraphics[width=15cm,height=6cm]{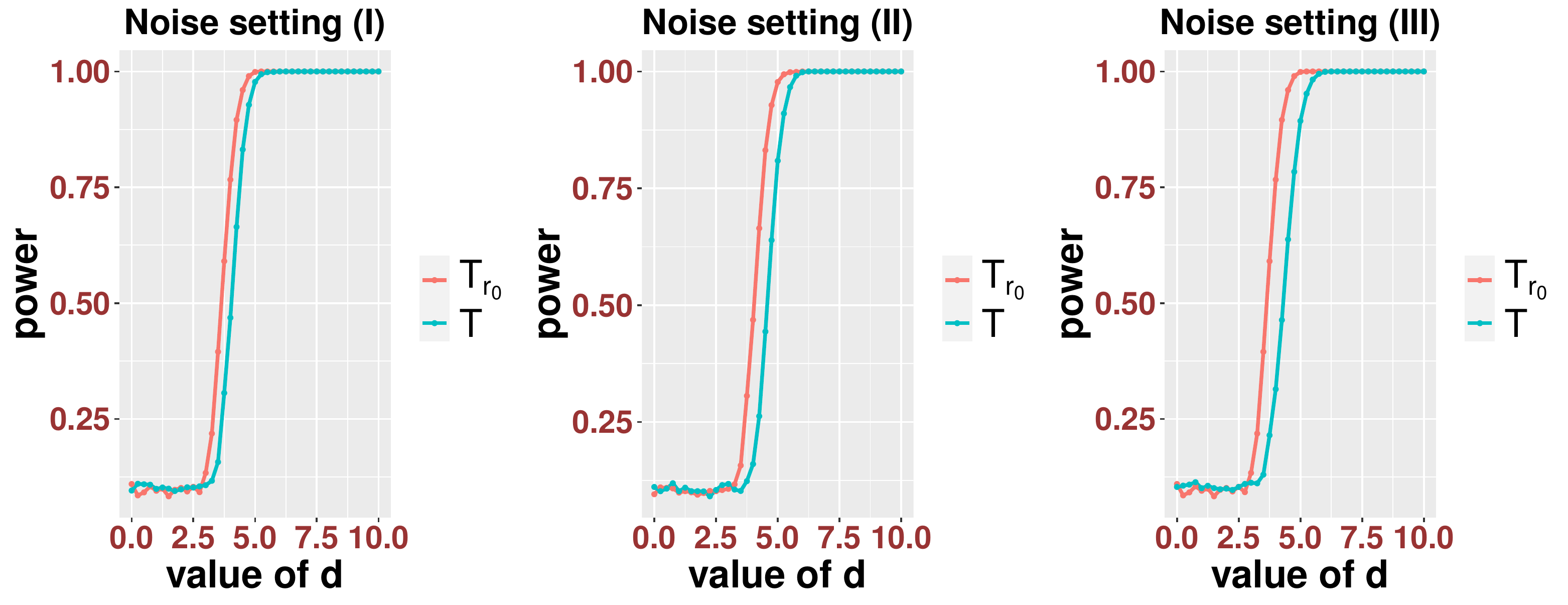}
\caption{Simulated power of the statistics $\mathbb{T}$ and $\mathbb{T}_{r_0}$ for the alternative (\ref{eq_alternativet1}) under the nominal level 0.1. We take $n=200$ and $c_n=2$. We report the results based on 2,000 Monte-Carlo simulations and the critical values from Table \ref{table_criticalvalue}. }
\label{fig_test1powers1}
\end{figure}

\section{Proof strategies}\label{sec_proofstrategy}

In this section, we describe the main strategy for the proof of Theorem \ref{thm_twgram}. All the technical details can be found in the appendix. 
From the theoretical point of view, our proof of Theorem \ref{thm_twgram} employs the following three step strategy. 

\vspace{5pt}

\noindent{\bf Step 1:} Proving a local law on the Stieltjes transform of the random Gram matrix $\cal Q$, $m_{\cal Q}(z):=p^{-1}\tr(\cal Q-z)^{-1}$. This is needed in order to check the square root behavior of the ESD of $\cal Q$ around the right edge.  

\vspace{5pt}

\noindent{\bf Step 2:} Establishing the asymptotic Tracy-Widom law for the edge eigenvalues of the Gaussian divisible random Gram matrix $\cal Q_t$ in Theorem \ref{informal2} for a small $t>0$. 

\vspace{5pt}

\noindent{\bf Step 3:} Showing that $\cal Q$ has the same edge eigenvalue statistics as $\cal Q_t$ asymptotically. 

\vspace{5pt}

\noindent This three step strategy has been widely used in the proof of bulk universality of random matrices \cite{EPRSY2010,ESY2011,ESYY2012,Bulk_univ}. For a more extensive review, we refer the reader to \cite{erdos2017dynamical} and references therein. However, it has been rarely (if any) used in the study of the edge eigenvalues of random Gram matrices. One of the main reasons is that the above Step 2 for Gram type random matrices---the core of the strategy---was not well-understood previously. 

Regarding the proof of Theorem \ref{thm_twgram}, even though the results of Step 1 have been established in \cite{alt20172,alt20171}, Steps 2 and 3 are still missing. For Step 3, we can employ some standard resolvent comparison arguments developed in e.g. \cite{bao2015,DY,EYY,Anisotropic,LY,pillai2014,yang20190}. 
In this paper, we mainly focus on Step 2, which is completed by Theorem \ref{informal2}. We will provide the formal statement of { Theorem \ref{informal2}} in Theorem \ref{thm_regularbm}. For this purpose, we first need to introduce some new notations.

Let $Y$ be a $p\times n$ data matrix, and $X$ be an independent $p\times n$ random matrix whose entries are i.i.d. centered Gaussian random variables with variance $n^{-1}$. Since the multivariate Gaussian distribution is rotationally invariant under orthogonal transforms, for any $t>0$ we have that 
$$Y+\sqrt{t}X\stackrel{d}{=}U_1\left(W+\sqrt{t}X\right)U_2^\top,$$
where $Y=U_1WU_2^\top$ is a singular value decomposition of $Y$ with $W$ being a $p\times n$ rectangular diagonal matrix,
$$W:= \begin{pmatrix} D & \ 0\end{pmatrix}, \quad D=\diag(\sqrt{d_1}, \cdots, \sqrt{d_p}).$$
Here, $\sqrt{d_1} \geq \sqrt{d_2} \geq \cdots \geq \sqrt{d_p}>0$ are the singular values of $Y$ arranged in descending order. Thus, to study the singular values of $Y+\sqrt{t}X$, it suffices to assume that the initial data matrix is $W$. 
 %
%
We assume that the ESD of $V:=WW^\top$ has a regular square root behavior near the spectral edge, which is generally believed to be a necessary condition for the appearance of the Tracy-Widom law. Following \cite{edgedbm}, we state the regularity conditions in terms of the Stieltjes transform of $V$, 
$$m_V(z):= \frac1p\tr \left(V-z \right)^{-1}= \frac1p\sum_{i=1}^p \frac1{d_i -z},\quad z\in \C_+ .$$

\begin{definition}[$\eta_*$-regular]\label{assumption_edgebehavior}
 Let $\eta_*$ be a deterministic parameter satisfying $\eta_*:=n^{-\phi_*}$ for some constant $0<\phi_* \le 2/3$. We say $V$ is $\eta_*$-regular around the right-edge $\lambda_+ :=d_{i_0}$ for a fixed $i_0\in \N$, if the following properties hold 
for some constants $c_V, C_V>0$. 
\begin{enumerate}
\item For $z=E+\ii \eta$ with $\lambda_+ - c_V \le E \le \lambda_+ $ and $\eta_* + \sqrt{\eta_* |\lambda_+ - E|}\le \eta \le 10$, we have
\begin{equation}\label{regular1}
\frac{1}{C_V} \sqrt{|\lambda_+ - E| + \eta} \le \im m_V(z) \le C_V\sqrt{|\lambda_+ - E| + \eta}, 
\end{equation}
and for $z=E+\ii \eta$ with $\lambda_+ \le E \le \lambda_+ + c_V$ and $\eta_* \le \eta \le 10$, we have
\begin{equation}\label{regular2}
\frac{1}{C_V} \frac{\eta}{\sqrt{|\lambda_+ - E| + \eta}} \le \im m_V(z) \le C_V\frac{\eta}{\sqrt{|\lambda_+ - E| + \eta}}.
\end{equation}

\item There are no eigenvalues $d_i$ of $V$ insider the interval $[\lambda_++\eta_*,\lambda_+ + c_V]$.

\item We have $2c_V \le \lambda_+ \le C_V/2$ and $\norm{V} \le N^{C_V}$.
\end{enumerate}

\end{definition}

\begin{remark}
For our setting in Theorem \ref{thm_twgram}, the index $i_0$ is equal to $r+1$, which labels the first non-outlier eigenvalue of $V$. 
The motivation for (i) is as follows: if $m(z)$ is the Stieltjes transform of a density $\rho$ with square root behavior around $\lambda_+$, i.e., 
\begin{equation}\label{strong_sqrt} \rho(x)\sim \sqrt{(\lambda_+-x)_+} ,
\end{equation} 
then \eqref{regular1} and \eqref{regular2} hold for $\im m(z)$ with $\eta_*=0$. For a general $\eta_*>0$, \eqref{regular1} and \eqref{regular2} essentially mean that the empirical spectral density of $V$ behaves like a square root function near $\lambda_+$ on any scale larger than $\eta_*$. The condition $\eta\le 10$ in the definition is purely for definiteness of presentation---we can replace 10 with any constant of order 1.
\end{remark}

Regarding $t$ as a time parameter, we are interested in the dynamics of the edge eigenvalues of $\mathcal{Q}_t:=(W+\sqrt{t}X)(W+\sqrt{t}X)^\top$ with respect to $t$ for $0< t\ll 1$. Let $\rho_{w,t}$ be the 
asymptotic spectral density of $\mathcal{Q}_t$, and $m_{w,t}$ be the corresponding 
Stieltjes transform. It is known that for any $t>0$, $m_{w,t}$ is the unique solution to
\begin{equation}\label{originaleqaution0}
m_{w,t}=\frac{1}{p} \sum_{i=1}^p \frac{1}{d_i(1+c_n tm_{w,t})^{-1}-(1+c_ntm_{w,t})z+t(1-c_n)},
\end{equation} 
such that  $\im m_{w,t}>0$ for $z \in \mathbb{C}_+$ \cite{DOZIER20071099,DOZIER2007678,VLM2012}. 
Adopting the notations from free probability theory, we shall call $\rho_{w,t}$ the \emph{rectangular free convolution} (RFC) of $\rho_{w,0}$ with Marchenko-Pastur (MP) law at time $t$. Let $\lambda_{+,t}$ be the rightmost edge of the bulk  component of $\rho_{w,t}$.  
By Lemma \ref{lem_asymdensitysquare}, we know that $\rho_{w,t}$ has a square root behavior near $\lambda_{+,t}$.

We introduce the notation 
\begin{equation}\label{eq_defnzeta0}
 \zeta_t(z):=[1+c_n t m_{w,t}(z)]^2 z-(1-c_n)t[1+c_n t m_{w,t}(z)],
\end{equation} 
which is the so-called subordination function for the RFC. Then, we define the function
\begin{equation}\label{eq_subcompansion}
\Phi_t(\zeta)=[1-c_nt m_{w,0}(\zeta)]^2\zeta +(1-c_n)t[1-c_n tm_{w,0}(\zeta)],
\end{equation} 
and the parameter
\begin{equation}\label{eq_generealscaling}
\gamma_n \equiv \gamma_n(t):=\left( \frac12 \left[4\lambda_{+,t}\zeta_{+,t} + (1-c_n)^2 t^2\right] c_n^2 t^2\Phi_{t}^{''}(\zeta_{+,t}) \right)^{-1/3},
\end{equation}
where we have abbreviated that $\zeta_{+,t}\equiv \zeta_t(\lambda_{+,t}).$ Here we used the short-hand notation $\zeta_t(\lambda_{+,t})\equiv \lim_{\eta \downarrow 0} \zeta_t(\lambda_{+,t}+\ii \eta).$
Now we are ready to give the formal statement of Theorem \ref{informal2}.  

\begin{theorem}\label{thm_regularbm}
Suppose $W$ is $\eta_*$-regular in the sense of Definition \ref{assumption_edgebehavior} with $\eta_*=n^{-\phi_*}$. Suppose $t$ satisfies $n^{\epsilon} \eta_* \leq t^2 \leq n^{-\epsilon}$ for a small constant $\epsilon>0$. Fix any $k \in \N$, and let $f: \mathbb{R}^{k} \rightarrow \mathbb{R}$ be a test function such that 
\begin{equation*}
\norm{f}_{\infty} \leq C, \quad \norm{\nabla f}_{\infty} \leq C,
\end{equation*} 
for a constant $C>0$. Denote the eigenvalues of $\mathcal{Q}_t$ by $\lambda_1(t)\ge \lambda_2(t) \ge \cdots \ge \lambda_p(t)$. Then, we have that 
\begin{equation}
\begin{split}\label{eq_maintheoremeq1}
\lim_{n\to \infty} \mathbb{E} \left[ f\left(\gamma_n p^{2/3}(\lambda_{i_0}(t)-\lambda_{+,t}),  \cdots, \gamma_n p^{2/3}(\lambda_{i_0+k-1}(t)-\lambda_{+,t})\right) \right] \\
 =\lim_{n\to \infty} \mathbb{E}\left[ f\left(p^{2/3}(\mu_1^{\rm GOE}-2),\cdots, p^{2/3}(\mu^{\rm GOE}_{k}-2)\right) \right]  ,
\end{split}
\end{equation}
 where we recall that $\mu_i^{\rm GOE}$ are the eigenvalues of GOE as given by \eqref{eq GOE}.
\end{theorem}
 
Since the edge eigenvalues of GOE at $\pm 2$ obey the type-1 TW fluctuation \cite{TW1,TW}, by Theorem \ref{thm_regularbm} and the Portmanteau lemma we immediately obtain that 
\begin{equation*}
\lim_{n \rightarrow \infty} \mathbb{P}(\gamma_n p^{2/3}(\lambda_1(t)-\lambda_{+,t}) \leq x)=F_1(x), \quad \text{for all} \ \ x \in \mathbb{R},
\end{equation*}
where recall that $F_1$ is the type-1 TW distribution function.

Following the literature, we shall call the evolution of $\cal Q_t$ with respect to $t$ the \emph{rectangular matrix Dyson Brownian motion}, while we call the evolution of the {\it eigenvalues} of $\cal Q_t$ with respect to $t$ the \emph{rectangular Dyson Brownian motion}. 
We remark that the edge statistics of the symmetric Dyson Brownian motion (DBM) have been studied in \cite{edgedbm} for Wigner type matrix ensembles. The above Theorem \ref{thm_regularbm} extends the result there to Gram type matrix ensembles.


Before the end of this section, we summarize the basic ideas for the proof of Theorem \ref{thm_regularbm} and provide some (possibly helpful) heuristic discussions. The proof utilizes the matching and coupling strategy in \cite{BEYY2016,edgedbm}. First, in order to see the Tracy-Widom limit, we need to show that: (i) the rectangular free convolution (RFC) has a square root behavior near the right edge in the sense of \eqref{strong_sqrt}, and (ii) the edge eigenvalues of $\cal Q_t$ distribute according to the RFC on scales $\ge n^{-2/3}$. However, at $t=0$, the conditions \eqref{regular1} and \eqref{regular2} are not strong enough for both of these purposes. We need to run the dynamics for an amount of time $t_0$ to regularize both the RFC and the rectangular DBM. To show (i), we need a detailed analysis of the RFC, which has been done in another paper \cite{DY20202}. In particular, the analysis shows that under the $\eta_*$-regular assumption, we are able to obtain the square root behavior of RFC once $t_0 \gg \sqrt{\eta_*}$. We summarize some key properties of the RFC in Appendix \ref{sec_summaryrectnagular}. To show (ii), we need to prove some sharp local laws on the resolvent $(\cal Q_{t_0}-z)^{-1}$ for $z=E+\ii\eta$ with $E$ around the right edge and $\eta\ge n^{-2/3-\epsilon}$. These local laws are also proved in \cite{DY20202} and summarized in Section \ref{sec_localaws}. 


Next, we consider the rectangular DBM starting with the regular initial data $\mathcal{Q}_{t_0}$ (i.e., the evolution of the eigenvalues of $\cal Q_{t_0+t}$). It is known from the literature that the rectangular DBM satisfies a system of SDEs in equation \eqref{SDE rDBM}, which is the main tool for our proof. We couple it with the system of SDEs for another rectangular DBM of a properly chosen sample covariance matrix, whose Tracy-Widom law is known from the literature and whose asymptotic ESD matches that of $\cal Q_{t_0}$ around the right edge. 
Under this coupling, we will show that after shifted by respective right edges, the differences between the edge eigenvalues of the two rectangular DBMs are much smaller than $n^{-2/3}$ if we run them for an amount of time $t_1$ so that $n^{-1/3}\ll t_1 \ll t_0$. This key result is summarized in Theorem \ref{thm_firstkey}. Here, $t_1 \ll t_0$ is required so that the RFC does not change much from $t_0$ to $t_0+t_1$. In particular, the right edge $\lambda_{+,t}$ and the scaling factor $\gamma_n(t)$ remain approximately constant throughout the evolution. On the other hand, the condition $t_1\gg n^{-1/3}$ is essential because the ``relaxation time to equilibrium" of the coupled DBM is of order $n^{-1/3}$ at the right edge, which we will explain below.

To prove Theorem \ref{thm_firstkey}, it suffices to study the differences between the two coupled rectangular DBMs, denoted by $\{\lambda_i(t)\}$ and $\{\mu_i(t)\}$, respectively. For this purpose, we consider an interpolating process $z_i(t,\alpha)$ for $0 \leq \alpha \leq 1$ (cf. equation (\ref{eq_interpolationsde})), which is a rectangular DBM with initial data $z_i(0,\alpha)=\alpha \lambda_i(0)+(1-\alpha) \mu_i(0).$ Note that $z_i(t,0)=\mu_i(t)$ and $z_i(t,1)=\lambda_i(t)$, so we only need to control $\partial_\alpha z_i(t,\alpha)$ for $0\le \alpha \le 1$. In the proof, we find that it is more convenient to work with the singular values ${y}_i(t,\alpha):=\sqrt{z_i(t,\alpha)}$ and its shifted (by the right edge) version $\widetilde{y}_i(t,\alpha)$. Then, it suffices to control $\partial_\alpha \widetilde{y}_i(t,\alpha)$ by analyzing a system of SDEs given by equation (\ref{eq_singsde2}). However, for the analysis, we have to cut off the effect of bulk eigenvalues away from the edge, because the $\eta_*$-regular condition only describes the edge behavior of the initial data. Hence, similar to \cite{Bourgade2017,edgedbm}, we localize the analysis by introducing to the SDEs of $\widetilde{y}_i(t,\alpha)$ a short-range approximation (cf. equations (\ref{wtySDE1})--(\ref{wtySDE3})), whose solutions are denoted by $\widehat{y}_i(t,\alpha)$. Through a careful analysis, we find that the bulk eigenvalues indeed have negligible effect and the differences $|\widehat{y}_i(t,\alpha)-\widetilde{y}_i(t,\alpha)|$ are much smaller than $n^{-2/3}$ (cf. Lemma \ref{lem_approximationcontrol}). 

Now, armed with the above preparation, it remains to control $\partial_\alpha\widehat{y}_i(t,\alpha)$, which turns out to satisfy a deterministic parabolic PDE in (\ref{eq_ui}). Using the local laws for $(\cal Q_{t_0}-z)^{-1}$, we can show that the eigenvalues of $\cal Q_{t_0}$ satisfy a rigidity estimate (see Lemma \ref{lem_correct_rigi}), which implies that the initial data $\{\widehat{y}_i(t,\alpha)\}$ has an $\ell^q$ norm bounded by $n^{-2/3+\epsilon}$ for any $q\ge 4$ and small constant $\epsilon>0$. The last piece is then to prove an energy estimate for this PDE, which is summarized in Proposition \ref{prop_energyestimate}. Roughly speaking, Proposition \ref{prop_energyestimate} shows that the $\ell^\infty$ norm of the solution at time $t$ is smaller than the $\ell^q$ norm of the initial data by a factor of order $n^{-1/3}t^{-1}$.  
Consequently, as long as $t_1 \ge n^{-1/3+\delta}$ for a constant $\delta>0$ and $\epsilon$ is chosen small enough, the $\ell^{\infty}$ norm of the solution at time $t_1$ is much smaller than $n^{-2/3}.$ 

Combining all the above pieces shows that the eigenvalues of $\cal Q_{t}$ satisfy \eqref{eq_maintheoremeq1} for $t=t_0+t_1$. We can see from the above arguments that there are two conditions that lead to a lower bound for $t$: $t > t_0 \gg \sqrt{\eta_*}$ to ensure a regular square root behavior of the RFC and sharp local laws for $\cal Q_{t_0}$; $t > t_1 \gg n^{-1/3}$ to ensure the ``closeness" of the two coupled rectangular DBMs. Since we have assumed $0<\phi_* \leq 2/3$ in Definition \ref{assumption_edgebehavior}, we only need to take $t\gg \sqrt{\eta_*}$.
In fact, in the application to the proof of Theorem \ref{thm_twgram}, we will take $\phi_*=2/3$ so that we run the rectangular DBM for an amount of time $t \gg n^{-1/3}$. 

Finally, we discuss the comparison argument for Step 3 of the proof of Theorem \ref{thm_twgram}. First, it requires a moment matching condition, as is well-known in the random matrix theory literature. More precisely, we will construct another random Gram matrix, say $Y'=(y'_{ij})$, with independent entries that have the same mean $r_{ij}$ but different variances $\Var( y'_{ij}) =s_{ij}-t/n$. Then, the rectangular matrix DBM $Y'+\sqrt{t}X$ has the same mean matrix $R$ and variance matrix $S$ as $Y$. Now, applying Theorem \ref{thm_regularbm} shows that the edge eigenvalues (denoted by $\lambda_{i,t}'$) of $Y'+\sqrt{t}X$ satisfy the Tracy-Widom law around the right edge (denoted by $\lambda'_{+,t}$) of the corresponding RFC. It remains to show that the limiting law of the (shifted and rescaled) edge eigenvalues $p^{2/3}(\lambda_i-\lambda_+)$, $1\le i \le k$, of $Y$ match that of \smash{$p^{2/3}(\lambda_{i,t}'-\lambda_{+,t}')$}, $1\le i \le k$,. This uses a standard resolvent comparison argument in the literature, and the key technical input is the local law for the resolvent of $(Y'+\sqrt{t}X)(Y'+\sqrt{t}X)^\top$, which is given in Appendix \ref{sec_localaws}. While the resolvent comparison argument is almost the same as the ones in e.g., \cite{DY,LY}, it only gives that $p^{2/3}(\lambda_i-\lambda_{+,t}')$ satisfy the Tracy-Widom law. We still need to show that the difference between the right edges $\lambda_+$ and $\lambda_{+,t}'$ is much smaller than the Tracy-Widom fluctuation scale $n^{-2/3}$. By analyzing the Stieltjes transform of the RFC, we will see (cf. equation \eqref{edge_diff}) that for any small constant $\epsilon>0$,
\begin{equation}\label{eq_bbbb}
|\lambda_{+,t}-\lambda_+| \le n^{-2/3-\epsilon} + n^{-2/3+\epsilon}t + n^{-1+\epsilon} t^{-1} \quad \text{with high probability}. 
\end{equation}
Since we need to control the second and third terms on the right-hand side, we have to take $n^{-1/3+\delta}\le t\le n^{-\delta}$ for a constant $\delta>0$. To summarize, for the above argument to work, we need that $n^{-1/3+\delta}\le t\le n^{-\delta}\wedge \min s_{ij}$. In particular, taking a smaller $t$ means relaxing the lower bound on $s_{ij}$, so that we can handle a more general class of random Gram matrices. On the other hand, we have seen a lower bound $t\gg \sqrt{\eta_*}\vee n^{-1/3}$ for Step 2. Therefore, in the proof of Theorem \ref{thm_twgram}, we will take (almost) optimal parameters: $\eta_*=n^{-2/3}$ and $t = n^{-1/3+\delta}$. This also leads to the lower bound on $s_{ij}$ in \eqref{eq flatS}.  


\normalcolor


\section*{Acknowledgment}
The authors would like to thank the associated editor and three anonymous reviewers for many insightful comments and suggestions, which have resulted in a significant improvement of the paper. The first author is partially supported by NSF DMS-2113489.

\bibliographystyle{abbrv}
\bibliography{references,aniso_bib}

\appendices
	\setcounter{page}{1}
	\renewcommand{\thepage}{SI.\arabic{page}}
	\renewcommand{\thesection}{\Alph{section}}

\section{Proofs of Theorem \ref{thm_twgram}, Corollary \ref{cor_separable} and Corollary \ref{thm_twgraph_sparse}}\label{sec_prooftwgram}

We will use the following notion of stochastic domination, which was first introduced in \cite{Average_fluc} and subsequently used in many works on random matrix theory. 
It simplifies the presentation of the results and their proofs by systematizing statements of the form ``$\xi$ is bounded by $\zeta$ with high probability up to a small power of $n$".

\begin{definition}[Stochastic domination and  high probability event]\label{stoch_domination}
(i) Let
\[\xi=\left(\xi^{(n)}(u):n\in\bbN, u\in U^{(n)}\right),\hskip 10pt \zeta=\left(\zeta^{(n)}(u):n\in\bbN, u\in U^{(n)}\right),\]
be two families of nonnegative random variables, where $U^{(n)}$ is a possibly $n$-dependent parameter set. We say $\xi$ is stochastically dominated by $\zeta$, uniformly in $u$, if for any fixed (small) $\epsilon>0$ and (large) $D>0$, 
\[\sup_{u\in U^{(n)}}\bbP\left(\xi^{(n)}(u)>n^\epsilon\zeta^{(n)}(u)\right)\le n^{-D}\]
for large enough $n \ge n_0(\epsilon, D)$, and we will use the notation $\xi\prec\zeta$ to denote it. Throughout this paper, the stochastic domination will always be uniform in all parameters that are not explicitly fixed, such as the matrix indices and the spectral parameter $z$.  
If for some complex family $\xi$ we have $|\xi|\prec\zeta$, then we will also write $\xi \prec \zeta$ or $\xi=\OO_\prec(\zeta)$.

\vspace{5pt}

\noindent (ii) We say an event $\Xi$ holds with high probability if for any constant $D>0$, $\mathbb P(\Xi)\ge 1- n^{-D}$ for large enough $n$.
\end{definition}

The following lemma collects basic properties of stochastic domination, which will be used tacitly in the following proof.

\begin{lemma}[Lemma 3.2 of \cite{isotropic}]\label{lem_stodomin}
Let $\xi$ and $\zeta$ be two families of nonnegative random variables, $U$ and $V$ be two parameter sets and $C>0$ be a large constant.
\begin{enumerate}
\item Suppose that $\xi (u,v)\prec \zeta(u,v)$ uniformly in $u\in U$ and $v\in V$. If $|V|\le n^C$, then $\sum_{v\in V} \xi(u,v) \prec \sum_{v\in V} \zeta(u,v)$ uniformly in $u \in U$.

\item If $\xi_1 (u)\prec \zeta_1(u)$ and $\xi_2 (u)\prec \zeta_2(u)$ uniformly in $u\in U$, then $\xi_1(u)\xi_2(u) \prec \zeta_1(u)\zeta_2(u)$ uniformly in $u\in U$.

\item Suppose that $\Psi(u)\ge n^{-C}$ is deterministic and $\xi(u)$ satisfies $\mathbb E\xi(u)^2 \le n^C$ for all $u\in U$. Then if $\xi(u)\prec \Psi(u)$ uniformly in $u\in U$, we have that $\mathbb E\xi(u) \prec \Psi(u)$ uniformly in $u\in U$.
\end{enumerate}
\end{lemma}

We introduce the following bounded support condition, which has been used in a sequence of papers to improve the moment assumption, see e.g. \cite{DY, dingyang2, LY,  yang20190}.

\begin{definition}[Bounded support condition]\label{defn_boundedsupport} We say a random matrix $Y$ satisfies the bounded support condition with $\phi_n$ if
\begin{equation}\label{eq_boundsupoorteq}
\max_{i,j}\left|y_{ij}-\E y_{ij}\right| \prec \phi_n,
\end{equation}
where $\phi_n$ is a deterministic parameter satisfying that $n^{-1/2} \leq \phi_n \leq n^{-c_\phi}$ for some small constant $c_\phi>0.$ Whenever (\ref{eq_boundsupoorteq}) holds, we say that $Y$ has support $\phi_n.$ 
\end{definition}  

We introduce the following $(p+n)\times (p+n)$ symmetric block matrix
 \begin{equation}\label{linearize_block}
   H \equiv H(Y): = \left( {\begin{array}{*{20}c}
   { 0 } & Y   \\
   Y^\top & {0}  \\
   \end{array}} \right),
 \end{equation}
and its resolvent 
 \begin{equation}\label{eqn_defG000}
G(z)\equiv G(Y,z) :=(z^{1/2}H - z)^{-1}, \quad z\in \mathbb C_+ . 
\end{equation}
Moreover, for $\cal Q_1:= YY^\top$ and $\cal Q_2:= Y^\top Y$, we define their resolvents as
\begin{equation}\label{def_green000}
  \mathcal G_1(z) :=\left({\mathcal Q}_1 -z\right)^{-1} , \quad  \mathcal G_2 (z):=\left({\mathcal Q}_2-z\right)^{-1} .
\end{equation}
Using the Schur complement formula, it is easy to check that 
\begin{equation} \label{green2000}
G = \left( {\begin{array}{*{20}c}
   { \mathcal G_1} & z^{-1/2} \mathcal G_1 Y \\
   {z^{-1/2} Y^\top \mathcal G_1} & { \mathcal G_2 }  \\
\end{array}} \right)= \left( {\begin{array}{*{20}c}
   { \mathcal G_1} & z^{-1/2} Y \mathcal G_2   \\
   z^{-1/2} {\mathcal G_2}Y^\top & { \mathcal G_2 }  \\
\end{array}} \right).
\end{equation}
Thus, a control of $G$ yields directly a control of the resolvents $\mathcal G_1$ and $\mathcal G_2$. 
 We denote the empirical spectral density $\rho_1$ of $ {\mathcal Q}_{1}$ and its Stieltjes transform by
\begin{equation}\label{defn_m000}
\rho_1 := \frac{1}{p} \sum_{i=1}^p \delta_{\lambda_i( {\mathcal Q}_1)},\quad g_1(z) :=\int \frac{1}{x-z}\rho_1(\dd x)=\frac{1}{p} \mathrm{Tr} \, \mathcal G_1(z).
\end{equation}

In \cite{alt20171}, it has been shown that if $Y$ is centered, i.e. $R= 0$, then the diagonal entries $(\cal G_1)_{ii}$ and $(\cal G_2)_{jj}$ can be approximated by $M_{1,i}$ and $M_{2,j}$, respectively, where $M_1=(M_{1,1},\cdots, M_{1,p}):\C_+\to \C^p$ and $M_2=(M_{2,1},\cdots, M_{2,n}):\C_+\to \C^n$ are the unique solution of
\begin{equation}\label{vecself1}
\frac{1}{M_1}= - z - z SM_2,\quad \frac{1}{M_2}= - z - z S^\top M_1, 
\end{equation}
such that $\im M_{1,i} (z)>0$, $i=1,2,\cdots, p,$ and $\im M_{2,j} (z)>0$, $j=1,2,\cdots, n$, for all $z\in \C_+$. Here both $1/{M_1}$ and $1/{M_2}$ denote the entrywise reciprocals. Notice that if we plug the second equation of \eqref{vecself1} into the first equation, then $M_1$ satisfies equation \eqref{selfm1}, which shows that $M_1(z)=\mb(z)$. Then we define the asymptotic matrix limit of $G$ as 
\begin{equation}\label{defn_Piz}
\Pi(z):=\diag\left( M_{1,1}(z), \cdots M_{1,p}(z),M_{2,1}(z), \cdots M_{2,n}(z)\right).
\end{equation}
We define the following spectral domains: for some small constants $c_0, \vartheta>0$,
\begin{align*}
 \mathcal D(c_0,\vartheta)&:=\left\{z=E+\ii \eta: \lambda_+ -c_0\le E \le \lambda_+ + c_0 , n^{-1+\vartheta}\le \eta \le c_0^{-1}\right\},\\
 \mathcal D_{out}(c_0,\vartheta)&:=\mathcal D(c_0,\vartheta)\cap   \{z=E+\ii\eta: E\ge \lambda_+, n\eta\sqrt{\kappa + \eta} \ge n^\vartheta\}.
\end{align*}
Finally, we define the distance to the rightmost edge as
\begin{equation}
\kappa \equiv \kappa(z) := \vert E -\lambda_+\vert  \quad \text{for } \ \ z= E+\ii \eta.\label{KAPPA}
\end{equation}
Then, the following local law has been proved in \cite{alt20172}.

\begin{lemma}[Theorem 2.6 of \cite{alt20172}]\label{lem weakerlocal}
Assume that $Y$ is a $p\times n$ random matrix with real independent entries satisfying \eqref{eq_gramgeneralassumption} and that for any fixed $k\in \N$, 
\begin{equation}\label{high moment}
\mathbb{E}|y_{ij}-\E y_{ij}|^k \leq C_k s_{ij}^{k/2}, 
\end{equation}
for some constant $C_k>0$. 
Moreover, suppose that the variance matrix $S$ satisfies Assumption \ref{assum_gram}, and the mean matrix is $R=0$. Then there exists a constant $c_0>0$ such that the following \emph{averaged local laws} hold for any (small) constant $\vartheta>0$.
For any $z\in \mathcal D(c_0,\vartheta)$,  we have that
\begin{equation} 
 \vert g_1(z)-m(z) \vert \prec  (n \eta)^{-1}, \label{aver_in1} 
\end{equation}
where $m(z)$ is defined in \eqref{defnmz} and $g_1(z)$ is defined in \eqref{defn_m000}, and for any $z\in \mathcal D_{out}(c_0,\vartheta)$, we have a stronger estimate
\begin{equation}\label{aver_out1}
 | g_1(z)-m(z)|\prec   \frac{1}{n(\kappa +\eta)} + \frac{1}{(n\eta)^2\sqrt{\kappa +\eta}}.
\end{equation}
%
%
Both of the above estimates are uniform in the spectral parameter $z$.
\end{lemma}

\begin{remark}
Strictly speaking, the estimate \eqref{aver_out1} was not proved in \cite{alt20172}. However, its proof is standard by combining the results in \cite{alt20172} with a separate argument for  $z\in \mathcal D_{out}(c_0,\vartheta)$; see e.g. the proof of (2.20) in \cite{Semicircle}.
\end{remark}

As a consequence of \eqref{aver_in1} and \eqref{aver_out1}, we obtain the following rigidity estimate in Lemma \ref{rigid_lem} for the eigenvalues 
of $\cal Q_1$ near the right edge $\lambda_{+}$. We define the classical location $\gamma_j$ of the $j$-th eigenvalue as
\begin{equation}\label{gammaj000}
\gamma_j:=\sup_{x}\left\{\int_{x}^{+\infty} \rho (x)\dd x > \frac{j-1}{p}\right\} ,
\end{equation}
where $\rho$ was defined in \eqref{defnmz}. In other words, $\gamma_j$'s are the quantiles of  the asymptotic spectral density $\rho$ of $\cal Q_1$. Note that under the above definition, we have $\gamma_1 = \lambda_{+}$.
\begin{lemma}\label{rigid_lem}
Under the assumptions of Lemma \ref{lem weakerlocal}, for any $j$ such that $\lambda_{+} -  c_0/2 < \gamma_j \le \lambda_{+}$, we have 
\begin{equation}\label{rigidity000}
\vert \lambda_j(\cal Q_1) - \gamma_j \vert \prec  j^{-1/3}n^{-2/3} .
\end{equation}
\end{lemma}
\begin{proof}
The estimate \eqref{rigidity000} follows from  \eqref{aver_in1} and \eqref{aver_out1} combined with a standard argument using Helffer-Sj\"ostrand calculus. The details are already given in \cite{EKYY1,EYY,pillai2014}. 
\end{proof}

Combining Lemma \ref{rigid_lem} with the Cauchy interlacing theorem, we immediately obtain the following result when $R$ is non-zero and satisfies Assumption \ref{assum_meanmatrix}.
\begin{lemma}\label{rigid_lem222}
Assume that $Y$ is a $p\times n$ random matrix with real independent entries satisfying \eqref{eq_gramgeneralassumption} and \eqref{high moment}. Suppose that the variance matrix $S$ satisfies Assumption \ref{assum_gram} and the mean matrix $R$ satisfies Assumption \ref{assum_meanmatrix}. Denote the eigenvalues of $YY^\top$ by $\lambda_1\ge \lambda_2 \ge \cdots \ge \lambda_p$. Then there exists a constant $c_0>0$ such that the following statements hold for any small constant $\vartheta>0$.
\begin{itemize}

\item[(1)] {\bf Outliers}: The first $r$ eigenvalues satisfy
\begin{equation}\label{large_supp_outlier222}
\lambda_1\ge \lambda_2 \ge \cdots \ge \lambda_r \ge \lambda_+ +2 c_0.
\end{equation}

\item[(2)] {\bf Eigenvalues rigidity}: For any $j$ such that $\lambda_{+} -  c_0/4 < \gamma_j \le \lambda_{+}$, we have that
\begin{equation}\label{rigidity222}
\vert \lambda_{j+r} - \gamma_j \vert \prec  j^{-1/3}n^{-2/3} .
\end{equation}

\item[(3)] {\bf Averaged local law}: \eqref{aver_in1} holds uniformly for all $z\in \mathcal D(c_0,\vartheta)$, and \eqref{aver_out1} holds uniformly for all $z\in \mathcal D_{out}(c_0,\vartheta)$.
\end{itemize}
\end{lemma}
\begin{proof}
For simplicity, we denote $\wt Y:= Y-\E Y$, and the eigenvalues of $\wt{\cal Q}_1:=\wt Y\wt Y^\top$ by $\wt\lambda_1\ge \wt\lambda_2\ge \cdots \ge \wt\lambda_p$. As in \eqref{defn_m000}, we define the Stieltjes transform of the ESD of $\wt{\cal Q}_1$ as
$$\wt g_1(z) := \frac1p\sum_{i=1}^p \frac{1}{\wt\lambda_i -z }= \frac{1}{p} \mathrm{Tr} \frac{1}{\wt{\cal Q}_1- z} .$$
By Lemma \ref{rigid_lem}, the eigenvalues $\wt\lambda_i$ satisfy that
\begin{equation}\label{rigidity000.222}
\vert \wt\lambda_j - \gamma_j \vert \prec  j^{-1/3}n^{-2/3} ,
\end{equation}
for any $j$ satisfying $\lambda_{+} -  c_0/2 < \gamma_j \le \lambda_{+}$.  By the Cauchy interlacing theorem, we have that  
\begin{equation}\label{Cauchy_inter_eq}\wt\lambda_{j+r}\le \lambda_j \le \wt\lambda_{j-r},\quad 1\le j \le p,
\end{equation}
where we adopt the conventions $\wt\lambda_j=\infty$ if $j\le 0$, and $\wt\lambda_j=0$ if $j\ge p+1$. By the square root behavior of $\rho(x)$ around $\lambda_+$ as shown in \eqref{sqrtrho}, it is easy to get that
\begin{equation}\label{Cauchy_inter_eq2}|\gamma_j -\gamma_{j+2r}| \lesssim j^{-1/3}n^{-2/3} \end{equation}
for any $j$ satisfying $\lambda_{+} -  c_0 < \gamma_j \le \lambda_{+}$ as long as $c_0$ is sufficiently small. Combining \eqref{rigidity000.222}, \eqref{Cauchy_inter_eq}  and \eqref{Cauchy_inter_eq2}, we obtain \eqref{rigidity222}.  

Now suppose the mean matrix $R$ has SVD
$$R=\sum_{i=1}^r \sigma_i(R)u_i v_i^\top,$$
with $\sigma_1(R)\ge \sigma_2(R) \ge \cdots \ge \sigma_r(R)\ge (4+\tau)\sqrt{\mathfrak M}$ by \eqref{eq_supercritical}. Using Weyl's inequality for singular values, we obtain that
\begin{align*}
 \lambda_r & \ge\left[ \sigma_r(R) -\wt\lambda_1^{1/2}\right]^2  \ge \left[(4+\tau)\sqrt{\mathfrak M} - \left( \lambda_+ + \OO_\prec(n^{-2/3})\right)^{1/2}\right]^2  \ge  \left[(4+\tau)\sqrt{\mathfrak M} -  2\sqrt{\mathfrak M} + \OO_\prec(n^{-2/3}) \right]^2 \\
 &\ge \left[ 4+4\tau + \OO_\prec(n^{-2/3})\right] \mathfrak M \ge \left[ 1+\tau + \OO_\prec(n^{-2/3})\right] \lambda_+,
\end{align*}
where we used \eqref{rigidity000.222} for $\wt\lambda_1$ in the second step, and \eqref{bound_center} in the third and last steps. This gives \eqref{large_supp_outlier222}.

Finally, using \eqref{rigidity000.222} and the interlacing result \eqref{Cauchy_inter_eq}, we can show that for $z\in \mathcal D(c_0,\vartheta)$,
$$ \left|\wt g_1(z) - g_1(z) \right| \prec (n \eta)^{-1},$$
and for $z\in \mathcal D_{out}(c_0,\vartheta)$,
$$ \left|\wt g_1(z) - g_1(z) \right| \prec  \left[n(\kappa +\eta)\right]^{-1} .$$ 
We omit the details because it is a standard argument, which involves bounding the real and imaginary parts of $\wt g_1(z) - g_1(z)$ using \eqref{Cauchy_inter_eq}. Combining the above two estimates with Lemma \ref{lem weakerlocal} for $\wt g_1(z)$, we conclude part (3) of Lemma \ref{rigid_lem222}.
\end{proof}

From \eqref{high moment} and Markov's inequality, we get that the matrix $Y$ in Lemma \ref{lem weakerlocal} has support \smash{$\max_{i,j}s_{ij}^{1/2}$}. Now combining the analysis of the vector Dyson equation \eqref{selfm1} in \cite{alt20172} with the arguments for local law in \cite{DY}, we can relax the moment condition \eqref{high moment} to a weaker bounded support condition. 

\begin{lemma} \label{lem stronglocal}
Assume that $Y$ is a $p\times n$ random matrix with real independent entries satisfying \eqref{eq_gramgeneralassumption}. 
Suppose that the variance matrix $S$ satisfies Assumption \ref{assum_gram} and the mean matrix $R$ satisfies Assumption \ref{assum_meanmatrix}. Moreover, assume that $Y$ satisfies the bounded support condition (\ref{eq_boundsupoorteq}) with $\phi_n\le n^{-c_\phi}$ for a small constant $c_\phi>0$. Then there exists a constant $c_0>0$ such that the following estimates hold for any small constant $\vartheta>0$.
\begin{itemize}
\item[(1)] {\bf Averaged local law}: 
For any $z\in \mathcal D(c_0,\vartheta)$,  we have that
\begin{equation} 
 \vert g_1(z)-m(z) \vert \prec  \min\left\{\phi_n,\frac{\phi_n^2}{\sqrt{\kappa+\eta}}\right\} + \frac{1}{n \eta}, \label{aver_inlarge} 
\end{equation}
and for $z\in \mathcal D_{out}(c_0,\vartheta)$, we have a stronger estimate
\begin{equation}\label{aver_outlarge}
 | g_1(z)-m(z)|\prec  \min\left\{\phi_n,\frac{\phi_n^2}{\sqrt{\kappa+\eta}}\right\} +  \frac{1}{n(\kappa +\eta)} + \frac{1}{(n\eta)^2\sqrt{\kappa +\eta}}.
\end{equation}
\item[(2)] {\bf Entrywise local law}: For any $z\in \mathcal D(c_0,\vartheta)$, we have that
\begin{equation}\label{entry_lawlarge}
\max_{1\le i,j\le p+n}\left|G_{ij}(z)  -\Pi_{ij}(z) \right| \prec \phi_n+ \sqrt{\frac{\im m(z)}{n\eta}} +\frac1{n\eta} ,
\end{equation}
where $\Pi$ is defined in \eqref{defn_Piz}. 
\end{itemize}
All of the above estimates are uniform in the spectral parameter $z$.
\end{lemma}
\begin{proof}
With the stability analysis of equation \eqref{selfm1} in \cite[Section 3]{alt20172}, we can repeat the same proofs for Lemma 3.11 of \cite{DY} and Theorem 3.6 of \cite{yang20190} to conclude \eqref{aver_inlarge}--\eqref{entry_lawlarge}. We omit the details.
\end{proof}

Now we are ready to give the proof of Theorem \ref{thm_twgram}.
\begin{proof}[Proof of Theorem \ref{thm_twgram}]
Using the estimates in Lemma \ref{lem stronglocal}, we can repeat the proof for \cite[Theorem 2.7]{DY} almost verbatim to conclude \eqref{eq_twgramnece} and the following universality result as $n\to \infty$:
\begin{equation}
 \mathbb{P}\left[ \left( p^{{2}/{3}}(\lambda_{i+r} - \lambda_+) \leq x_i\right)_{1\le i \le k} \right]- \mathbb{P}^G\left[ \left( p^{{2}/{3}}(\lambda_{i+r} - \lambda_+) \leq x_i\right)_{1\le i \le k} \right]\to 0  \label{SUFFICIENT}
\end{equation}
for any $(x_1 , x_2, \ldots, x_k) \in \mathbb R^k$, where $\mathbb P^G$ denotes the law for $Y=(y_{ij})$ with independent Gaussian entries satisfying \eqref{eq_gramgeneralassumption}. To conclude \eqref{eq_twgram} and \eqref{SUFFICIENT2}, it remains to show that $(\varpi^{2/3} p^{{2}/{3}}(\lambda_{i+r} - \lambda_+))_{1\le i \le k} $ has the same asymptotic distribution as \smash{$(p^{{2}/{3}}(\mu_i^{\rm GOE} - 2) )_{1\le i \le k}$} in the Gaussian case.
For simplicity of notations, we only write down details of the proof for the $r=0$ case, which is based on Theorem \ref{thm_regularbm}, Lemma \ref{lem weakerlocal} and Lemma \ref{rigid_lem}. {The argument for the $r>0$ case is similar and will be discussed at the end of the proof. }

Let $t_0=n^{-1/3+\e_0}$ for a small constant $\e_0<\e_*$, where recall that $\e_*$ is the constant in \eqref{eq flatS}. Then, we pick the initial data matrix $W$ to be a $p\times n$ random matrix with independent Gaussian entries satisfying 
$$ \E w_{ij}=0, \quad \mathbb{E} w_{ij}^2=s_{ij}- {t_0}/{n}. $$
Let $X$ be an independent $p\times n$ matrix with i.i.d. Gaussian entries of mean zero and variance $n^{-1}$. Then, we have that
$$  Y\stackrel{d}{=} W+ \sqrt{t_0} X.$$
We regard $W+\sqrt{t}X$ as a rectangular matrix DBM starting at $W$, and at time $t_0$ it has the same distribution as $ Y$.

We now fix the notations for the proof. First, in light of (\ref{SUFFICIENT}), we denote the eigenvalues of $\cal Q:=(W+ \sqrt{t_0} X)(W+ \sqrt{t_0} X)^\top$ by $\lambda_1\ge \lambda_2\ge \cdots \ge \lambda_p$. We 
define its asymptotic spectral density $\rho$ and the corresponding Stieltjes transform $m(z)$ as in \eqref{defnmz}. Moreover, let $\lambda_+$ be the right edge of $\rho$, and $\gamma_j$ be the quantiles of $\rho$ defined as in \eqref{gammaj000}.
We denote the variance matrix of $W$ by $S_w=(s_{ij}-{t_0}/{n}:1\le i \le p, 1\le j \le n)$, and let $\mathbf M_w(z)=(M_{w,1}(z),\cdots, M_{w,p}(z)): \mathbb{C}_+ \rightarrow \mathbb{C}^p$ be the unique solution to the vector Dyson equation 
\begin{equation}\label{selfmw}
  \frac{1}{\mathbf M_w}=-z \mathbf{1}+S_w\frac{1}{\mathbf{1}+S_w^\top \mathbf M_w}, 
\end{equation}
such that $\im M_{w,k} (z)>0$, $k=1,2,\cdots, p,$ for any $z\in \C_+$. Then, we define $M_w(z):={p}^{-1} \sum_k M_{w,k}(z)$, which is the Stieltjes transform of the asymptotic spectral density of $WW^\top$, denoted by $\rho_{w}$. We denote the right edge of $\rho_{w}$ by $\lambda_{+,w}$, and define the quantiles of $\rho_{w}$ as
\begin{equation}\label{gammaj-t0}
\gamma_{j,w}:=\sup_{x}\left\{\int_{x}^{+\infty} \rho_{w} (x)dx > \frac{j-1}{p}\right\},\quad 1\le j \le p .
\end{equation}
Finally, following the notations in Section \ref{sec_proofstrategy}, we denote 
$$m_V(z)\equiv m_{w,0}(z):=p^{-1}\tr (WW^\top - z)^{-1},$$ and the eigenvalues of $WW^\top$ by $d_1\ge d_2 \ge \cdots \ge d_p$. Then, we define $m_{w,t}$ as in \eqref{originaleqaution0}, and let $\lambda_{+,t}$ be the rightmost edge of the rectangular free convolution $\rho_{w,t}$. 

We take $\eta_*=n^{-2/3+\e_1}$ for a small enough constant $0<\e_1<\e_0$. We first verify that $m_V$ is $\eta_*$-regular in the sense of Definition \ref{assumption_edgebehavior}. Notice that $W$ is also a random Gram matrix satisfying the assumptions of Lemma \ref{lem weakerlocal}. Denoting $z=E+\ii \eta$ and $\kappa=|E-\lambda_{+,w}|$, by \eqref{aver_in1} and \eqref{aver_out1} we have that for $ \lambda_{+,w} - c_0 \le E\le \lambda_{+,w}$ and $n^{-2/3+\vartheta}\le \eta \le 10,$
\begin{equation}\label{aver_inw0} 
 \vert m_{w,0}(z)-M_w(z) \vert \prec  (n \eta)^{-1} ,
 \end{equation}
 and for $ \lambda_{+,w}  \le E\le \lambda_{+,w} + c_0$ and $ n^{-2/3+\vartheta}\le \eta \le 10$,
\begin{equation}\label{aver_outw0}  
 \vert m_{w,0}(z)-M_w(z) \vert \prec   \frac{1}{n(\kappa +\eta)} + \frac{1}{(n\eta)^2\sqrt{\kappa +\eta}}.
\end{equation}
Moreover, as a consequence of the square root behavior of $\rho_{+,w}$ around $\lambda_{+,w}$ as given by \eqref{sqrtrho}, it is easy to show that 
\begin{equation}\label{Immc}
\vert M_w(z) \vert \sim 1,  \quad  \im M_w(z) \sim \begin{cases}
    {\eta}/{\sqrt{\kappa+\eta}}, & \text{ if } E\geq \lambda_{+,w} \\
    \sqrt{\kappa+\eta}, & \text{ if } E \le \lambda_{+,w}\\
  \end{cases},
\end{equation}
for any $z=E+\ii\eta$ satisfying that $\lambda_{+,w} - c_0\le E \le \lambda_{+,w}+c_0$ and $0\le \eta\le c_0^{-1}$ for a small enough constant $c_0>0$. {In this paper, given two sequences of positive values $a_n$ and $b_n,$ we use $a_n \sim b_n$ to mean that there exists a constant $C>0$ so that $C^{-1} a_n \leq b_n \leq C a_n.$} Finally, using \eqref{rigidity000} we get that 
\begin{equation}\label{rigiditywww}
\vert d_{j} - \gamma_{j,w} \vert \prec  j^{-1/3}n^{-2/3} ,
\end{equation}
for any $j$ such that $\lambda_{+,w} -  c_0/2 < \gamma_{j,w} \le \lambda_{+,w}$. Combining the above estimates \eqref{aver_inw0}--\eqref{rigiditywww}, we obtain that for some constants $0<c_V<c_0/2$ and $C_V>0$, the following estimates hold on a high probability event $\Xi$:  for $ d_1 - c_V \le E\le d_1$ and $ \eta_*\le \eta \le 10$,
\begin{equation*} 
 \frac{1}{C_V} \sqrt{|d_1 - E| + \eta} \le \im m_{w,0}(E+\ii \eta) \le C_V\sqrt{|d_1 - E| + \eta}\ ;  
 \end{equation*}
for $d_{1}  \le E\le d_1 + c_V$ and $\eta_*\le \eta \le 10$,
\begin{equation*} 
\frac{1}{C_V} \frac{\eta}{|d_1 - E| + \eta} \le \im m_{w,0}(E+\ii \eta) \le C_V\frac{\eta}{|d_1 - E| + \eta} \ .
\end{equation*}
 Thus, on event $\Xi$, $m_V$ is $\eta_*$-regular. Then, applying Theorem \ref{thm_regularbm} to $\cal Q=(W+ \sqrt{t_0} X)(W+ \sqrt{t_0} X)^\top$, we conclude that there exists a parameter $\gamma_n \sim 1$ such that for any fixed $k\in \N$,
\begin{equation}\label{diffTW1}
\left(\gamma_n p^{2/3}(\lambda_i-\lambda_{+,t_0})\right)_{1\le i \le k} \stackrel{d}{\sim} \left(p^{{2}/{3}}(\mu_i^{\rm GOE} - 2) \right)_{1\le i \le k},
\end{equation} 
where $\stackrel{d}{\sim}$ means that the two random vectors have the same asymptotic distribution. 
Now, to conclude the proof, it remains to show that  
	\begin{equation}\label{match edge}
	p^{2/3}|\lambda_{+,t_0}-\lambda_+| \to 0 \quad \text{in probability}.
\end{equation}
We recall that $\lambda_+$ is the right edge of the asymptotic density $\rho$, which by definition is also the rectangular free convolution of $\rho_w$ with MP law at time $t_0$. On the other hand, for a given $W$, $\lambda_{+,t_0}$ is the right edge of $\rho_{w,t}$, which is the rectangular  free convolution of $\rho_{w,0}:=p^{-1}\sum_{i=1}^p \delta_{d_i}$ with MP law at time $t_0$. Hence $\lambda_{+,t_0}$ and $\lambda_+$ are different quantities, but we can control their difference using \eqref{aver_inw0}, \eqref{aver_outw0} and \eqref{rigiditywww}. 

 Recalling the notation in \eqref{eq_defnzeta0}, we denote
\begin{equation*}
\zeta_{+,t_0} :=[1+c_nt_0 m_{w,t_0}(\lambda_{+,t_0})]^2 \lambda_{+,t_0} -(1-c_n)t_0[1+c_nt_0 m_{w,t_0}(\lambda_{+,t_0})],
\end{equation*} 
and 
$$  \zeta_+:=[1+c_nt_0 m(\lambda_+)]^2 \lambda_+-(1-c_n)t_0[1+c_nt_0 m(\lambda_+)].$$
Using (\ref{eq_edgebound}) below and \eqref{rigiditywww}, we can obtain that
\begin{equation}\label{eq rightedge2} |\zeta_+-\lambda_{+,w}|\sim |\zeta_{+,t_0} - \lambda_{+,w}| \sim t_0^2 .\end{equation}
Then, repeating the proof of Lemma \ref{lem edgemoving} (which is given in \cite[Lemma A.2]{DY20202}), we can obtain that 
\begin{equation}\label{eq boundedgediff}
\begin{split}
|\lambda_{+,t_0}-\lambda_+| & \lesssim |\zeta_{+,t_0} -\zeta_+| + t_0 |M_w(\zeta_{+,t_0})-M_w(\zeta_{+})| + t_0 \left| m_{w,0}(\zeta_{+,t_0})- M_w(\zeta_{+,t_0}) \right| ,
\end{split}
\end{equation}
and 
\begin{equation}\label{eq boundedgediff2} |\zeta_{+,t_0} -\zeta_+| \lesssim t_0^3 \left| m'_{w,0}(\zeta_{+,t_0})- M'_w(\zeta_{+,t_0}) \right|.
\end{equation}
Using the definition of $\gamma_{j,w}$, we can get that  
\begin{align}
   \left| m_{w,0}'(\zeta_{+,t_0})- M_w'(\zeta_{+,t_0}) \right| &= \left| \frac1p\sum_{j}\frac{1}{(d_j-\zeta_{+,t_0})^2}-\int_0^{\lambda_{+,w}}\frac{\rho_w(x)}{(x-\zeta_{+,t_0})^2}\dd x\right| \nonumber\\
&\le \sum_{j:\gamma_{j,w}>\lambda_{+,w}-c_0/2}\int_{\gamma_{j+1,w}}^{\gamma_{j,w}}\left| \frac{\rho_w(x)}{(d_j - \zeta_{+,t_0})^2}-\frac{\rho_w(x)}{(x- \zeta_{+,t_0})^2}\right|  \dd x+\OO(1) \nonumber\\
&\prec \sum_{j:\gamma_{j,w}>\lambda_{+,w}-c_0/2}\int_{\gamma_{j+1,w}}^{\gamma_{j,w}}  \frac{j^{-1/3}n^{-2/3}(|\lambda_{+,w}-x| + t_0^2)  \rho_w(x)}{|x- \zeta_{+,t_0}|^4}  \dd x+\OO(1) \nonumber\\
&\lesssim n^{-1}\int_{\lambda_{+,w}-c_0/2}^{\lambda_{+,w}}\frac{(\lambda_{+,w}-x) + t_0^2}{|(\lambda_{+,w}-x) + t_0^2|^4}\dd x+\OO(1) \lesssim \frac{1}{nt_0^4}.\label{derivdiff1}
\end{align}
Here in the third step we used that for $\gamma_{j+1,w}\le  x \le \gamma_{j,w}$,
$$|(x- \zeta_{+,t_0})^2-(d_j - \zeta_{+,t_0})^2|\prec  j^{-1/3}n^{-2/3}(|\lambda_{+,w}-x| + t_0^2),$$
by \eqref{rigiditywww}, \eqref{eq rightedge2} and $\lambda_{+,w}-\gamma_{j+1,w}\sim j^{2/3}n^{-2/3}$. In the fourth step, we used that $\rho_w(x)\sim \sqrt{\lambda_{+,w} - x}$ and $j^{-1/3}\sim n^{-1/3}(\lambda_{+,w} - x)^{-1/2}$. Plugging \eqref{derivdiff1} into \eqref{eq boundedgediff2}, we obtain that
\begin{equation}\label{eq boundedgediff3} |\zeta_{+,t_0} -\zeta_+| \prec  n^{-1}t_0^{-1}.
\end{equation}
Moreover, as a consequence of the square root behavior of $\rho_w$ around $\lambda_+$, it is easy to check that 
\begin{equation}\label{eq boundedgediff3.5}
 t_0|M_w(\zeta_{+,t_0})-M_w(\zeta_{+})| \lesssim t_0\frac{|\zeta_{+,t_0} -\zeta_+|}{\min\{|\zeta_+-\lambda_{+,w}|^{1/2},|\zeta_{+,t_0}-\lambda_{+,w}|^{1/2}\}} \prec  n^{-1}t_0^{-1},
\end{equation}
where we used \eqref{eq rightedge2} and \eqref{eq boundedgediff3} in the last step. Finally, we need to bound $ \left| m_{w,0}(\zeta_{+,t_0})- M_w(\zeta_{+,t_0}) \right|$. Denote $z_0:= \zeta_{+,t_0}+\ii \eta_0$ with $\eta_0:=n^{-2/3+\vartheta}$ for some small constant $\vartheta>0$. We now decompose $m_{w,0}(\zeta_{+,t_0})- M_w(\zeta_{+,t_0})$ as
\begin{align*}
&m_{w,0}(\zeta_{+,t_0})- M_w(\zeta_{+,t_0})= m_{w,0}(z_0)- M_w(z_0) + \cal K_1 + \cal K_2 ,
\end{align*}
where 
\begin{align*}
 \cal K_1 &:=\sum_{ j:\gamma_{j,w}>\lambda_{+,w}-c_0/2} \left(\frac1p \frac{1}{d_j - \zeta_{+,t_0}}-\int_{\gamma_{j+1,w}}^{\gamma_{j,w}} \frac{\rho_w(x)}{x- \zeta_{+,t_0}}\dd x\right)   - \sum_{ j :\gamma_{j,w}>\lambda_{+,w}-c_0/2} \left( \frac1p \frac{1}{d_j - z_{0}}-\int_{\gamma_{j+1,w}}^{\gamma_{j,w}} \frac{\rho_w(x)}{x-z_{0}}\dd x\right) ,\\
 \cal K_2 &:=  \sum_{ j:\gamma_{j,w}\le \lambda_{+,w}-c_0/2} \left( \frac1p \frac{1}{d_j - \zeta_{+,t_0}}-\int_{\gamma_{j,w}}^{\gamma_{j-1,w}} \frac{\rho_w(x)}{x-\zeta_{+,t_0}}\dd x\right)  -  \sum_{ j:\gamma_{j,w}\le \lambda_{+,w}-c_0/2} \left(\frac1p  \frac{1}{d_j - z_0}-\int_{\gamma_{j,w}}^{\gamma_{j-1,w}} \frac{\rho_w(x)}{x- z_0}\dd x\right) .
\end{align*}
By \eqref{aver_outw0}, we have
\begin{equation} \nonumber
 \left| m_{w,0}(z_0)- M_w(z_0) \right| \prec  \frac{1}{nt_0^2}+\frac{1}{(n\eta_0)^2 t_0}.
\end{equation}
Using \eqref{rigiditywww}, it is easy to bound $\cal K_2\lesssim \eta_0$ with high probability. Then using a similar argument as for \eqref{derivdiff1}, we can bound 
\begin{align*}
\cal K_1&\prec \sum_{j:\gamma_{j,w}>\lambda_{+,w}-c_0/2}\int_{\gamma_{j+1,w}}^{\gamma_{j ,w}}  \frac{j^{-1/3}n^{-2/3}(|\lambda_{+,w}-x| + t_0^2)  \rho_w(x)}{|x- \lambda_{+,t_0}|^2}  \dd x  \lesssim n^{-1}\int_{\lambda_{+,w}-c_0/2}^{\lambda_{+,w}}\frac{\dd x}{(\lambda_{+,w}-x) + t_0^2}  \lesssim n^{-1} \log n.
\end{align*}
Combining the above three estimates, we get that
\begin{equation}\label{eq boundedgediff4}
|m_{w,0}(\lambda_{+,t_0})- M_w(\lambda_{+,t_0})|\prec \eta_0 +  \frac{1}{nt_0^2}+\frac{1}{(n\eta_0)^2 t_0}.
\end{equation}
Now, with \eqref{eq boundedgediff}, \eqref{eq boundedgediff3}, \eqref{eq boundedgediff3.5} and \eqref{eq boundedgediff4}, we can bound that 
\begin{equation}\label{edge_diff}
|\lambda_{+,t_0}-\lambda_+| \prec t_0 \eta_0+\frac{1}{nt_0} + \frac{1}{(n\eta_0)^2}  .
\end{equation}
Plugging into $t_0=n^{-1/3+\e_0}$ and $\eta_0=n^{-2/3+\vartheta}$, we conclude \eqref{match edge} as long as $\e_0$ and $\vartheta$ are chosen such that $\e_0+\vartheta<1/3$.

Combining \eqref{diffTW1} and \eqref{match edge}, we obtain that $(\gamma_n p^{2/3}(\lambda_i-\lambda_{+}))_{1\le i\le p}$ converges weakly to the Tracy-Widom law. Furthermore, matching the gap between the quantiles $\gamma_1$ and $\gamma_2$ (recall \eqref{gammaj000}) of the density $\rho$ in \eqref{sqrtrho} and the one for the semicircle law $\rho_{sc}(2-x)=\pi^{-1}\sqrt{x}+\OO(x)$ around the right edge at $2$, we see that $\gamma_n$ must be $\varpi^{2/3}$. This concludes the proof of \eqref{eq_twgram} and \eqref{SUFFICIENT2}.

{Finally, we briefly discuss the proof for the $r>0$ case. In fact, its proof uses the same argument as above, except that we need to replace Lemma \ref{rigid_lem} with Lemma \ref{rigid_lem222} and apply Theorem \ref{thm_regularbm} with $i_0=r+1$. 
For example, the equation (\ref{diffTW1}) above should be replaced by
\begin{equation*}
\left(\gamma_n p^{2/3}(\lambda_{i+r}-\lambda_{+,t_0})\right)_{1\le i \le k} \stackrel{d}{\sim} \left(p^{{2}/{3}}(\mu_i^{\rm GOE} - 2) \right)_{1\le i \le k}.
\end{equation*} 
We omit the details.}
\end{proof}

Finally, we complete the proofs of Corollaries \ref{cor_separable} and \ref{thm_twgraph_sparse} using Theorem \ref{thm_twgram}.

\begin{proof}[Proof of Corollary \ref{cor_separable}]
In \cite{yang20190}, the following edge universality result was proved under the assumptions of this corollary:
\begin{equation}\label{add_edge_univ}
 \lim_{n\to \infty}\left\{\mathbb{P}\left[ \left( p^{{2}/{3}}(\lambda_{i} - \lambda_+) \leq x_i\right)_{1\le i \le k} \right]  - \mathbb{P}^G\left[ \left( p^{{2}/{3}}(\lambda_{i} - \lambda_+) \leq x_i\right)_{1\le i \le k} \right]\right\} = 0,  
\end{equation}
for all $(x_1 , x_2, \ldots, x_k) \in \mathbb R^k$, where $\mathbb P^G$ denotes the law for $\cal N$ with i.i.d. Gaussian entries of mean zero and variance $n^{-1}$. In particular, the condition \eqref{assm_3rdmoment} is not necessary if $A$ or $B$ is diagonal. 
Note that 
if $\cal N$ is Gaussian, then using the rotational invariance of multivariate Gaussian distribution, we can reduce $Q=YY^\top$ to a  
random Gram matrix satisfying \eqref{eq_gramgeneralassumption} with $R=0$ and variance matrix $S=(({a_ib_j})/n)$. Furthermore, notice that \eqref{eq_seperablecheckable} is stronger than \eqref{eq flatS} and equivalent to (A3) of Assumption \ref{assum_gram}. Hence $YY^\top$ satisfies the assumptions of Theorem \ref{thm_twgram} with $r=0$, which immediately concludes the proof.
\end{proof}

\begin{remark}\label{rem outliers}
Regarding Example \ref{example_separable}, suppose there are some spikes in the eigenvalue spectrum of $A$ and $B$ such that  $a_1\ge  \cdots \ge a_r \ge a_{r+1}+\tau$ and $b_1\ge  \cdots \ge b_s \ge b_{s+1}+\tau$ for some $r,s\in \N$ and a small constant $\tau>0$. Then it is easy to check that 
$$\min \left\{ \inf_{1\le i\le p} \frac{1}{p} \sum_{l} \frac{1}{\epsilon+|a_{i}-a_l|^2}, \inf_{1\le j\le n} \frac{1}{n} \sum_{l} \frac{1}{\epsilon+|b_{j}-b_l|^2} \right\} \lesssim \frac{1}{n\e} + 1,$$
and the condition \eqref{assmA4 rem} cannot hold for all $n$. Hence the condition \eqref{assmA4 rem} rules out the existence of outliers.
But the condition (\ref{assmA4 rem}) sometimes is too strong because it does not allow for any spikes or isolated eigenvalues in the eigenvalue spectrum of $A$ and $B$. (Here by an isolated eigenvalue of $A$, we mean an $a_i$ such that $a_{i+1}+\tau \le a_i \le a_{i-1}-\tau$ for some $1\le i\le p$ and a small constant $\tau>0$. For the isolated eigenvalues of $B$, we have a similar definition.) On the other hand, in \cite{dingyang2} we have found that a spike of $A$ or $B$ gives rise to an outlier only when it is above the BBP transition threshold. In fact, the following weaker regularity condition was used in \cite{dingyang2,yang20190}. For $\mb(z)$ in \eqref{selfm1}, we define another two holomorphic functions
\begin{equation*}
m_{1c}(z):=\frac{1}{n} \sum_{i=1}^p a_i m_i(z), \quad  m_{2c}(z):=\frac{1}{n}\sum_{j=1}^n \frac{b_j}{-z(1+b_j m_{1c}(z))}. 
\end{equation*}
Then, we say that the spectral edge $\lambda_+$ is \emph{regular} if for some constant $\tau>0,$
\begin{equation}\label{eq_oldregular}
1+m_{1c}(\lambda_+)b_1 \geq \tau, \quad 1+m_{2c}(\lambda_+)a_1 \geq \tau. 
\end{equation} 
This condition not only allows for isolated eigenvalues of $A$ and $B$, but also allows for zero $a_i$'s or $b_j$'s, that is, the lower bounds in \eqref{eq_seperablecheckable} can be relaxed to some extent. Compared with conditions \eqref{assmA4 rem} and \eqref{assmA4s rem}, the condition \eqref{eq_oldregular} is less explicit and harder to check, but it appears more often in the random matrix theory literature. 
\end{remark}

\begin{proof}[Proof of Corollary \ref{thm_twgraph_sparse}]
Combining \eqref{eq_sy} with Markov's inequality, we see that $Y$ satisfies the bounded support condition (\ref{eq_boundsupoorteq}) with $\phi_n =  q^{-1} \le n^{-1/3-c_\phi}$. Then Lemma \ref{lem stronglocal} holds, and in \cite[Lemma 3.11]{DY} we have shown that \eqref{aver_inlarge} and \eqref{aver_outlarge} imply the following weaker rigidity estimate than \eqref{rigidity000}: 
\begin{equation}\label{rigidityqqq}
\vert \lambda_j - \gamma_j \vert \prec  j^{-1/3}n^{-2/3}+\phi_n^2 .
\end{equation}
With \eqref{aver_inlarge}, \eqref{aver_outlarge} and \eqref{rigidityqqq} as the main inputs, using the same argument as for \cite[Theorem 2.7]{EKYY}, we can show that the edge statistics of $\cal Q$ match those of the Gaussian case in the sense of \eqref{SUFFICIENT} as long as $\phi_n\le n^{-1/3-c_\phi}$. Then we immediately conclude the proof using Theorem \ref{thm_twgram}. 
\end{proof}

\begin{remark}\label{rmk_techinical}
	
We make a few remarks on the technical assumptions (\ref{assm_3rdmoment}) and $q \geq n^{1/3+c_{\phi}}$ in Corollaries \ref{cor_separable} and \ref{thm_twgraph_sparse}, respectively. First, as mentioned in the proof of Corollary \ref{cor_separable}, we need to use the edge universality result \eqref{add_edge_univ} from \cite{yang20190}, where the vanishing third moment condition (\ref{assm_3rdmoment}) is needed (see the discussion below Theorem 3.6 in \cite{yang20190}). More precisely, a continuous self-consistent comparison argument is used in \cite{yang20190} to show that the non-Gaussian case is close to the Gaussian case in the sense of limiting distributions of edge eigenvalues. For the comparison argument to work, we need to match the third moment of $\widetilde y_{ij}$ with that of a standard Gaussian random variable, which leads to the condition (\ref{assm_3rdmoment}). 
However, we believe that (\ref{assm_3rdmoment}) is not necessary and can be removed with further theoretical development. 

Second, we believe that the condition $q \geq n^{1/3+c_{\phi}}$ in Corollary \ref{thm_twgraph_sparse} can be weakened to $q \geq n^{1/6+c_{\phi}}$. In fact, following the arguments in \cite{hwang2019}, we expect that (\ref{rigidityqqq}) can be sharpened to 
\begin{equation*}
\vert \lambda_j - \gamma_j - \delta(q) \vert \prec  j^{-1/3}n^{-2/3}+q^{-4},
\end{equation*}
for some deterministic shift $\delta(q)=\OO(q^{-2})$. As long as $q \ge n^{1/6+c}$, the term $q^{-4}$ will be much smaller than the Tracy-Widom scale $n^{-2/3}$, and the Tracy-Widom law around $\lambda_++\delta(q)$ can be established. 
However, when $q \ll n^{1/6},$ the limiting distribution of the second largest eigenvalue (i.e., the largest edge eigenvalue) of the Erd{\H o}s-R{\'e}nyi graph will become Gaussian \cite{TWsharp, huang2020}. We conjecture that a similar phenomenon also occurs for the model in Corollary \ref{thm_twgraph_sparse}. 


Since the above directions are not the focus of this paper, we will pursue them in future works.  
\end{remark}

\section{Rectangular free convolution and local laws} \label{sec_preliminary_feelocal}

In this section, we collect some basic estimates on the rectangular free convolution $\rho_{w,t}$ and its Stieltjes transform $m_{w,t}$ for an $\eta_*$-regular $V=WW^\top$ as in Definition \ref{assumption_edgebehavior}. Furthermore, we will state an (almost) sharp local law on the resolvent of $\mathcal{Q}_t=(W+\sqrt{t}X)(W+\sqrt{t}X)^\top$, and a rigidity estimate on the rectangular DBM $\{\lambda_i(t):1\le i \le p\}$. 
These estimates will serve as important inputs for the detailed analysis of the rectangular DBM in Section \ref{sec_calculationDBM} below. Most of the results in this section were proved in \cite{DY20202} under more general assumptions on $X$, and we will provide the exact reference for each of them.  Without loss of generality, throughout this section, we assume that $i_0=1$. { The general case with $i_0>1$ will be discussed in Remark \ref{rmk_iogeqonecase} below. }


\subsection{Properties of rectangular free convolution}\label{sec_summaryrectnagular}

For simplicity, we denote $b_t(z):=1+c_n t m_{w,t}(z).$ It is easy to see from (\ref{originaleqaution0}) that $b_t$ satisfies the following equation 
\begin{equation}\label{originaleqautionbbb}
b_t =1+\frac{c_n t}{p} \sum_{i=1}^p \frac{1}{b_t^{-1} d_i-b_t z+t(1-c_n)}.
\end{equation} 
Recalling $\zeta_t$ defined in \eqref{eq_defnzeta0}, 
the equation \eqref{originaleqautionbbb} can be also rewritten as
\begin{equation}\label{eq_keyequation11}
\frac{1}{c_nt}\left(1-\frac{1}{b_t}\right)=m_{w,0}(\zeta_t).
\end{equation}
Recall that $\rho_{w,t}$ is the asymptotic probability density associated with $m_{w,t}$, and let $\mu_{w,t}$ be the corresponding probability measure. Moreover, we denote the support of $\mu_{w,t}$ by $S_{w,t}$, with a right-most edge at $\lambda_{+,t}$. We first summarize some basic properties of these quantities, which have been proved in previous works \cite{DOZIER20071099,DOZIER2007678,VLM2012}. 

\begin{lemma}[Existence and uniqueness of asymptotic density] \label{existuniq}
The following properties hold for any $t>0$.
\begin{itemize}
\item[(i)] There exists a unique solution $m_{w,t}$ to equation \eqref{originaleqaution0} satisfying that $\im m_{w,t}(z)> 0$ and $\im z m_{w,t}(z)> 0$ for $z\in \C_+$.

\item[(ii)] For all $x \in \mathbb{R}\setminus\{0\},$ $\lim_{\eta  \downarrow 0} m_{w,t}(x+\ii \eta)$ exists, and we denote it by $m_{w,t}(x).$ The function $m_{w,t}(x)$ is continuous on $\mathbb{R}\setminus \{0\}$, and the measure $\mu_{w,t}$ has a continuous density $\rho_{w,t}$ given by $\rho_{w,t}(x)=\pi^{-1} \im m_{w,t}(x)$ on $\mathbb{R}\setminus\{0\}$. Finally, $m_{w,t}(x)$ is a solution to (\ref{originaleqaution0}) for $z=x$. 

\item[(iii)] For all $x \in \mathbb{R}\setminus\{0\},$ $\lim_{\eta  \downarrow 0} \zeta_t(x+\ii \eta)$ exists, and we denote it by $\zeta_t(x).$ Moreover, we have $\im \zeta_t(z)>0$ for $z\in \C_+$. 

\item[(iv)] For any $z\in \C_+$, we have $\re b_t(z)>0$ and $|m_{w,t}(z)|\le (c_nt|z|)^{-1/2}.$
\item[(v)]  The interior $\mathtt{Int}(S_{w,t})$ of $S_{w,t}$ is given by 
\begin{equation*}
\mathtt{Int}(S_{w,t})=\{x>0: \im  m_{w,t}(x)>0  \}=\{x>0: \im  \zeta_t(x)>0 \} , 
\end{equation*}
which is a subset of $\R_+:=\{x\in \R:x>0\}$. Moreover, $\zeta_t(x) \notin \{ d_1, \cdots, d_p\}$ if $x \notin \partial S_{w,t}.$ 
\end{itemize}
\end{lemma}
\begin{proof}
(i) follows from \cite[Theorem 4.1]{DOZIER2007678}, (ii) and (iii) follow from \cite[Lemma 2.1]{DOZIER20071099} and \cite[Proposition 1]{VLM2012}, (iv) follows from \cite[Lemma 2.1]{DOZIER20071099}, and (v) follows from \cite[Propositions 1 and 2]{VLM2012}.
\end{proof}


The following lemma characterizes the right-most edge of $S_{w,t}.$ 
Using $\zeta_t$ in \eqref{eq_defnzeta0} and the definition of $b_t$, we can rewrite the equation \eqref{eq_keyequation11} as
\begin{equation}\label{simplePhizeta}\Phi_t(\zeta_t(z))=z,
\end{equation}
where $\Phi_t$ is defined in (\ref{eq_subcompansion}). 
We recall that by definition 
\begin{equation}\label{eq_defng}
m_{w,0}(\zeta)=p^{-1} \operatorname{Tr}[(WW^\top-\zeta)^{-1}] = \frac{1}{p}\int \frac{1}{x-\zeta}\dd\mu_{w,0}(x).
\end{equation}
In \cite{VLM2012}, the authors characterize the support of $\mu_{\omega,t}$ and its edges using the local extrema of $\Phi_t$ on $ \R$. 

\begin{lemma}\label{lem_eigenvalueslarger} 
Fix any $t>0$. The function $\Phi_t(x)$ on $\R\setminus \{0\}$ admits $2q$ positive local extrema counting multiplicities for some $q\in \N$. 
The preminages of these extrema are denoted by $\zeta_{1,-}(t)<0<\zeta_{1,+}(t) \leq \zeta_{2,-}(t) \leq \zeta_{2,+}(t) \leq \cdots \leq \zeta_{q,-}(t) \leq \zeta_{q,+}(t),$ and they all belong to the set $\{\zeta \in \mathbb{R}: 1-c_n tm_{w,0}(\zeta)>0 \}.$ Moreover, the rightmost edge of $\supp(\mu_{w,t})$ is given by $\lambda_{+}(t)=\Phi_t(\zeta_{q,+}(t))$, and $\Phi_t$ is strictly increasing on the intervals $(-\infty, \zeta_{1,-}(t)],\ [\zeta_{1,+}(t), \zeta_{2,-}(t)],\ \cdots, \ [\zeta_{q-1,+}(t), \zeta_{q,-}(t)]$ and $[\zeta_{q,+}(t), \infty).$ Finally, for $k=1,2,\cdots, q,$ each interval $(\zeta_{k,-}(t), \zeta_{k,+}(t))$ contains at least one of the elements in $\{d_1,\cdots, d_p, 0\}$, and in particular, $d_1 \in (\zeta_{q,-}(t) ,\zeta_{q,+}(t))$. 
\end{lemma}
\begin{proof}
See \cite[Proposition 3]{VLM2012} and the discussion below \cite[Theorem 2]{VLM2012}, or see \cite[Lemma 1]{loubaton2011}.
\end{proof}

Now, we rewrite \eqref{eq_keyequation11} into another equation in terms of $\zeta$ and $z$. We focus on $z\in \C_+$ with $\re z>0$. Then, we can solve from \eqref{eq_defnzeta0} that
\begin{equation}\label{eq_banotherform}
b_t=\frac{t(1-c_n)+\sqrt{t^2(1-c_n)^2+4 \zeta z}}{2z},
\end{equation}
where we have chosen the branch of the solution such that Lemma \ref{existuniq} (iv) holds.
Plugging \eqref{eq_banotherform} into (\ref{eq_keyequation11}), we find that $(z, b_t)$ is a solution to (\ref{eq_keyequation11}) if and only if $(z,\zeta_t)$ is a solution to
\begin{equation}\label{final_derivation}
F_t(z, \zeta)=0, \quad \text{with}\quad F_t(z, \zeta):=1+\frac{t(1-c_n)-\sqrt{t^2(1-c_n)^2+4 \zeta z}}{2 \zeta}-c_n t m_{w,0}(\zeta). 
\end{equation}
Since the two equations $\Phi_t(\zeta_t(x))=x$ and $F_t(x,\zeta_t(x))=0$ are equivalent, from Lemma \ref{lem_eigenvalueslarger} we can obtain the following characterization of the edges of $S_{w,t}$.

\begin{lemma}\label{lem_partialedge} Denote $a_{k,\pm}(t):= \Phi_t(\zeta_{k,\pm}(t))$, $1\le k \le q$. Then $(a_{k,\pm}(t), \zeta_{k,\pm}(t))$ are real solutions to
\begin{equation}\label{eq_Ftderivative}
F_t(z,\zeta)=0, \quad \text{and} \quad \frac{\partial F_t}{\partial \zeta}(z,\zeta)=0. 
\end{equation} 
\end{lemma} 
\begin{proof}
By chain rule, if we regard $z$ as a function of $\zeta,$ then we have {
\begin{equation}\label{partialFt}
0=\frac{\dd F_t}{\dd \zeta}=\frac{\partial F_t}{\partial \zeta}+\frac{\partial F_t}{\partial z} z'(\zeta). 
\end{equation}}
By Lemma \ref{lem_eigenvalueslarger}, we have $\Phi_t'(\zeta_{k,\pm})=0$ since $\zeta_{k,\pm}$ are local extrema of $\Phi_t$. Then, from equation \eqref{simplePhizeta}, we can derive that 
$$z'(\zeta_{k,\pm})=\Phi_t'(\zeta_{k,\pm})=0,$$ 
with $z(\zeta_{k,\pm})=a_{k,\pm}$. Plugging this equation into \eqref{partialFt}, we get
$$\frac{\partial F_t}{\partial \zeta}(a_{k,\pm}, \zeta_{k,\pm})=0,$$
which concludes the proof.
\end{proof}

Now we use Lemma \ref{lem_partialedge} to derive an expression for the derivative $\partial_t \lambda_{+,t}$, which will be used in the analysis of the rectangular DBM in Section \ref{sec_calculationDBM}.  
Taking derivative of \eqref{final_derivation} with respect to $t$ and using \eqref{eq_Ftderivative}, we get that for $z=\lambda_{+,t}$ and $\zeta_{+,t} :=\zeta_t(\lambda_{+,t})$,
\begin{align*}
\frac{\partial F(t,\lambda_{+,t},\zeta_{+,t})}{\partial t}  + \frac{\partial F(t,\lambda_{+,t},\zeta_{+,t})}{\partial z}\frac{\dd\lambda_{+,t}}{\dd t} =0,
\end{align*}
where we denoted $F(t,z,\zeta)\equiv F_t(z,\zeta)$. From this equation, we can solve that
\begin{align} 
\frac{\dd \lambda_{+,t}}{\dd t}&=\left[\frac{1-c_n}{2\zeta_{+,t}}-c_n m_{w,0}(\zeta_{+,t}) \right]\sqrt{t^2(1-c_n)^2+4 \zeta_{+,t} \lambda_{+,t}} - \frac{(1-c_n)^2t}{2\zeta_{+,t}} \nonumber\\
&=\left[\frac{1-c_n}{2\zeta_{+,t}}-\frac{c_n m_{w,t}(\lambda_{+,t})}{b(\lambda_{+,t})} \right]\sqrt{t^2(1-c_n)^2+4 \zeta_{+,t} \lambda_{+,t}} - \frac{(1-c_n)^2t}{2\zeta_{+,t}},\label{eq_defnpsiderviative0}
\end{align}
where we used \eqref{eq_keyequation11} in the second step.


Next we describe some more precise properties of $\rho_{w,t}$ and $m_{w,t}$ for an $\eta_*$-regular
$V$ as in Definition \ref{assumption_edgebehavior}. For the following results, we always assume that 
\begin{equation}\label{restr_t}
t:=n^{-1/3+\omega},  \quad \text{with } \ \   1/3-\phi_*/2+\e/2 \le \omega \le 1/3-\e/2,
\end{equation}
for some constant $\e>0$. Note that under this condition, we have $n^\e \eta_*\le t^2 \le n^{-\e}$. 


\begin{lemma}[Lemma 3.7 of \cite{DY20202}]\label{lem rightedge}
Suppose $V=WW^\top$ is $\eta_*$-regular and $t$ satisfies \eqref{restr_t}. Then, we have $\zeta_{+,t} \geq \lambda_+$ and 
\begin{equation}\label{eq_edgebound}
\zeta_{+,t} - \lambda_+\sim t^2 .
\end{equation}
\end{lemma}
 
The following lemma describes the square root behavior of the asymptotic density $\rho_{w,t}$. 
\begin{lemma}[Lemmas 3.18 and 3.19 of \cite{DY20202}]\label{lem_asymdensitysquare} Suppose $V=WW^\top$ is $\eta_*$-regular and $t$ satisfies \eqref{restr_t}. If $\kappa:=|E-\lambda_+| \leq 3c_V/4,$ then the asymptotic density satisfies that
\begin{equation}\label{sqrtdensity}
\rho_{w,t}(E) \sim \sqrt{(\lambda_{+,t}-E)_+}  .
\end{equation}
Moreover, if $-\tau t^2 \le E-\lambda_{+,t} \leq 0$ for a sufficiently small constant $\tau>0$, then we have that
\begin{equation}\label{sqrtdensity2}
\rho_{w,t}(E)=\frac{1}{\pi}\sqrt{\frac{2(\lambda_{+,t}-E)}{[4\lambda_{+,t}\zeta_{+,t} + (1-c_n)^2 t^2] c_n^2 t^2\Phi_t^{''}(\zeta_{+,t})}} \left[1+\OO\left(\frac{|E-\lambda_{+,t}|}{t^2} \right)\right],
\end{equation}
where {$t^2\Phi_t^{''}(\zeta_+(t)) \sim 1$}. Finally, as a consequence of \eqref{sqrtdensity}, the following estimates hold:
\begin{equation}\label{eq_imasymptoics}
|m_{w,t}|\lesssim 1,\quad  \im m_{w,t}(z) \sim
\begin{cases}
\sqrt{\kappa+\eta}, & E \leq \lambda_{+,t} \\
 {\eta}/{\sqrt{\kappa+\eta}}, & E \geq \lambda_{+,t} 
\end{cases},
\end{equation}
for any $z=E+\ii\eta$ satisfying $|E-\lambda_+| \leq 3c_V/4$ and $0\le \eta\le 10$. 
\end{lemma}

We also need to control the derivative $\partial_z m_{w,t}(z)$. First, note that with the definition of $m_{w,t}$, we can get the trivial estimate
\begin{equation}\label{eq_trivialbounderivative}
\left| \partial_z m_{w,t}(z) \right| = \left|\int \frac{\dd \mu_{w,t}(x)}{(x-z)^2}\right|\leq \int \frac{\dd \mu_{w,t}(x)}{|x-z|^2} = \frac{\im m_{w,t}}{\eta}.
\end{equation}
Moreover, we claim the following estimates.

\begin{lemma} [Lemma 3.20 of \cite{DY20202}]\label{lem_partialm}
Suppose $V=WW^\top$ is $\eta_*$-regular and $t$ satisfies \eqref{restr_t}. Consider any $z=E+\ii\eta$ with $\kappa:=|E-\lambda_+| \leq 3c_V/4$ and $0\le \eta\le 10$. If $\kappa+\eta \leq t^2,$ then we have that
\begin{equation}\label{eq_bound1}
\left| \partial_z m_{w,t}(z) \right| \lesssim ( {\kappa+\eta})^{-1/2}.
\end{equation}
If $\kappa+\eta \geq t^2$,  we have that for $E \geq \lambda_{+,t},$ 
\begin{equation}\label{eq_bound2}
|\partial_z m_{w,t}(z)| \lesssim  ( {\kappa+\eta})^{-1/2},
\end{equation}
and for $E \leq \lambda_{+,t}$, 
\begin{equation}\label{eq_bound3}
|\partial_z m_{w,t}(z)| \lesssim \frac{ \sqrt{\kappa+\eta}}{t \sqrt{\kappa+\eta}+\eta}. 
\end{equation}
\end{lemma}


Finally, in Section \ref{sec_calculationDBM}, we will need to compare the edge behaviors of two free rectangular convolutions satisfying certain matching properties. Specifically, let $t_0= N^{-1/3+\omega_0}$ for some constant $0<\omega_0<1/3$. We consider two probability measures $\rho_1$ and $\rho_2$ having densities on the interval $[0,2\psi]$ with $\psi\sim 1$ being a positive constant, such that for some constant $c_\psi>0$ the following properties hold:
\begin{equation}\label{eq_closerho}
\rho_1(\psi-x)=\rho_2(\psi-x)\left[1+\OO\left( {|x|}/{t_0^2} \right) \right], \quad 0 \leq x \leq c_\psi t_0^2,
\end{equation}
and 
\begin{equation}\label{sqrt12}
\rho_1(x)=\rho_2(x)=0 \ \ \text{on}\ \ [\psi,2\psi],\quad \rho_1(x)\sim \rho_2(x) \sim \sqrt{\psi-x} \ \ \text{on}\ \ [\psi- c_\psi,\psi].
\end{equation} 
Let $\rho_{1,t}$ and $\rho_{2,t}$ be the free rectangular convolutions of the MP law with $\rho_1$ and $\rho_2$, respectively. Moreover, the Stieltjes transform of $\rho_{i,t}$, denoted by $m_{i,t}$, satisfies a similar equation as in \eqref{eq_keyequation11}:
\begin{equation*}
\frac{1}{c_n t}\left(1-\frac{1}{b_{i,t}}\right)=\int\frac{\rho_i(x)}{x-\zeta_{i,t}}\dd x ,\quad i=1,2,
\end{equation*}
where
\begin{equation}\label{defnbit} b_{i,t}(z):=1+c_ntm_{i,t}(z),\quad  \zeta_{i,t}(z):= zb_{i,t}^2 - (1-c_n)t b_{i,t}. \end{equation}
For $i=1,2,$ let $\lambda_{+,i}(t)$ be the right edge of $\rho_{i,t}$, and denote $\zeta_{+,i}(t):=\zeta_{i,t}(\lambda_{+,i}(t))$. Due to the matching condition \eqref{eq_closerho}, we can show that $\zeta_{+,1}(t)$ and $\zeta_{+,2}(t)$ are close to each other with a distance of order $\oo(t^2)$ for $t\ll t_0$.

\begin{lemma}[Lemma A.2 of \cite{DY20202}]\label{lem edgemoving}
Suppose \eqref{eq_closerho} and \eqref{sqrt12} hold. Then there exists a constant $C>0$ such that for any $0\le t\le t_0$, 
 \begin{equation}\label{eq_edgecomparebound}
| \zeta_{+,1}(t)- \zeta_{+,2}(t)| \leq \frac{C t^3}{t_0},
\end{equation}
and
 \begin{equation}\label{eq_edgecomparebound2}
|\lambda_{+,1}(t)-\psi|+|\lambda_{+,2}(t)-\psi| \leq Ct.
\end{equation}
 \end{lemma}

 The following matching estimates will play an important role in constructing the short-range
approximation of the rectangular DBM in Section \ref{subsec_short}. 
  
\begin{lemma}[Lemmas A.4 and A.5 of \cite{DY20202}]\label{lem comparison1}
Suppose \eqref{eq_closerho} and \eqref{sqrt12} hold, and $0 < t \leq t_0 n^{-\epsilon_0}$ for a constant $\epsilon_0>0$. If $0 \leq x \leq \tau n^{-2\epsilon} t_0^2$ for some small enough constants $\tau,\epsilon >0$, then for any (large) constant $D>0$ we have that
\begin{equation}\label{eq_densitycomparison}
\rho_{1,t}(\lambda_{+,1}-x)=\rho_{2,t}(\lambda_{+,2}-x)\left[1+\OO\left(\frac{n^{\epsilon}t}{t_0}+n^{-D}\right)\right] ,
\end{equation}
and
\begin{equation} \label{eq_realpartinside}
\begin{split}
 \left| \re [m_{1,t}(\lambda_{+,1}-x)-m_{1,t}(\lambda_{+,1})]-\re [m_{2,t}(\lambda_{+,2}-x)-m_{2,t}(\lambda_{+,2})] \right|  \lesssim \left( \frac{n^{\epsilon}}{t_0}+\frac{n^{-D}}{t}\right)x.
 \end{split}
\end{equation}
If $ 0\le x \le \tau n^{-2\epsilon} t_0 t ,$ 
then for any (large) constant $D>0$ we have that
\begin{equation}\label{eq_realpartoutside}
\begin{split}
\left| \re   [m_{1,t}(\lambda_{+,1}+x)-m_{1,t}(\lambda_{+,1})]-\re   [m_{2,t}(\lambda_{+,2}+x)-m_{2,t}(\lambda_{+,2})] \right|  \lesssim \left(n^{\epsilon} \frac{t^{1/2} }{t_0^{1/2}} + n^{-D}\frac{t_0^{1/2}}{t^{1/2} }\right)x^{1/2}.
\end{split}
\end{equation}
\end{lemma}

\subsection{Local laws}\label{sec_localaws} 

In this section, we state the local laws and rigidity estimates for the rectangular DBM considered in this paper. We first consider $t$ satisfying \eqref{restr_t}. Define the following $(p+n)\times (p+n)$ symmetric block matrix 
$$H_t := \begin{bmatrix} 0 & W +\sqrt{t} X \\ (W+\sqrt{t}X)^\top & 0\end{bmatrix}.$$

\begin{definition}[Resolvents]\label{resol_not}
We define the resolvent of $H_t$ as
 \begin{equation}\label{eqn_defG}
G(z)\equiv G_t(X,W,z):=(z^{1/2}H_t - z)^{-1}, \quad z\in \mathbb C_+ . 
\end{equation}
For $ \cal Q_{1,t}:= (W+\sqrt{t}X)(W+\sqrt{t}X)^\top$ and $\cal Q_{2,t}:= (W+\sqrt{t}X)^\top(W+\sqrt{t}X)$, we define the resolvents 
\begin{equation}\label{def_green}
  \mathcal G_1(z)\equiv \cal G_{1,t}(X,W,z) :=\left({\mathcal Q}_{1,t} -z\right)^{-1} , \ \ \ \mathcal G_2 (z)\equiv \cal G_{2,t}(X,W,z):=\left({\mathcal Q}_{2,t}-z\right)^{-1} .
\end{equation}
 We denote the empirical spectral density $\rho_{1,t}$ of $ {\mathcal Q}_{1,t}$  and its transform by
\begin{equation}\nonumber
\rho_1\equiv \rho_{1,t}(X,W,z):= \frac{1}{p} \sum_{i=1}^p \delta_{\lambda_i( {\mathcal Q}_{1,t})},\quad  m_1(z)\equiv m_{1,t}(X,W,z):=\int \frac{1}{x-z}\rho_1(\dd x)=\frac{1}{p} \mathrm{Tr} \, \mathcal G_1(z).
\end{equation}

\end{definition}

For any constant $\vartheta>0$, we define the spectral domain  
\begin{equation}\label{eq_domantheta}
\begin{split}
\mathcal{D}_{\vartheta}:=&\left\{z=E+\ii \eta: \lambda_{+,t}- \frac34c_V \leq E \leq \lambda_{+,t}, \frac{n^{\vartheta}}{n \eta}  \le \sqrt{\kappa+\eta} \leq 10 \right\} \\
&\bigcup \left\{z=E+\ii \eta:  \lambda_{+,t} \le E \le \lambda_{+,t} +\frac34c_V, n^{-2/3+\vartheta} \le \eta \le 10\right\},
\end{split}
\end{equation}
where recall that $\lambda_{+,t}$ is the right-edge of $\rho_{w,t}$.
The following theorem gives the local laws on the domain $\mathcal{D}_{\vartheta}$. 

\begin{theorem}[Theorem 2.7 of \cite{DY20202}]\label{thm_local}
Suppose $V=WW^\top$ is $\eta_*$-regular, and $t$ satisfies \eqref{restr_t}. For any constant $\vartheta>0$, the following estimates hold uniformly in $z \in \mathcal{D}_{\vartheta}$:
\begin{itemize}
\item for $E\le \lambda_{+,t}$, we have
\begin{equation}\label{averin}
|m_{1,t}(z)-m_{w,t}(z)|\prec \frac{1}{n\eta};
\end{equation}
\item for $E\ge \lambda_{+,t}$, we have
\begin{equation}\label{averout}
|m_{1,t}(z)-m_{w,t}(z)|\prec \frac{1}{n(\kappa+\eta)}+\frac{1}{(n\eta)^2\sqrt{\kappa+\eta}}.
\end{equation}
\end{itemize}
\end{theorem}

As a consequence of this theorem, we can obtain the following rigidity estimate for the eigenvalues $\lambda_1\ge \lambda_2\ge \cdots \ge \lambda_p$ of $\cal Q_{1,t}$ near the right edge $\lambda_{+,t}$. We define the quantiles of $\rho_{w,t}$ as in \eqref{gammaj000}:
\begin{equation}\label{gammaj}
\gamma_j:=\sup_{x}\left\{\int_{x}^{+\infty} \rho_{w,t}(x)\dd x > \frac{j-1}{p}\right\},\quad 1\le j \le p .
\end{equation}
\begin{lemma}\label{lem_correct_rigi}
Suppose the local laws \eqref{averin} and \eqref{averout} hold. Then, for any $j$ such that $\lambda_{+,t} -  c_V/2 < \gamma_j \le \lambda_{+,t}$, we have
\begin{equation}\label{rigidity}
\vert \lambda_j - \gamma_j \vert \prec  j^{-1/3}n^{-2/3} .
\end{equation}
\end{lemma}
\begin{proof}
The estimate \eqref{rigidity} follows from the local laws \eqref{averin} and \eqref{averout} combined with a standard argument using Helffer-Sj\"ostrand calculus. The details are already given in \cite{EKYY1,EYY,pillai2014}. 
\end{proof}

Then, we present the local laws for the case where $W$ already satisfies a local law. 
\begin{assumption}\label{assm regularW}
Suppose $m_{V}(z)\equiv m_{w,0}(z)$ satisfies the following estimates for any constant $\vartheta>0$:
$$|m_{w,0}(z)-m_c(z)| \prec \frac{1}{n\eta},$$
for $\lambda_+ - c_V \le E \le \lambda_+$ and $ n^\vartheta (n\eta)^{-1}\le \sqrt{|E-\lambda_+|+\eta} \le 10$; 
$$|m_{w,0}(z)-m_c(z)| \prec \frac{1}{n(|E-\lambda_+|+\eta)}+ \frac{1}{(n\eta)^2\sqrt{|E-\lambda_+|+\eta}},$$
for $\lambda_+ \le E \le \lambda_++c_V$ and $n^{-2/3+\vartheta}\le \eta\le 10$. Here $m_c(z)$ is the Stieltjes transform of a deterministic probability density $\rho_c(x)$ that is compactly supported on $[0,\lambda_+]$, and satisfies $\rho_c(x)\sim \sqrt{x}$ for $\lambda_+-c_V\le x\le \lambda_+.$
\end{assumption}
We denote the rectangular free convolution of $\rho_c$ with MP law at time $t$ by $\rho_{c,t}$, and its Stieltjes transform by $m_{c,t}$. We also denote the right edge of $\rho_{c,t}$ by $\lambda_{c,t}$ and define $\kappa_c:=|E-\lambda_{c,t}|$. 
Then we define the following spectral domain 
\begin{equation}\label{eq_domanthetas}
\begin{split}
\mathcal{D}_{\vartheta,c}:=&\left\{z=E+\ii \eta: \lambda_{c,t}- \frac34c_V \leq E \leq \lambda_{c,t},  \frac{n^{\vartheta}}{n \eta} \le \sqrt{\kappa_c+\eta} \le 10 \right\} \\
&\bigcup \left\{z=E+\ii \eta:  \lambda_{c,t} \le E \le \lambda_{c,t}+\frac34c_V,  n^{-2/3+\vartheta}\le \eta \le 10 \right\}.
\end{split}
\end{equation}
%
Then, we have the following local law on the domain $\mathcal{D}_{\vartheta,c}$. 
 \begin{theorem}[Theorem 2.10 of \cite{DY20202}]\label{thm_local2}
Suppose Assumption \ref{assm regularW} holds. For any fixed constants $\vartheta,\delta>0$, the following estimates hold uniformly in $z \in \mathcal{D}_{\vartheta,c}$ and $0\le t \le n^{-\delta}$:
\begin{itemize}
\item for $E\le \lambda_{c,t}$, we have
\begin{equation}\label{averin2}
|m_{1,t}(z)-m_{c,t}(z)|\prec \frac{1}{n\eta};
\end{equation}
\item for $E\ge \lambda_{c,t}$, we have
\begin{equation}\label{averout2}
|m_{1,t}(z)-m_{c,t}(z)|\prec \frac{1}{n(\kappa_c+\eta)}+\frac{1}{(n\eta)^2\sqrt{\kappa_c+\eta}}.
\end{equation}
\end{itemize}
\end{theorem}

Again using Theorem \ref{thm_local2}, we can prove the following rigidity estimate for the eigenvalues of $\cal Q_{1,t}$ near the right edge $\lambda_{c,t}$. We define the quantiles $\gamma_j^c$ as in \eqref{gammaj} but with $\rho_{w,t}$ replaced by $\rho_{c,t}$. 
\begin{lemma}\label{lem rigidity2}
Suppose the local laws \eqref{averin2} and \eqref{averout2} hold. Then, for any $j$ such that $\lambda_{c,t} -  c_V/2 < \gamma_j^c \le \lambda_{c,t}$, we have
\begin{equation}\label{rigidity2}
\vert \lambda_j - \gamma_j^c \vert \prec  j^{-1/3}n^{-2/3} .
\end{equation}
\end{lemma}
\begin{proof}
The estimate \eqref{rigidity2} follows from the local laws \eqref{averin2} and \eqref{averout2} combined with a standard argument using Helffer-Sj\"ostrand calculus. The details are already given in \cite{EKYY1,EYY,pillai2014}. 
\end{proof}

\begin{remark}\label{rmk_iogeqonecase}
We now briefly discuss how to handle the general case with $i_0>1$. When $i_0>1$, the $i_0-1$ outliers will give rise to several small peaks of $\rho_{w,t}$ around the spikes $d_i$, $1\le i \le i_0-1$. We can exclude them and only consider the bulk component of $\rho_{w,t}$ with a right edge $\lambda_{+,t}$ that is close to $d_{i_0}$. Then, all the results in this section still hold for the $i_0>1$ case with $\lambda_+:=d_{i_0}$ except that $c_V$ needs to be chosen sufficiently small so that the spectral domains $\mathcal{D}_{\vartheta}$ and $\mathcal{D}_{\vartheta,c}$ are away from the spikes $d_i$, $1\le i \le i_0-1$, by a distance of order 1 and $j$ will be restricted to $j\ge i_0$ in Lemmas \ref{lem_correct_rigi} and \ref{lem rigidity2}. 
\end{remark}

\section{Proof of Theorem \ref{thm_regularbm}}
\label{sec_calculationDBM}

This section is devoted to the proof of Theorem \ref{thm_regularbm}. {For simplicity of presentation, we only provide the detailed proof for the $i_0=1$ case without outliers. The general case with $i_0>1$ will be discussed in Remark \ref{rmk_commentcomment} below.}

In the proof, we fix two time scales 
\begin{equation}\label{defn_t0t1}
t_0= {n^{\omega_0}}/{n^{1/3}}, \quad t_1= {n^{\omega_1}}/{n^{1/3}},
\end{equation}
for some constants $\omega_0$ and $\omega_1$ satisfying $1/3-\phi_*/2+\e/2 \le \omega_0 \le 1/3-\e/2$ and $0<\omega_1<\omega_0/100$. The reason for choosing these two scales is the same as the one in \cite{edgedbm}. That is, we first run the DBM for $t_0$ amount of time to regularize the global eigenvalue density, and then for the DBM from $t_0$ to $t_0+t_1$, we will show that the local statistics of the edge eigenvalues converge to the Tracy-Widom law. Since $t_1\ll t_0$, for the time period $t_0 \le t \le t_0+t_1$ the locations of the quantiles defined in \eqref{gammaj} remain approximately constant. 

The eigenvalue dynamics of $\cal Q_t=(W+\sqrt{t}X)(W+\sqrt{t}X)^\top$ with respect to $t$ is described by the \emph{rectangular Dyson Brownian motion} defined as follows. Let $B_i(t),$ $i=1,\cdots,p,$ be independent standard Brownian motions. For $t\ge 0$, we define the process $\{\lambda_i(t):1\le i \le p\}$ as the unique strong solution to the following system of SDEs \cite[Appendix C]{Bourgade2017}:
\begin{equation}\label{SDE rDBM}
d \lambda_i=2 \sqrt{\lambda_i}\frac{dB_i}{\sqrt{n}}+\left( \frac{1}{n} \sum_{j \neq i} \frac{\lambda_i+\lambda_j}{\lambda_i-\lambda_j}+1\right)dt, \quad 1\le i \le p,
\end{equation}
with initial data 
$$\lambda_i(0):=\lambda_i(\gamma_w \mathcal{Q}_{t_0}), \quad \gamma_w:=\left( \frac12 [4\lambda_{+,t_0}\zeta_{+,t_0} + (1-c_n)^2 t_0^2] c_n^2 t_0^2\Phi_{t_0}^{''}(\zeta_{+,t_0}) \right)^{-1/3}.$$ 
In other words, the initial data is chosen as the eigenvalues of the regularized matrix $\cal Q_{t_0}$, and $\gamma_w$ is chosen to match the edge eigenvalue gaps of $\cal Q_{t_0}$ with those of the Wigner matrices. Here we recall that the asymptotic density $\rho_{w,t}$ is given by \eqref{sqrtdensity2}, while the Wigner semicircle law has density $\pi^{-1}\sqrt{(2-x)_+} +\OO((2-x)_+)$ around $2$. The system of SDEs \eqref{SDE rDBM} for the rectangular DBM is defined in a way such that for any time $t>0$, the process $\{\lambda_i(t)\}$ has the same joint distribution as the eigenvalues of the matrix 
$$\gamma_w \cal Q_{t_0+t/\gamma_w}= (\sqrt{\gamma_w} W+\sqrt{\gamma_w t_0+t}X)(\sqrt{\gamma_w} W+\sqrt{\gamma_w t_0+t}X)^\top.$$ 
We shall denote the rectangular free convolution of the empirical spectral density of $\sqrt{\gamma_w}V$ with MP law at time $\gamma_w t_0 + t$ by $\rho_{\lambda,t}$, which gives the asymptotic ESD for $\gamma_w \cal Q_{t_0+t/\gamma_w}$. Moreover, we use $m_{\lambda,t}$ to denote the Stieltjes transform of $\rho_{\lambda,t}$. It is easy to see that the right  edge of $\rho_{\lambda,t}$ is given by
\begin{equation*}
E_{\lambda}(t):=\gamma_w \lambda_{+, t_0+t/\gamma_w},
\end{equation*}
where recall that $\lambda_{+,t}$ denotes the right edge of $\rho_{w,t}$ at time $t$. Note that the scaling factor $\gamma_w$ is fixed throughout the evolution, but the right edge evolves in time. 

We would like to compare the edge eigenvalue statistics of the rectangular DBM $\{\lambda_i(t)\}$ with those of a carefully chosen deformed Wishart matrix. We define a $p \times p$ sample covariance matrix $\cal U \cal U^\top$, where $\cal U$ is a random matrix of the form $\cal U:=\Sigma^{1/2}\cal X$. {Here $\cal X$ is a $p\times n$ random matrix with i.i.d. Gaussian entries of mean zero and variance $n^{-1}$, and $\Sigma=\diag(\sigma_1,\cdots, \sigma_p)$ is a diagonal population covariance matrix. }
Recall that the asymptotic ESD of $\cal U\cal U^\top$, denoted as $\rho_{\mu,0}$, is given by the multiplicative free convolution of the Marchenko-Pastur law and the ESD of $\Sigma$, which is also referred to as the deformed Marchenko-Pastur law \cite{MP}. We choose $\Sigma$ such that $\rho_{\mu,0}$ matches $\rho_{\lambda,0}$ near the right edge $E_{\lambda}(0)$, that is, $\rho_{\mu,0}(x)$ satisfies that
\begin{equation}\label{same rho}\rho_{\mu,0}(x)=\pi^{-1}\sqrt{(E_{\lambda}(0)-x)_+}+\OO((E_{\lambda}(0)-x)_+),\end{equation}
for $x$ around $E_{\lambda}(0)$. Note that there are only two parameters to match, i.e. the right spectral edge and the curvature of the spectral density at the right edge, but there are a lot of degrees of freedom in $\Sigma$ for tuning to ensure that \eqref{same rho} holds. Now we define a rectangular DBM with initial data $\{\mu_i\}$ being the eigenvalues of $\cal U \cal U^\top$. More precisely, for $t\ge 0$ we define the process $\{\mu_i(t):1\le i \le p\}$ as the unique strong solution to the following system of SDEs:
\begin{equation}\label{eq_coupledsde}
d \mu_i=2 \sqrt{\mu_i}\frac{dB_i}{\sqrt{n}}+\left( \frac{1}{n} \sum_{j \neq i} \frac{\mu_i+\mu_j}{\mu_i-\mu_j}+1\right)dt, \quad 1\le i \le p,
\end{equation}
with initial data $\mu_i(0):=\mu_i( \mathcal{U}\cal U^\top ).$ 
For any $t>0$ the process $\{\mu_i(t)\}$ has the same joint distribution as the eigenvalues of the matrix $(\cal U+\sqrt{t}X)(\cal U+\sqrt{t}X)^\top$, which is still a sample covariance matrix with population covariance $\Sigma + tI$. In particular, {by \cite{LS}} we know that the edge eigenvalues of $\{\mu_i(t)\}$ obey the Tracy-Widom distribution asymptotically. We will denote the rectangular free convolution of $\rho_{\mu,0}$ with MP law at time $ t$ by $\rho_{\mu,t}$, which gives the asymptotic ESD for $(\cal U+\sqrt{t}X)(\cal U+\sqrt{t}X)^\top$. Furthermore, we denote the Stieltjes transform of $\rho_{\mu,t}$ by $m_{\mu,t}$, and the right edge of $\rho_{\mu,t}$ by $E_{\mu}(t)$. Note that we have $E_\mu(0)=E_\lambda(0)$ by \eqref{same rho}.



The main result of this section is the following comparison theorem. 
\begin{theorem}\label{thm_firstkey} 
Fix any integer $k \in \N$. Under the assumptions of Theorem \ref{thm_regularbm}, there exists a constant $\e>0$ such that 
\begin{equation}\label{comparison_eig}
\max_{1\le i \le k}\left| [\lambda_{i}(t_1)-E_{\lambda}(t_1)]-[\mu_i(t_1)-E_{\mu}(t_1)] \right| \leq  n^{-2/3-\e} \quad \text{with high probability.}
\end{equation}
\end{theorem}
With Theorem \ref{thm_firstkey}, we can conclude Theorem \ref{thm_regularbm}. 
\begin{proof}[Proof of Theorem \ref{thm_regularbm}]
We take $t_0=t-t_1$ for a small enough constant $\omega_1$. Then, together with the fact that $\mu_i(t_1)-E_{\mu}(t_1)$ satisfies the Tracy-Widom fluctuation by \cite{bao2015,DY,elkaroui2007,LS,Regularity4}, the estimate\eqref{comparison_eig} implies \eqref{eq_maintheoremeq1}. 
\end{proof}

\begin{remark}\label{rmk_commentcomment}
We now make some remarks about the general case with $i_0>1$.
Its proof is almost the same as that for the $i_0=1$ case, except that we need to apply some standard arguments in the study of DBM regarding the reindexing of the eigenvalues and the padding with dummy particles. More precisely, in equation \eqref{comparison_eig}, we should control
$$\left| [\lambda_{i+i_0-1}(t_1)-E_{\lambda}(t_1)]-[\mu_i(t_1)-E_{\mu}(t_1)] \right| \leq  n^{-2/3-\e}.$$
Then, in defining the two rectangular DBMs, we add to the initial data of the SDEs some dummy particles, which are away from the edge eigenvalues by a distance of order $N^C$ for a large constant $C>0$. These dummy particles have a negligible effect on the evolution of edge eigenvalues, and hence are irrelevant to our final results. But they allow us to take the difference $\lambda_{i+i_0-1}-\mu_i$ for all $1\le i \le p$. We refer the reader to equations (3.10)-(3.12) of \cite{edgedbm} for more details.
\end{remark}

%
%
%
%

\subsection{Interpolating processes}

To estimate the difference $\lambda_i(t)-\mu_i(t)$, we study the following interpolating processes for $0 \leq \alpha \leq 1$:
\begin{equation}\label{eq_interpolationsde}
d z_i(t, \alpha)=2 \sqrt{z_i(t, \alpha)}\frac{\dd B_i}{\sqrt{n}}+\left(\frac{1}{n} \sum_{j \neq i} \frac{z_i(t,\alpha)+z_j(t,\alpha)}{z_i(t,\alpha)-z_j(t,\alpha)}+1\right) \dd t, \quad 1\le i \le p,
\end{equation}
with the interpolated initial data $z_i(0,\alpha):=\alpha \lambda_i(0)+(1-\alpha) \mu_i(0).$ Correspondingly, we denote the Stieltjes transform of the ESD of $\{z_i(t,\al)\}$ by
\begin{equation}\label{eq_defnzstiel}
\wt m_{t}(z,\al):=\frac{1}{p} \sum_{i=1}^p \frac{1}{z_i(t,\alpha)-z}. 
\end{equation}
Note that by Lemma \ref{lem_asymdensitysquare}, due to the choice of  $\gamma_w$ and \eqref{same rho}, we have that
\begin{equation}\label{eq_edgecloseinitial}
\rho_{\lambda,0}(E_{\lambda}(0)-E)=\rho_{\mu,0}(E_{\mu}(0)-E)\left[1+\OO\left( \frac{|E|}{t_0^2}\right)\right], \quad 0 \leq E \leq \tau t_0^2,
\end{equation}
for a sufficiently small constant $\tau>0.$ Let $\gamma_{\mu,i}(t)$ and $\gamma_{\lambda,i}(t)$ be the quantiles of $\rho_{\mu,t}$ and $\rho_{\lambda,t}$ defined as
\begin{equation}\label{eq_definitionquantile}
\begin{split}
\gamma_{\mu,i}(t):=\sup_{x}\left\{\int_{x}^{+\infty} \rho_{\mu,t}(x)\dd x > \frac{i-1}{p}\right\},\quad
 \gamma_{\lambda,i}(t):=\sup_{x}\left\{\int_{x}^{+\infty} \rho_{\lambda,t}(x)\dd x > \frac{i-1}{p}\right\}.
 \end{split}
\end{equation} 
By Theorem \ref{thm_local} and \cite[Theorem 3.2]{bao2015}, both $|\wt m_{0}(z,0)-m_{\mu,0}(z)|$ and $|\wt m_{0}(z,1)-m_{\lambda,0}(z)|$ satisfy local laws as in \eqref{averin} and \eqref{averout}. 
Hence, by {Lemma \ref{lem_correct_rigi}}, there exists a small enough constant $c_0>0$  depending on $c_V$ such that for $k_0:=\lceil c_0 n\rceil$,
\begin{equation}\label{eq_rigidityalphaequal1}
\sup_{0 \leq t \leq 10t_1}\left( |z_i(t,0)-\gamma_{\mu,i}(t)| +  |z_i(t,1)-\gamma_{\lambda,i}(t)|\right) \prec i^{-1/3} n^{-2/3},\quad 1\le i \le k_0.
\end{equation}
Here to get \eqref{eq_rigidityalphaequal1}, we used a standard stochastic continuity argument to pass from fixed times $t$ to all times. Roughly speaking, taking a sequence of fixed times $t_k= 10t_1 \cdot k/n^{C}$ for a large constant $C>0$, by {Lemma \ref{lem_correct_rigi}} and a simple union bound we get that 
\begin{equation}\label{eq_rigidityalphaequal1union}
\sup_{0\le k \le n^C}\left( |z_i(t_k,0)-\gamma_{\mu,i}(t_k)| +  |z_i(t_k,1)-\gamma_{\lambda,i}(t_k)|\right) \prec i^{-1/3} n^{-2/3}.
\end{equation}
Then, we can show that with high probability, the difference $|z_i(t,0)-z_i(t_k,0)| +  |z_i(t,1)-z_i(t_k,1)|$ is small enough for all $t_k \le t \le t_{k+1}$ using a simple continuity estimate. We refer the reader to Appendix B of \cite{dysonbulk} for more details.

Combining \eqref{eq_edgecloseinitial} and \eqref{eq_definitionquantile}, we can get the following simple control on the quantiles near the edge. 

\begin{lemma} For $i =\OO( n^{6 \omega_0/5})$, we have that 
\begin{equation}\label{eq_edgecontrolcontrol}
\left| \gamma_{\mu,i}(0) - \gamma_{\lambda,i}(0)) \right| \lesssim \frac{i^{4/3}}{n^{2 \omega_0} n^{2/3}}.
\end{equation} 
\end{lemma} 
\begin{proof}
For simplicity, we denote $x:=E_{\mu}(0)-\gamma_{\mu,i}(0)$ and $y:=E_{\lambda}(0)-\gamma_{\lambda,i}(0).$ Without loss of generality, we assume that $x \leq y.$ Note that by the square root behaviors of $\rho_{\mu,0}$ and $\rho_{\lambda,0}$ near the right edges, it is easy to get that $x \sim y \sim i^{2/3} n^{-2/3}$ for $i\ge 2$. Now using (\ref{eq_edgecloseinitial}) and (\ref{eq_definitionquantile}), we obtain that
\begin{align*}
\int_{0}^{x} \left[\rho_{\mu,0}(E_{\mu}(0)-E)-\rho_{\lambda,0}(E_\lambda(0)-E)\right] \dd E =\int^{y}_{x} \rho_{\lambda,0}(E_\lambda(0)-E) \dd E ,
\end{align*}
which gives $|y^{3/2}-x^{3/2}| \lesssim {x^{5/2}}/{t_0^2}.$ From this estimate, we get that $|y-x|\lesssim x^2/t_0^2$, which concludes the proof together with the facts $x \sim i^{2/3}n^{-2/3}$ and $E_{\lambda}(0)=E_{\mu}(0)$. 
\end{proof}

Next, we will construct a collection of measures that match the asymptotic densities of the interpolating ensembles and have well-behaved square root densities near the right edge. Our main goal is that for each $0\le \al \le 1$, we have a density which matches the distribution of $\{z_i(0,\al)\}$ approximately, and with which we can take a rectangular free convolution for any $0\le t \le t_1$. 

At $t=0$, define the eigenvalue counting functions near the edge $E_{\mu}(0)=E_{\lambda}(0)$ as  
\begin{equation*}
n_\mu(E)=\int_E^{E_{\mu}(0)} \rho_{\mu,0}(y) \dd y, \quad n_{\lambda}(E)=\int_E^{E_{\lambda}(0)} \rho_{\lambda,0}(y) \dd y.
\end{equation*}
Since $ \rho_{\mu,0}(y) >0$ for $E_{\mu}(0) - \tau\le y < E_{\mu}(0)$ and $ \rho_{\lambda,0}(y) >0$ for $E_\lambda(0) - \tau\le y < E_\lambda(0)$ for a small enough constant $\tau>0$, the functions $n_\mu$ and $n_\lambda$ are strictly increasing near the right edges. Hence, we can define the inverse functions (i.e. the continuous versions of quantiles) $\varphi_\mu(s)$ and $\varphi_\lambda(s)$ by the equations
\begin{equation*}
n_{\mu}(\varphi_\mu(s))=s, \quad n_{\lambda}(\varphi_{\lambda}(s))=s, \quad 0\le s \le c_*, 
\end{equation*}
where {$c_*\equiv c_*(n):=\lceil c_0 n\rceil/n$ for a small enough constant $c_0>0$}. Then, for $\alpha \in [0,1],$ we define 
\begin{equation*}
\varphi(s,\alpha):=\alpha \varphi_\mu(s)+(1-\alpha) \varphi_\lambda(s),
\end{equation*}
which maps $[0,c_*]$ onto 
\begin{equation}\label{mathttDal}
\mathtt{D}_{\alpha}:= [\alpha \varphi_\mu(c_*)+(1-\alpha)\varphi_{\lambda}(c_*), E_{\lambda}(0)].
\end{equation} 
Now, for any $\al\in [0,1]$, we define the inverse function $n(E, \alpha):\mathtt{D}_{\alpha} \rightarrow [0, c_*]$ of $\varphi(s,\alpha)$ by the equation
\begin{equation*}
n(\varphi(s,\alpha),\alpha)=s.
\end{equation*}
Then, we define the asymptotic density as
\begin{equation*}
\rho(E, \alpha):=\frac{\partial n(E, \alpha)}{\partial E}. 
\end{equation*} 
By inverse function theorem, we can calculate that
\begin{equation*}
\rho(E, \alpha)= \left[ \frac{\alpha}{\rho_{\mu,0}(\varphi_\mu(n(E,\alpha)))}+\frac{1-\alpha}{\rho_{\lambda,0}(\varphi_\lambda(n(E,\alpha)))} \right]^{-1}.
\end{equation*}
Combining it with (\ref{eq_edgecloseinitial}), we immediately find that 
\begin{equation}\label{eq_constructdensityclose}
\rho(E_{+}(0,\alpha)-E, \alpha)=\rho_{\mu,0}(E_{\mu}(0)-E)\left[1+\OO\left( \frac{|E|}{t_0^2}\right)\right], \quad 0 \leq E \leq \tau t_0^2, 
\end{equation} 
for a sufficiently small constant $\tau>0$, where $E_{+}(0,\alpha)$ is the right edge of $\rho(E, \alpha)$. We now construct a (random) measure $\mu(E, \alpha)$ as 
\begin{equation*}
\dd \mu(E, \alpha)=\rho(E, \alpha) \mathbf{1}_{\{E \in \mathtt{D}_{\alpha}\}} \dd E+p^{-1} \sum_{i>c_* n} \delta_{z_i(0,\alpha)}(\dd E). 
\end{equation*}  
This measure is defined in a way such that its Stieltjes transform is close to $\wt m_{0}(z,\al)$ in (\ref{eq_defnzstiel}). Moreover, the motivation behind this definition is as follows. We need a deterministic density that behaves well around the right edge in order to use the results in Section \ref{sec_preliminary_feelocal}. But we do not have any estimate on the density far away from the edge. Hence for the remaining eigenvalues that are away from the right edge by a distance of order 1, we just take $\delta$ functions. Although the sum of delta measures is random, its effect on deterministic quantities that we are interested in is negligible.

Let $\rho_t(E, \alpha)$ be the rectangular free convolution of $\dd \mu(E, \alpha)$ with the MP law at time $t$. Moreover, we denote its Stieltjes transform by $m_t(z, \alpha)$ and its right edge by $E_+(t, \alpha).$ 
Some key properties of $\rho_t(E, \alpha)$ and $m_t(z, \alpha)$ have been given in Section \ref{sec_preliminary_feelocal}. In particular, we know that $\rho_t(E, \alpha)$ has a square root behavior near $E_+(t, \alpha).$  Although $\rho_t(E,\alpha)$ is random, with the results in Section \ref{sec_preliminary_feelocal} we can provide a deterministic control on it. 
\begin{lemma}\label{coro_shoruse}
 Let $\e, \tau>0$ be sufficiently small constants. For $0 \leq E \leq \tau n^{-2\epsilon} t_0^2$, we have that for any constant $D>0$,
\begin{equation}\label{eq_corcontrol1}
\rho_t(E_+(t,\alpha)-E,\al)=\rho_{\mu,t}(E_{\mu}(t)-E)\left[1+\OO(n^{\epsilon}t/t_0 + n^{-D})\right].
\end{equation}
Moreover, for a small constant $c_\tau>0$ we have that
\begin{equation}\label{eq_corcontrol4}
\max_{1\le i \le c_\tau n^{1-3\e}t_0^3}|\wt{\gamma}_i(t,\alpha)-\wt{\gamma}_i(t,0)| \leq \left(n^{\epsilon}\frac{t}{t_0}+n^{-D}\right) \frac{i^{2/3}}{n^{2/3}},
\end{equation}
where we introduced the short-hand notation $\wt{\gamma}_i(t,\alpha):={E_+(t, \alpha)}-{\gamma_i(t,\alpha)}.$
\end{lemma}
\begin{proof}
The estimates (\ref{eq_corcontrol1}) follows directly from \eqref{eq_densitycomparison}. The estimate (\ref{eq_corcontrol4}) follows from (\ref{eq_corcontrol1}) using the same argument as in the proof of \eqref{eq_edgecontrolcontrol}. 
\end{proof}

With the eigenvalues rigidity \eqref{eq_rigidityalphaequal1} and the construction of $\dd \mu(E, \alpha)$, we can verify that $|m_{0}(z,\al)-\wt m_{0}(z,\al)|$ satisfies Assumption \ref{assm regularW}. Then, by {Lemma \ref{lem rigidity2}}, we have the following rigidity estimate of $\{z_i(t,\al)\}$. As before, we define the quantiles $\gamma_i(t,\alpha)$ by
\begin{equation*}
\gamma_i(t,\alpha):=\sup_{x}\left\{\int_{x}^{+\infty} \rho_{t}(E,\al)\dd E > \frac{i-1}{p}\right\}.
\end{equation*}   
\begin{lemma}\label{lem_rigidty z} 
There exists a constant $c_* >0$ so that 
\begin{equation}\label{rigiditywtz}
\sup_{0 \leq \alpha \leq 1} \sup_{0 \leq t \leq 10 t_1}\left| z_i(t,\alpha)-\gamma_i(t,\alpha) \right| \prec i^{-1/3}n^{-2/3} ,\quad 1\le i \le c_* n.
\end{equation}
\end{lemma}
\begin{proof}
This estimate follows from {Lemma \ref{lem rigidity2}} combined with a standard stochastic continuity argument in $t$. 
\end{proof}

Using \eqref{eq_defnpsiderviative0} and \eqref{eq_imasymptoics}, we can calculate that
\begin{align} 
\frac{\dd\sqrt{E_+(t,\al)}}{\dd t}&=\left[\frac{1-c_n}{2\zeta_t(E_+(t,\al),\al)}-\frac{c_n m_{t}(E_+(t,\al),\al)}{b_t(E_+(t,\al),\al)} \right]\sqrt{\zeta_t(E_+(t,\al),\al) }  - \frac{(1-c_n)^2t}{4 \zeta_t(E_+(t,\al),\al)\sqrt{E_+(t,\al)}} +\OO(t^2) \nonumber\\
& = \frac{1-c_n}{2\sqrt{\zeta_t(E_+(t,\al),\al)}}- c_n m_{t}(E_+(t,\al),\al) \sqrt{E_+(t,\al)- t(1-c_n)}   - \frac{(1-c_n)^2t}{4 [E_+(t,\al)]^{3/2} } +\OO(t^2) ,\label{eq_deriofsing}
\end{align}
where we used the notations
$$b_t(z,\al):=1+c_ntm_t(z,\al), \quad\zeta_t(z,\al):=zb^2_t(z,\al)-  t(1-c_n)b_t(z,\al) .$$ 
In the proof, we will also need to use the following function defined for $ E\in [-\tau, \tau]$ for a small enough constant $\tau>0$:
\begin{equation}
\begin{split}
 \Psi_t(E,\al)&:= \frac{1-c_n}{2\sqrt{\zeta_t(E_+(t,\al),\al)}}-\frac{1-c_n}{2\sqrt{ E_+(t,\al)-E}}  - \frac{(1-c_n)^2t}{4  [E_+(t,\al)]^{3/2} } \\
&-\re \left[c_n m_{t}(E_+(t,\al),\al) \sqrt{E_+(t,\al)-t(1-c_n)}-c_n m_{t}(E_+(t,\al)-E,\al) \sqrt{E_+(t,\al)-E}\right]  . \label{eq_defnpsiderviative2}
\end{split}
\end{equation}
 Next, we prove some matching estimates for the function $\Psi_t(E,\al)$ in Lemma \ref{coro_shoruse2}. The proof of this lemma explores a rather delicate cancellation in $\Psi_t(E,\al)$. 


\begin{lemma}\label{coro_shoruse2}
 Let $\e,\tau>0$ be sufficiently small constants. For $0 \leq E \leq \tau n^{-2\epsilon} t_0^2$, we have that for any constant $D>0$,
\begin{equation}\label{eq_corcontrol2}
|\Psi_t(E,\alpha)-\Psi_t(E,0)| \lesssim \left( \frac{n^{\epsilon}}{t_0}+\frac{n^{-D}}{t}\right)E + t^2.
\end{equation}
For $0\le E \le \tau n^{-2\epsilon}t t_0,$ we have that for any constant $D>0$,
\begin{equation}\label{eq_corcontrol3}
|\Psi_t(-E,\alpha)-\Psi_t(-E,0)|\lesssim \left( n^{\epsilon}\frac{t^{1/2}}{t_0^{1/2}} + n^{-D}\frac{t_0^{1/2}}{t^{1/2} }\right)E^{1/2} + t^2 .
\end{equation}
\end{lemma}
\begin{proof}
First, we claim that
\begin{equation}\label{nontrivialcancel} 
\Psi_t(E,\alpha) = \wt\Psi_t(E,\alpha)+\OO(t^2),
\end{equation}
where $\wt\Psi_t(E,\al)$ is defined by
\begin{equation}\nonumber
\begin{split}
\wt\Psi_t(E,\al)&:= \frac{1-c_n}{2\sqrt{E_+(t,\al)}}-\frac{1-c_n}{2\sqrt{ E_+(t,\al)-E}} \\
&-\re \left[c_n m_{t}(E_+(t,\al),\al) \sqrt{E_+(t,\al)}-c_n m_{t}(E_+(t,\al)-E,\al) \sqrt{E_+(t,\al)-E}\right]  .
\end{split}
\end{equation}
In fact, subtracting $\wt\Psi_t(x,\al)$ from $\Psi_t(x,\al)$ and using the definition of $\zeta_t(E_+(t,\al),\al)$, we get that
\begin{align*}
&\Psi_t(E,\alpha) - \wt\Psi_t(E,\alpha) =\frac{1-c_n}{2\sqrt{\zeta_t(E_+(t,\al),\al)}}-\frac{1-c_n}{2\sqrt{ E_+(t,\al)}} +\frac{(1-c_n)c_n t m_{t}(E_+(t,\al),\al)}{\sqrt{E_+(t,\al)}+\sqrt{E_+(t,\al)-t(1-c_n)}}  - \frac{(1-c_n)^2t}{4 [ E_+(t,\al)]^{3/2} } \\
&=\frac{(1-c_n) [E_+(t,\al) - b_t^2(E_+(t,\al),\al) \cdot E_+(t,\al) + (1-c_n)tb_t(E_+(t,\al),\al)]}{2\sqrt{\zeta_t(E_+(t,\al),\al)}\sqrt{ E_+(t,\al)} \left( \sqrt{ E_+(t,\al)}+\sqrt{\zeta_t(E_+(t,\al),\al)} \right)} +\frac{(1-c_n)c_n t m_{t}(E_+(t,\al),\al)}{2\sqrt{E_+(t,\al)} } \\
&\quad - \frac{(1-c_n)^2t}{4 [ E_+(t,\al)]^{3/2} } +\OO(t^2) \\
&=\frac{-2(1-c_n)c_n t m_{t}(E_+(t,\al),\al) \cdot E_+(t,\al)+ (1-c_n)^2t}{4 [E_+(t,\al)]^{3/2} }+\frac{(1-c_n)c_n t m_{t}(E_+(t,\al),\al)}{2\sqrt{E_+(t,\al)}}  - \frac{(1-c_n)^2t}{4  [E_+(t,\al)]^{3/2} } +\OO(t^2) \\
& =\OO(t^2).
\end{align*}
On the other hand, using \eqref{eq_edgecomparebound2} and the fact that $E_+(0,\al)=E_\lambda(0)$ for $0\le \al \le 1$, we get that  
\begin{equation}\label{E+al1}|E_+(t,\al)-E_\lambda(0)|=\OO(t) ,\quad 0\le \al \le 1.\end{equation}
Hence, we can estimate that 
\begin{align*}
 &\wt\Psi_t(E,\alpha) -  \wt\Psi_t(E,0) \\
 &=-c_n \re \left[m_{t}(E_+(t,\al),\al)-m_{t}(E_+(t,\al)-E,\al) \right] \sqrt{E_+(t,\al)} \\
 &\quad +c_n \re \left[m_{t}(E_+(t,0),0)-m_{t}(E_+(t,0)-E,0) \right] \sqrt{E_+(t,0)}+ \OO(E)\\
 &=c_n \re \left[(m_{t}(E_+(t,0),0)-m_{t}(E_+(t,0)-E,0))-(m_{t}(E_+(t,\al),\al)-m_{t}(E_+(t,\al)-E,\al)) \right] \sqrt{E_+(t,\al)} \\
 &\quad +\OO \left(t\left|m_{t}(E_+(t,0),0)-m_{t}(E_+(t,0)-E,0)\right|+E\right).
\end{align*}
By \eqref{eq_bound3}, we have that
$$t\left|m_{t}(E_+(t,0),0)-m_{t}(E_+(t,0)-E,0)\right| \lesssim (Et^{-1})\cdot t=E,\quad E\ge 0,$$
and by \eqref{eq_bound1} and \eqref{eq_bound2}, we have that
$$t\left|m_{t}(E_+(t,0),0)-m_{t}(E_+(t,0)-E,0)\right|  \lesssim  t\sqrt{|E|} \le t^2 +|E|,\quad E\le 0.$$
Using these two estimates and Lemma \ref{lem comparison1}, we can bound $\wt\Psi_t(E,\alpha) -  \wt\Psi_t(E,0)$ and conclude \eqref{eq_corcontrol2} and \eqref{eq_corcontrol3}.
\end{proof}

\begin{remark}\label{rem restart}
 Later we will only consider the dynamics after $t=n^{-C}$ for some large constant $C>0$, so that the $n^{-D}$ terms in \eqref{eq_corcontrol2} and \eqref{eq_corcontrol3} are negligible as long as $D$ is large enough. 
\end{remark}

Note that the interpolating measures $\dd \mu(E,0)$ (resp. $\dd \mu(E,1)$) only matches the asymptotic measure $\rho_{\mu,0}(E)\dd E$ (resp. $\rho_{\lambda,0}(E)\dd E$) for $E\in \mathtt D_0$ (resp. $E\in \mathtt D_1$). For the random part, we control its effect using the local laws. 
With the eigenvalues rigidity \eqref{eq_rigidityalphaequal1}, we can check that $|m_0(z,0)-\wt m_{0}(z,0)|$ and $|m_0(z,1)-\wt m_1(z,1)|$ satisfy the two estimates in Assumption \ref{assm regularW}. (Recall that { $\widetilde{m}_0(z,\al)$} was defined in \eqref{eq_defnzstiel} and $m_0(z,\al)$ is the Stieltjes transform of $\dd\mu(E,\al)$.) Moreover by Theorem \ref{thm_local} and \cite[Theorem 3.2]{bao2015}, we also have that $|\wt m_{0}(z,0)-m_{\mu,0}(z)|$ and $|\wt m_{0}(z,1)-m_{\lambda,0}(z)|$ satisfy the two estimates in Assumption \ref{assm regularW}. Hence we have that  
\begin{equation}\label{matchingms}
|m_{0}(z,0)-m_{\mu,0}(z)|\prec \frac{1}{n\eta},\quad |E-E_{\mu}(0)|\le \frac34c_V,\   n^{-2/3+\vartheta}\le \eta \le 10,
\end{equation}
and 
\begin{equation}\label{matchingms2}
|m_{0}(z,1)-m_{\lambda,0}(z)|\prec \frac{1}{n\eta},\quad |E-E_\lambda(0)|\le \frac34c_V,\  n^{-2/3+\vartheta}\le \eta \le 10.
\end{equation}
With the above two estimates, we can control $\left| E_+(t,1)-E_\lambda(t) \right|$ and $\left| E_+(t,0)-E_{\mu}(t) \right|$ for $0\le t \le 10t_1$.
\begin{lemma}
We have that  
\begin{equation}\label{E+1}
\max_{0\le t \le 10t_1}\left| E_+(t,1)-E_\lambda(t) \right| \prec t^3+n^{-1/2}t ,
\end{equation}
and 
\begin{equation}\label{E+2}
\max_{0\le t \le 10t_1}\left| E_+(t,0)-E_{\mu}(t) \right| \prec t^3+n^{-1/2}t  .
\end{equation}
\end{lemma}
\begin{proof}
Repeating the proof of 
Lemma \ref{lem edgemoving} (as given in Lemma A.2 of \cite{DY20202}) with $t_0$ replaced by 1, we can get that 
$$ |\zeta_{+,1} - \zeta_{+,\lambda} | \le Ct^3,$$
where we abbreviate $\zeta_{+,1}\equiv\zeta_t(E_+(t,1),1)$ and $\zeta_{+,\lambda}\equiv \zeta_{\lambda, t}(E_\lambda(t))$. Then with equation \eqref{simplePhizeta},  we get that
\begin{equation}\label{addmatchingeq} \left| E_+(t,1)-E_\lambda(t) \right|  \lesssim |\zeta_{+,1}-\zeta_{+,\lambda}| + t|m_0(\zeta_{+,1},1)-m_{\lambda,0}(\zeta_{+,\lambda})|. \end{equation}
Recall that $\zeta_{+,1}-E_\lambda(0)\sim t^2$ and $\zeta_{+,\lambda}-E_\lambda(0)\sim t^2$ by \eqref{eq_edgebound}. Then using \eqref{eq_bound1} and \eqref{eq_bound2}, we get that
$ |m_{\lambda,0}'(\zeta)|\lesssim t^{-1}$ for $\zeta$ between $\zeta_{+,1}$ and $\zeta_{+,\lambda}$. Thus, we can bound \eqref{addmatchingeq} as
\begin{align} \left| E_+(t,1)-E_\lambda(t) \right| & \lesssim t^3 + t|m_0(\zeta_{+,1},1)-m_{\lambda,0}(\zeta_{+,1})|+t|m_{\lambda,0}(\zeta_{+,1})-m_{\lambda,0}(\zeta_{+,\lambda})| \nonumber\\
&\lesssim t^3 + t|m_0(\zeta_{+,1},1)-m_{\lambda,0}(\zeta_{+,1})|. \label{E+Elambda}
\end{align}
For the second part, since $\dd\mu(E,1)$ matches $\rho_{\lambda,0}(E)$ for $E\in \mathtt D_1$, we can bound that
\begin{align*}
&\left| \left[m_0(\zeta_{+,1},1)-m_{\lambda,0}(\zeta_{+,1})\right] - [m_0(\zeta_{+,1}+\ii n^{-1/2},1)-m_{\lambda,0}(\zeta_{+,1}+\ii n^{-1/2})]\right| \\
&\le \sum_{i>c_* n}\frac{n^{-1/2}}{|z_i(0,1) - \zeta_{+,1}||z_i(0,1) - \zeta_{+,1}-\ii n^{-1/2}|} \lesssim n^{-1/2}
\end{align*}
with high probability. On the other hand, we have that
$$ \left|m_0(\zeta_{+,1}+\ii n^{-1/2},1)-m_{\lambda,0}(\zeta_{+,1}+\ii n^{-1/2})\right|\prec n^{-1/2}$$ 
by \eqref{matchingms2}. Combining the above two estimates, we obtain that
$$ \left|m_0(\zeta_{+,1},1)-m_{\lambda,0}(\zeta_{+,1})\right|\prec n^{-1/2}.$$
Plugging it into \eqref{E+Elambda}, we conclude \eqref{E+1}. The estimate \eqref{E+2} can be proved in the same way.
\end{proof}

In later proof, we will also need to study the evolution of the singular values $y_i(t,\alpha):=\sqrt{z_i(t,\alpha)}$. It is easy to see that the asymptotic density for $y_{i}(t,\al)$ is given by 
$$f_t(E,\al):= 2E \rho_t(E^2,\al),\quad 0\le \al \le 1.$$
Similarly we can define $f_{\lambda,t}$ and $f_{\mu,t}$. Moreover, the quantiles of $f_t(E,\al)$ are exactly given by $\sqrt{\gamma_i(t,\al)}$. Now with Lemma \ref{coro_shoruse} and Lemma \ref{lem_rigidty z}, we can easily conclude the following lemma.
\begin{lemma}
We have the following rigidity estimate of singular values:
\begin{equation}\label{rigiditywhy}
\sup_{0 \leq \alpha \leq 1} \sup_{0 \leq t \leq 10 t_1}\left|y_i(t,\alpha)-\sqrt{\gamma_i(t,\alpha)} \right| \prec i^{-1/3}n^{-2/3} , \quad  1\le i \le c_* n .
\end{equation}
Let $\e,\tau>0$ be sufficiently small constants. For $0 \leq E \leq \tau n^{-2\epsilon} t_0^2$, we have that for any constant $D>0$,
\begin{equation}\label{eq_corcontrol1sing}
f_t\left(\sqrt{E_+(t,\alpha)}-E,\al\right)=f_{\mu,t}\left(\sqrt{E_\mu(t)}-E\right)\left[ 1+\OO(n^{\epsilon}t/t_0 + n^{-D})\right],
\end{equation}
and
\begin{equation}\label{eq_corcontrol4sing}
\max_{1\le i \le c_\tau n^{1-3\e}t_0^3}|\wh{\gamma}_i(t,\alpha)-\wh{\gamma}_i(t,0)| \leq \left(n^{\epsilon}\frac{t}{t_0}+n^{-D}\right) \frac{i^{2/3}}{n^{2/3}},
\end{equation}
where we introduced the short-hand notation $\wh{\gamma}_i(t,\alpha):=\sqrt{E_+(t, \alpha)}-\sqrt{\gamma_i(t,\alpha)}.$
\end{lemma}
\begin{proof}
The rigidity result \eqref{rigiditywhy} follows directly from Lemma \ref{lem_rigidty z}, \eqref{eq_corcontrol1sing} follows from \eqref{eq_corcontrol1} together with \eqref{E+al1}, and \eqref{eq_corcontrol4sing} can be proved easily using \eqref{eq_corcontrol1sing}.
\end{proof}

\subsection{Short-range approximation}\label{subsec_short}

As in \cite{edgedbm}, we will build a short-range approximation for the interpolating processes $\{z_i(t,\al)\},$ which is based on the simple intuition that the eigenvalues that are far away from the edge have negligible effect on the edge eigenvalues. It turns out that it is more convenient to study the SDEs for singular values $y_i(t,\alpha)$. By Ito's formula, we get that for $1\le i \le p$,
\begin{equation}\label{eq_singsde}
\begin{split}
 \dd y_i(t,\alpha)&=\frac{\dd B_i}{\sqrt{n}}+\frac{1}{2y_i(t,\alpha)} \left( \frac{1}{n} \sum_{j \neq i} \frac{y^2_i(t,\alpha)+y^2_j(t,\alpha)}{y^2_i(t,\alpha)-y^2_j(t,\alpha)}+\frac{n-1}{n}  \right)  \dd t \\
&=\frac{\dd B_i}{\sqrt{n}}+\frac{1}{2n} \sum_{j \neq i}\left( \frac{\dd t}{y_i(t,\alpha)-y_j(t,\alpha)}+  \frac{\dd t}{y_i(t,\alpha)+y_j(t,\alpha)}\right)+\frac{n-p}{2ny_i(t,\alpha)}  \dd t .
\end{split}
\end{equation}
Note that the diffusion term now has a constant coefficient. 
For convenience, we introduce the shifted processes
\begin{equation}\label{eq_connection}
\widetilde{z}_i(t,\alpha):={E_+(t,\alpha)}-z_i(t,\alpha),\quad 
\widetilde{y}_i(t,\alpha):=\sqrt{E_+(t,\alpha)}-y_i(t,\alpha).
\end{equation}
Clearly, we have that $\widetilde{z}_i(t,\alpha) \sim \widetilde{y}_i(t,\alpha).$ 
We see that $\widetilde{y}_i(t,\alpha)$ obeys the SDE
\begin{equation}
\begin{split}
 \dd \wt y_i(t,\alpha)=& -\frac{\dd B_i}{\sqrt{n}}+ \frac{1}{2n} \sum_{j \neq i} \frac{\dd t}{\wt y_i(t,\alpha)-\wt y_j(t,\alpha)} +\frac{\dd \sqrt{E_+(t,\alpha)}}{\dd t}\dd t\\
&-\frac{1}{2n} \sum_{j \neq i} \frac{\dd t}{2\sqrt{E_+(t,\al)}-\wt y_i(t,\alpha)-\wt y_j(t,\alpha)} -\frac{n-p}{2n[\sqrt{E_+(t,\al)}-\wt y_i(t,\alpha)]}  \dd t ,\label{eq_singsde2}
\end{split}
\end{equation}
where $\partial_t \sqrt{E_+(t,\alpha)} $ is given by (\ref{eq_defnpsiderviative0}).



We now define a "short-range" set of indices $\mathcal{A} \subset [ 1,p ] \times  [ 1,p ].$ Let $\mathcal{A}$ be a symmetric set of indices in the sense that $(i,j) \in \mathcal{A}$ if and only if $(j,i) \in \mathcal{A}$, and choose a parameter $\ell:=n^{\omega_{\ell}}$, where 
$\omega_{\ell}>0$ is a constant that will be specified later. Then we define
\begin{equation}\label{eq_defnmathcalA}
\mathcal{A}:=\left\{(i,j): |i-j| \leq \ell(10 \ell^2+i^{2/3}+j^{2/3})\right\} \bigcup \left\{(i,j): i,j>i_*/{2}\right\}
\end{equation}
for $i_*:=c_*n,$
where $c_*$ is the constant as appeared in Lemma \ref{lem_rigidty z}. It is easy to check that for each $i,$ the set $\{j: (i,j) \in \mathcal{A}\}$ consists of consecutive integers. For convenience, we introduce the following short-hand notations {
\begin{equation*}
\sum_j^{\mathcal{A}, (i)}:=\sum_{\substack{j:(i,j) \in \mathcal{A}, j \neq i} }, \quad \sum_{j}^{\mathcal{A}^c, (i)}:=\sum_{\substack{j:(i,j) \notin \mathcal{A}, j \neq i}}.  
\end{equation*}  
}
For each $i,$ we denote $[ i_-, i_+]:=\{j: (i,j)\in \cal A\}$ and 
\begin{equation*}
\mathcal{I}_i(t,\alpha):=[{\wt{\gamma}_{i-}(t,\alpha)}, {\wt{\gamma}_{i+}(t,\alpha)}],\quad \wh{\mathcal{I}}_i(t,\alpha):=[{\wh{\gamma}_{i-}(t,\alpha)}, {\wh{\gamma}_{i+}(t,\alpha)}]
\end{equation*}
where we recall that 
$ \wt\gamma_i(t,\al)$ and $\wh\gamma_i(t,\al)$ are defined below \eqref{eq_corcontrol4} and \eqref{eq_corcontrol4sing}, respectively. Finally, we denote 
\begin{equation*}
\mathcal{J}(t,\alpha):=[-\wt c_V,{\wh{\gamma}_{3i_*/4}}(t,\alpha)],
\end{equation*}
where $\wt c_V>0$ is a small constant depending only on $c_V$. 

Let $\omega_a>0$ be a constant that will be specified later.  The short-range approximation to $\widetilde{y}$ is a process $\widehat{y}$ defined as the solution to the following SDEs for $t\ge n^{-C_0}$ with the same initial data (recall Remark \ref{rem restart})
\begin{equation*}
\widehat{y}_i(t=n^{-C_0},\alpha)=\widetilde{y}_i(t=n^{-C_0},\alpha),\quad 0\le \al\le 1,
\end{equation*}
where $C_0$ is an absolute constant (for example, $C_0=100$ will be more than enough). 
For $1 \leq i \leq n^{\omega_a}$, the SDEs are 
\begin{align}
\dd \widehat{y}_i(t,\alpha)=&  -\frac{\dd B_i}{ \sqrt{n}}+\frac{1}{ 2n } \sum_{j}^{\mathcal{A}, (i)}  \frac{1}{\widehat{y}_i(t,\alpha)-\widehat{y}_j(t,\alpha)} \dd t -\frac{n-p}{2n\sqrt{ E_+(t,0)} }  \dd t  +\frac{\dd \sqrt{E_+(t,0)}}{\dd t}\dd t  \nonumber\\
&- \left[c_n\int_{\cal I_i^c(t,0)} \frac{\sqrt{E_+(t,0)}\rho_t(E_+(t,0)-E,0)}{ E - {E_+(t,0)}+(\sqrt{E_+(t,0)}-\wh{y}_i(t,\alpha))^2 }\dd E\right] \dd t; \label{wtySDE1}
\end{align}
for $n^{\omega_a} < i \leq i_*/2,$ the SDEs are 
\begin{align}
 \dd \widehat{y}_i(t,\alpha)=&  -\frac{\dd B_i}{ \sqrt{n}}+ \frac{1}{ 2n } \sum_{j}^{\mathcal{A}, (i)}  \frac{\dd t }{\widehat{y}_i(t,\alpha)-\widehat{y}_j(t,\alpha)}+\frac{1}{2n} \sum_{j \geq 3i_*4}  \frac{\dd t}{\wt{y}_i(t,\alpha)-\wt{y}_j(t,\alpha)} \nonumber\\
&+\frac{c_n}{2} \int\limits_{\wh{\mathcal{I}}_i^c(t,\alpha)\cap \mathcal{J}(t,\alpha)} \frac{f_t(\sqrt{E_+(t,\alpha)}-E,\alpha) }{\widehat{y}_i(t,\alpha)-E} \dd E\dd t  +\frac{\dd \sqrt{E_+(t,\alpha)}}{\dd t}\dd t\nonumber\\
&-\frac{n-p}{2n[\sqrt{E_+(t,\al)}-\wt y_i(t,\alpha)]}  \dd t  -\frac{1}{2n} \sum_{j \neq i} \frac{\dd t}{2\sqrt{E_+(t,\al)}-\wt y_i(t,\alpha)-\wt y_j(t,\alpha)} ; \label{wtySDE2}
\end{align} 
for $i_*/2 < i \leq p$, the SDEs are
\begin{align}
\dd \widehat{y}_i(t,\alpha)=&  -\frac{\dd B_i}{ \sqrt{n}}+ \frac{1}{ 2n } \sum_{j}^{\mathcal{A}, (i)}  \frac{\dd t }{\widehat{y}_i(t,\alpha)-\widehat{y}_j(t,\alpha)}  +\frac{1}{2n} \sum_{j}^{\cal A^c, (i)}  \frac{\dd t}{\wt{y}_i(t,\alpha)-\wt{y}_j(t,\alpha)} +\frac{\dd \sqrt{E_+(t,\alpha)}}{\dd t}\dd t \nonumber\\
& -\frac{n-p }{2n[\sqrt{E_+(t,\al)}-\wt y_i(t,\alpha)]}  \dd t -\frac{1}{2n} \sum_{j \neq i} \frac{\dd t }{2\sqrt{E_+(t,\al)}-\wt y_i(t,\alpha)-\wt y_j(t,\alpha)}. \label{wtySDE3}
\end{align} 
Corresponding to \eqref{eq_connection}, we denote
\begin{equation}\label{eq_defnzhat}
	\wh{z}_i(t,\alpha):={E_+(t,\alpha)}[\sqrt{E_+(t,\alpha)}-\wh{y}_i(t,\alpha)]^2.
\end{equation}

We now choose the hierarchy of the scale parameters in the following quantities: 
\begin{equation*}
t_0=n^{-1/3+\omega_0}, \quad t_1=n^{-1/3+\omega_1}, \quad \ell=n^{\omega_{\ell}}, \quad \text{and} \quad n^{\omega_a}.
\end{equation*}
In fact, we will choose the constants $\omega_0$, $\omega_1$, $\omega_{\ell}$ and $\omega_a$ such that 
\begin{equation}\label{hierachypara}
0< \omega_1 \le \mathfrak C^{-1} \omega_{\ell} \le \mathfrak C^{-2}\omega_a \le \mathfrak C^{-3}\omega_0 \le \mathfrak C^{-1}
\end{equation}
for some constant $\mathfrak C>0$ that is as large as needed. Here the purpose of the scale $\ell$ is to cut off the effect of the initial data far away from the right edge, since $\widetilde{y}_i(0,\al=1)$ and $\widetilde{y}_i(0,\al=0)$ only match for small $i$. 
Moreover, by choosing scale $\omega_a \ll \omega_0$, we can make use of the matching estimates in Lemma  \ref{coro_shoruse2} to show that the drifting terms in the SDEs with $1\le i \le n^{\omega_a}$ are approximately $\al$ independent.


Next, we show that $\widetilde{y}_i(t,\alpha)$ are good approximations to $\widehat{y}_i(t,\alpha)$. 
Before that, we recall the semigroup approach for first order parabolic PDE. Let $\Omega$ be a real Banach space with a given norm and $\mathcal{L}(\Omega)$ be the Banach algebra of all linear continuous mappings. We say a family of operators $\{T(t): t \geq 0\}$ in $\mathcal{L}(\Omega)$ is a semigroup if 
\begin{equation*}
T(0)=\id, \quad \text{and}\quad  T(t+s)=T(t)T(s) \ \ \text{for all} \ t,s \geq 0. 
\end{equation*}
For a detailed discussion of semigroups of operators, we refer the readers to \cite{Bensoussan2007}.
  
\begin{definition}\label{defn_semigroup}
For any operator $\mathcal{W}\in \cal L(\mathbb{R}^p),$ we denote $\mathcal{U}^{\mathcal{W}}$ as the semigroup associated with $\mathcal{W},$ i.e., $\mathcal{W}$ is the infinitesimal generator of $\mathcal{U}^{\mathcal{W}}.$  Moreover, we denote $\mathcal{U}^{\mathcal{W}}(s,t)$ as the semigroup from $s$ to $t$, that is, $\mathcal{U}^{\mathcal{W}}(s,s)=\id$ and
\begin{equation*}
\partial_t \mathcal{U}^{\mathcal{W}}(s,t)=\mathcal{W}(t) \mathcal{U}^{\mathcal{W}}(s,t),\quad \text{for any $t \geq s$.}
\end{equation*} 
\end{definition}

For the rest of this subsection, we prove the following short-range approximation estimate. 

\begin{lemma}\label{lem_approximationcontrol} 
With high probability, we have that for any constant $\epsilon>0$,
\begin{equation}\label{eq_shortmain}
\sup_{0 \leq \alpha \leq 1} \sup_{n^{-C_0} \leq t \leq 10t_1} \max_{1\le i\le p} |\widetilde{y}_i(t,\alpha)-\wh{y}_i(t,\alpha) | \leq n^{-2/3 + \e+ \omega_1 - 2\omega_{\ell}} . 
\end{equation} 
\end{lemma}
\begin{proof}
We abbreviate $v_i:=\widetilde{y}_i-\widehat{y}_i.$ Subtracting the SDEs for $\widetilde{y}_i$ and $\widehat{y}_i,$ we obtain the following inhomogeneous PDE for $v$:
\begin{equation*}
\partial_t v=(\mathcal{B}_1+\mathcal{V}_1)v+\zeta,
\end{equation*}
where $\mathcal{B}_1$ is a linear operator defined by 
\begin{equation*}
(\mathcal{B}_1 v)_i=-\frac{1}{2n} \sum_j^{\mathcal{A}, (i)} \frac{ v_i - v_j}{(\widetilde{y}_i-\widetilde{y}_j)(\widehat{y}_i-\widehat{y}_j)},
\end{equation*}
and $\cal V_1$ is a diagonal operator defined as follows: $\cal V_1(i)=0$ for $i> i_*/2$; for $1\le i \le n^{\omega_a}$,  
\begin{align*}
  \cal V_1(i)v_i&:= c_n \sqrt{E_+(t,0)}   \int_{\cal I_i^c(t,0)} \frac{\rho_t(E_+(t,0)-E,0)}{ E - {E_+(t,0)}+(\sqrt{E_+(t,0)}-\wh{y}_i(t,\alpha))^2 }\dd E \\
 &\quad - c_n \sqrt{E_+(t,0)} \int_{\cal I_i^c(t,0)} \frac{\rho_t(E_+(t,0)-E,0)}{ E - {E_+(t,0)}+(\sqrt{E_+(t,0)}-\wt{y}_i(t,\alpha))^2 }\dd E \\
 &=- v_i \int_{\cal I_i^c(t,0)} \frac{c_n \sqrt{E_+(t,0)}  [2\sqrt{E_+(t,0)}-\wt{y}_i(t,\alpha)-\wh{y}_i(t,\alpha)]\rho_t(E_+(t,0)-E,0)}{ [E - {E_+(t,0)}+(\sqrt{E_+(t,0)}-\wh{y}_i(t,\alpha))^2] [ E - {E_+(t,0)}+(\sqrt{E_+(t,0)}-\wt{y}_i(t,\alpha))^2]}\dd E ;
\end{align*}
for $n^{\omega_a} < i \le i_*/2$, 
$$ \cal V_1(i)=-\frac{c_n}{2} \int_{\wh{\mathcal{I}}_i^c(t,\alpha)\cap \mathcal{J}(t,\alpha)} \frac{f_t( \sqrt{E_+(t,\alpha)}-E,\alpha) }{(\wt{y}_i(t,\alpha)-E)(\widehat{y}_i(t,\alpha)-E)} \dd E.$$
The term $\zeta$ contains the remaining errors, and we will control its $\ell^\infty$ norm later. 

For the following proof, we assume a rough bound on $\wh y_i(t,\al)$:
\begin{equation}\label{rough remain}
\sup_{0 \leq \alpha \leq 1} \max_{1\le i\le i_*/2} |\wt y_i(t,\al)-\wh y_i(t,\al)| \le n^{-2/3}, \quad \text{for}\ \ \ n^{-C_0} \leq t \leq 10t_1.
\end{equation}
Later, we will remove it with a simple continuity argument. Since $\cal V_1(i)\le 0$, the operator $\cal V_1$ is negative. Then, the semigroup of $\mathcal{B}_1+\mathcal{V}_1$ is a contraction on every $\ell^q(\{ 1,\cdots,p \})$ space. To see this, for $u(s)=\cal U^{\mathcal{B}_1+\mathcal{V}_1}(0,s)u_0$ and $q\ge 1$, we have that 
\begin{align*}
\partial_t \sum_i |u_i(s)|^q &= \sum_i |u_i(s)|^{q-1}\sgn(u_i(s))\left[\left(\mathcal{B}_1 u(s)\right)_i+\mathcal{V}_1(i)u_i(s)\right]  \le \sum_i |u_i(s)|^{q-1}\sgn(u_i(s))\left(\mathcal{B}_1 u(s)\right)_i \\
&= - \frac{1}{4n}\sum_{i,j\in \cal A} \frac{\left[|u_i(s)|^{q-1}\sgn(u_i(s))- |u_j(s)|^{q-1}\sgn(u_j(s))\right] \left[u_i(s) - u_j(s)\right]}{(\widetilde{z}_i-\widetilde{z}_j)(\widehat{z}_i-\widehat{z}_j)}\le 0.
\end{align*}
On the space $l^\infty(\{ 1,\cdots,p \})$, we just need to use $\|u\|_\infty =\lim_{q\to \infty}\|u\|_q$. 
By Duhamel's principle, we have 
\begin{equation*}
v(t)=\int_{n^{-C_0}}^t \mathcal{U}^{\mathcal{B}_1+\mathcal{V}_1}(s,t)\zeta(s)\dd s,
\end{equation*}  
which gives that
\begin{equation}\label{eq_duhamel1}
\norm{v(t)}_{\infty} \leq \int_{n^{-C_0}}^t  \norm{\zeta(s)}_{\infty} \dd s. 
\end{equation}
Next, we provide the bounds on $\norm{\zeta(s)}_{\infty}.$ 
Fist, we have that $\zeta_i(t)=0$ for $i\ge i_*/2$. 
Second, under \eqref{rough remain}, for $n^{\omega_a} < i \leq i_*/2,$ we have that
\begin{align}\label{zetai2}
\zeta_i(t) & =\frac{1}{2n}\sum^{\mathcal{A}^c, (i)}_{j \leq 3i_*/4} \frac{1}{\widetilde{y}_i(t,\alpha)-\widetilde{y}_j(t,\alpha)} - \frac{c_n}{2} \int_{\wh{\mathcal{I}}_i^c(t,\alpha)\cap \mathcal{J}(t,\alpha)} \frac{f_t(\sqrt{E_+(t,\alpha)}-E,\alpha) }{\wt{y}_i(t,\alpha)-E} \dd E .
\end{align}
Decomposing the integral in \eqref{zetai2} according to the quantiles of $f_t$ as $\sum_j  \int^{\sqrt{\gamma_j}}_{\sqrt{\gamma_{j+1}}}$ and using \eqref{rigiditywhy}, we obtain that for any constant $\e>0$,
\begin{equation}\label{eq_zetailarge}
\left| \zeta_i \right| \le  \frac{n^\e}{n^{5/3}} \sum_{j\le 3i_*/4}^{\mathcal{A}^c, (i)} \frac{1}{(\widehat{\gamma}_i-\widehat{\gamma}_j)^2  j^{1/3}}  \leq \frac{C n^{\epsilon}}{n^{1/3}} \sum_{j\le 3i_*/4}^{\mathcal{A}^c,(i)} \frac{i^{2/3}+j^{2/3}}{(i-j)^2 j^{1/3}} 
\end{equation}
with high probability, where for the second inequality we used 
$$|\widehat{\gamma}_i-\widehat{\gamma}_j|\sim |i^{2/3}-j^{2/3}|n^{-2/3} \gtrsim |i-j| (i+j)^{-1/3}n^{-2/3},\quad  \text{for}\ \ (i,j)\notin \cal A.$$ 
Using the inequalities (3.67) and (3.68) of \cite{edgedbm}, we can bound \eqref{eq_zetailarge} by
\begin{equation}\label{zetai22}
|\zeta_i| \leq C\frac{n^{\epsilon}}{n^{1/3} n^{2 \omega_{\ell}}} \quad \text{with high probability.}
\end{equation}


For $1 \leq i \leq n^{\omega_a}$, through a lengthy but straightforward calculation, we find that
\begin{align*}
& \zeta_i =:A_1+A_2+A_3+A_4+A_5+A_6,
\end{align*}
where $A_i$, $1\le i\le 6$, are defined as (recall \eqref{eq_connection})
\begin{align*}
A_1:=& - \frac{\sqrt{E_+(t,\al) - \wt z_i(t,\al)} }{n}\sum_{j}^{\cal A^c,(i)} \frac{1}{z_i(t,\al) - z_j(t,\al)} +c_n\sqrt{E_+(t,\al) - \wt z_i(t,\al)}\int_{\cal I_i^c(t,\al)} \frac{\rho_t(E_+(t,\al)-E,\al)}{E-\wt z_i (t,\al) }\dd E,\\
A_2:=& c_n\sqrt{E_+(t,0) - \wt z_i(t,\al)}\int_{\cal I_i^c(t,0)} \frac{\rho_t(E_+(t,0)-E,0)}{E-\wt z_i (t,\al)}\dd E -c_n\sqrt{E_+(t,\al) - \wt z_i(t,\al)}\int_{\cal I_i^c(t,\al)} \frac{\rho_t(E_+(t,\al)-E,\al)}{E-\wt z_i (t,\al) }\dd E ,\\
A_3:=& c_n\left(\sqrt{E_+(t,0)}-\sqrt{E_+(t,0) - \wt z_i(t,\al)}\right)\int_{\cal I_i^c(t,0)} \frac{\rho_t(E_+(t,0)-E,0)}{E-\wt z_i (t,\al)}\dd E  ,\\
A_4:=&c_n\sqrt{E_+(t,0)} \left[ \int_{\cal I_i^c(t,0)} \frac{\rho_t(E_+(t,0)-E,0)}{ E - {E_+(t,0)}+(\sqrt{E_+(t,0)}-\wt{y}_i(t,\alpha))^2 }\dd E-\int_{\cal I_i^c(t,0)} \frac{\rho_t(E_+(t,0)-E,0)}{E-\wt z_i (t,\al)}\dd E\right],\\
A_5:=&-\frac{n-p}{2n\sqrt{E_+(t,\al)-\wt z_i(t,\al)}}   + \frac{n-p}{2n\sqrt{E_+(t,0)-\wt z_i(t,\alpha)}} +\frac{\dd\sqrt{ E_+(t,\al) }}{\dd t}- \frac{\dd\sqrt{ E_+(t,0) }}{\dd t},\\
A_6:=&-\frac{n-p}{2n\sqrt{E_+(t,0)-\wt z_i(t,\alpha)}}+\frac{n-p}{2n \sqrt{E_+(t,0)} }  -\frac{1}{2n} \sum_{j}^{\cal A, (i)} \frac{1}{2\sqrt{E_+(t,\al)}-\wt y_i(t,\alpha)-\wt y_j(t,\alpha)}.
\end{align*}

First, for term $A_6$, we notice that for $1\le i \le n^{\omega_a}$ there are at most $\OO(n^{2\omega_a/3 + \omega_{\ell}})$ many indices $j$ such that $(i,j)\in \cal A$. On the other hand, by \eqref{rigiditywtz} and \eqref{rough remain}, we have $|\wt z_i(t,\alpha)| \le n^{-2/3+2\omega_a/3 + \e}$ with high probability. Hence, we can bound that
\begin{equation}\label{boundA6}
|A_6|\le n^{-2/3+\omega_a},\quad \text{with high probability}.
\end{equation}
Next, using $\rho_t(E_+(t,0)-E,0)=\OO(\sqrt{E})$, we can bound that
$$\int_{\cal I_i^c(t,0)} \frac{\rho_t(E_+(t,0)-E,0)}{|E-\wt z_i (t,\al)|}\dd E=\OO(1),$$
which immediately gives 
\begin{equation}\label{boundA3}
|A_3|\lesssim |\wt z_{i}(t,\al)|\le n^{-2/3+\omega_a} ,\quad \text{with high probability}.
\end{equation}
For $A_4$, we have 
\begin{align}\label{intA4}
 |A_4|&\lesssim \int_{\cal I_i^c(t,0)}\frac{|\wt y_i(t,\al)| |\sqrt{E_+(t,0)}-\sqrt{E_+(t,\al)}|\rho_t(E_+(t,0)-E,0)}{ [E - {E_+(t,0)}+(\sqrt{E_+(t,0)}-\wt{y}_i(t,\alpha))^2][E-\wt z_i (t,\al)] }\dd E.
\end{align}
Note that for $E\in \cal I_i^c(t,\al)$, we have
$$
|E-\wt z_i (t,\al)|   \gtrsim n^{-2/3+2\omega_{\ell}} + i^{2/3}n^{-2/3},
$$
and
$$ |E - {E_+(t,0)}+(\sqrt{E_+(t,0)}-\wt{y}_i(t,\alpha))^2| \gtrsim n^{-2/3+2\omega_{\ell}}+ i^{2/3}n^{-2/3} .
$$ 
Thus, we can bound the integral on the right-hand side of \eqref{intA4} by
$$ \int_{\cal I_i^c(t,0)}\frac{ \rho_t(E_+(t,0)-E,0)\dd E}{ [E - {E_+(t,0)}+(\sqrt{E_+(t,0)}-\wt{y}_i(t,\alpha))^2][E - \wt z_i(t,\al)] }\lesssim n^{1/3-\omega_{\ell}} .$$
Together with $|\sqrt{E_+(t,0)}-\sqrt{E_+(t,\al)}|\lesssim t$ by \eqref{E+al1} and the rigidity estimate \eqref{rigiditywhy} for $|\wt y_i(t,\al)|$, we get that for any constant $\e>0$,
\begin{align}
 |A_4|&\lesssim t|\wt y_i(t,\al)|n^{1/3-\omega_{\ell}} \lesssim n^{-1/3+\omega_1}\left( \frac{i^{2/3}}{n^{2/3}} +\frac{n^\e}{i^{1/3}n^{2/3}}\right) n^{1/3-\omega_{\ell}} \le n^{-2/3+\omega_a} ,\label{boundA4}
\end{align}
with high probability. The term $A_1$ can be handled in exactly the same way as $B_1$ in (3.71) of \cite{edgedbm} and we get that for any constant $\e>0$,
\begin{align}
 |A_1|&\le  n^{-1/3-2\omega_{\ell}+\e}+n^{-1/2+\e} , \quad \text{with high probability.}\label{boundA1}
\end{align}
Finally, using the definitions of $m_t(\wt z_i(t,\al),0)$ and $m_t(\wt z_i(t,\al),\al)$ we can write $A_2+A_5$ as
\begin{equation}\label{B1+B2}
\begin{split}
 A_2+A_5 & =\frac{\dd\sqrt{ E_+(t,\al) }}{\dd t}+ c_n\sqrt{E_+(t,\al) - \wt z_i(t,\al)} \re m_t(E_+(t,\al) - \wt z_i(t,\al),\al) \\
& - \frac{\dd\sqrt{ E_+(t,0) }}{\dd t} - c_n \sqrt{E_+(t,0) - \wt z_i(t,\al)} \re m_t(E_+(t,0) - \wt z_i(t,\al),0)   \\
&-\frac{1-c_n}{2\sqrt{E_+(t,\al)-\wt z_i(t,\al)}}   + \frac{1-c_n}{2n\sqrt{E_+(t,0)-\wt z_i(t,\alpha)}}\\
&+c_n\sqrt{E_+(t,\al) - \wt z_i(t,\al)}\int_{\cal I_i(t,\al)} \frac{\rho_t(E_+(t,\al)-E,\al)}{E-\wt z_i (t,\al) }\dd E \\
& - c_n\sqrt{E_+(t,0) - \wt z_i(t,\al)}\int_{\cal I_i(t,0)} \frac{\rho_t(E_+(t,0)-E,0)}{E-\wt z_i (t,\al)}\dd E = B_1+B_2+ \OO(t^2) ,
\end{split}
\end{equation}
where 
 \begin{align*}
B_1&:=c_n \sqrt{E_+(t,\al) - \wt z_i(t,\al)}\int_{\cal I_i(t,\al)} \frac{\rho_t(E_+(t,\al)-E,\al)}{E-\wt z_i (t,\al) }\dd E   - c_n\sqrt{E_+(t,0) - \wt z_i(t,\al)}\int_{\cal I_i(t,0)} \frac{\rho_t(E_+(t,0)-E,0)}{E-\wt z_i (t,\al)}\dd E, \\
B_2&:= \Psi_t(\wt z_{i}(t,\al),\al)-  \Psi_t(\wt z_{i}(t,\al),0) .
\end{align*}
In the second step of \eqref{B1+B2}, we used \eqref{eq_deriofsing} and \eqref{eq_defnpsiderviative2}. Now using Lemma \ref{coro_shoruse2}, we can bound that for any constant $\e>0$,
\begin{equation}\label{boundB2}
|B_2| \le \frac{n^{\e} t^{1/2}}{t_0^{1/2}}|\wt z_i(t,\al)|^{1/2}+\frac{n^{\e}|\wt z_i(t,\al)|}{t_0}  + \OO(t^2) \lesssim n^{-1/3 +\e+ \omega_a/3 + \omega_1/2-\omega_0/2}
\end{equation}
with high probability, where in the second step we used $|\wt z_i(t,\al)|\lesssim n^{-2/3 + 2\omega_a/3}$ by \eqref{rigiditywtz} because the largest index $i_+$ is at most $\OO(n^{\omega_a})$. 

 It remains to bound $B_1$:
\begin{align*}
B_1&=B_{11}+B_{12},
\end{align*}
where 
\begin{align*}
B_{11}&:=c_n\sqrt{E_+(t,\al) - \wt z_i(t,\al)}\left[\int_{\cal I_i(t,\al)} \frac{\rho_t(E_+(t,\al)-E,\al)}{E-\wt z_i (t,\al) }\dd E-\int_{\cal I_i(t,0)} \frac{\rho_t(E_+(t,0)-E,0)}{E-\wt z_i (t,\al)}\dd E\right] ,\\
B_{12}&:=c_n\left[ \sqrt{E_+(t,\al) - \wt z_i(t,\al)}- \sqrt{E_+(t,0) - \wt z_i(t,\al)}\right]\int_{\cal I_i(t,0)} \frac{\rho_t(E_+(t,0)-E,0)}{E-\wt z_i (t,\al)}\dd E.
\end{align*}
  The term $B_{11}$ can be bounded in the same way as (3.82) of \cite{edgedbm}, which gives that for any constant $\e>0$,
\begin{align}
 |B_{11}|&\le  n^{-1/3 + \e + 2\omega_a/3+\omega_1 -\omega_0},\quad \text{with high probability. }\label{boundB11}
\end{align}
For $B_{12}$, we need to obtain a bound on 
\begin{align}\label{detereq0}
\int_{\cal I_i(t,0)} \frac{\rho_t(E_+(t,0)-E,0)}{E-\wt z_i (t,\al)}\dd E.
\end{align}
Note that this is a principal value, so we need to deal with the logarithmic singularity at $\wt z_i (t,\al)$. 
First assume that $i \geq n^{\delta}$ for some $\delta<\omega_{\ell}/10$. Then with \eqref{rigiditywtz} and \eqref{E+al1}, it is easy to check that $\widetilde{z}_i(t,\alpha)$ is away from the boundary of $\mathcal{I}_i(t,0)$ at least  by a distance $n^{-2}$ with high probability, and that $|\widetilde{z}_i(t,\alpha)| \geq n^{-2}$ with high probability. Then we can bound
\begin{align*}
 \left|\int_{\cal I_i(t,0)} \frac{\rho_t(E_+(t,0)-E,0)}{E-\wt z_i (t,\al)}\dd E\right|  & \leq  \int_{\mathcal{I}_i(t,0) , |E-\widetilde{z}_i(\alpha, t)|>n^{-50}} \frac{\sqrt{|E|}}{|\widetilde{z}_i(t,\alpha)-E|} \dd E \\
&+\left| \int_{|E-\widetilde{z}_i(t,\alpha)| \le n^{-50}} \frac{\rho_t(E_+(t,0)-E,0)-\rho_t(E_+(t,0)-\wt z_{i}(t,\al),0)}{\widetilde{z}_i(t,\alpha)-E} \dd E \right|=:D_1+D_2.
\end{align*} 
For the term $D_1$, using $ \sqrt{|E|}=\OO(n^{-1/3+\omega_a/3})$ for $E\in \cal I_{i}(0,t)$, we obtain that
$$D_1\lesssim n^{-1/3+\omega_a/3}\int_{\mathcal{I}_i(t,0) , |E-\widetilde{z}_i(\alpha, t)|>n^{-50}} \frac{\dd E}{|\widetilde{z}_i(t,\alpha)-E|} \lesssim n^{-1/3+\omega_a/3}\log n.$$
For the term $D_2$, using Lemma \ref{lem_partialm} we obtain that
\begin{align*}
|\rho_t(E_+(t,0)-E,0)-\rho_t(E_+(t,0)-\wt z_{i}(t,\al),0)|& \lesssim \frac{|E-\wt z_{i}(t,\al)|}{\min(t,|E_+(t,0)-\wt z_{i}(t,\al)|^{1/2})} \le n |E-\wt z_{i}(t,\al)|,
\end{align*}
which gives that
$$D_2\lesssim n  \int_{|\widetilde{z}_i(t,\alpha)-E| < n^{-50}} \dd E  \le 2 n^{-49}.$$
Next, we consider the case with $i < n^{\delta}.$ It suffices to assume that $|\widetilde{z}_i(t,\alpha)| \leq n^{-100}$, because otherwise we can obtain an estimate in the same way as the case $i \ge n^{\delta}.$ Then we decompose \eqref{detereq0} as
\begin{align*}
 \int_{\cal I_i(t,0)} \frac{\rho_t(E_+(t,0)-E,0)}{E-\wt z_i (t,\al)}\dd E &= \int_{\cal I_i(t,0), E\ge n^{-50}} \frac{\rho_t(E_+(t,0)-E,0)}{E-\wt z_i (t,\al)}\dd E + \int_{3\wt z_i(t,\al)/2\le E < n^{-50}} \frac{\rho_t(E_+(t,0)-E,0)}{E-\wt z_i (t,\al)}\dd E\\
& + \int_{\wt z_i(t,\al)\le E < 3\wt z_i(t,\al)/2} \frac{\rho_t(E_+(t,0)-E,0)}{E-\wt z_i (t,\al)}\dd E + \int_{0\le E < \wt z_i(t,\al)/2} \frac{\rho_t(E_+(t,0)-E,0)}{E-\wt z_i (t,\al)}\dd E\\
&=:F_1+F_2+F_3+F_4.
\end{align*}
The term $F_1$ can be estimated in the same way as $D_1$. For term $F_2$, we used that $\rho_t(E_+(t,0)-E,0)=\OO(\sqrt{E}) $ to bound the integral as $|F_2|\lesssim n^{-25} $. If $\wt z_i (t,\al)\le 0$, then we have $F_3=F_4=0$. Otherwise, for $F_4$ we have 
$$|F_4| \le \int_{0\le E < \wt z_i(t,\al)/2} E^{-1/2}\dd E \lesssim |\wt z_i(t,\al)|^{1/2}\le n^{-50},$$
and for $F_3$ we have
\begin{align*}
|F_3|&\le \int_{\wt z_i(t,\al)\le E < 3\wt z_i(t,\al)/2} \frac{|\rho_t(E_+(t,0)-E,0)-\rho_t(E_+(t,0)-\wt z_i(t,\al),0)|}{|E-\wt z_i (t,\al)|}\dd E \lesssim |\wt z_i(t,\al)|^{1/2}\le n^{-50},
\end{align*}
where in the second step we used that $|\rho_t(E_+(t,0)-E,0)-\rho_t(E_+(t,0)-\wt z_i(t,\al),0)|\le |\wt z_i(t,\al)|^{-1/2}|E-\wt z_i(t,\al)|$. Combining the above estimates, we get that for any constant $\e>0$,
$$ \left|\int_{\cal I_i(t,0)} \frac{\rho_t(E_+(t,0)-E,0)}{E-\wt z_i (t,\al)}\dd E\right|\le n^{-1/3+\omega_a +\e},\quad \text{with high probability},$$
which further implies that 
\begin{equation}\label{boundB22}
|B_{12}| \lesssim n^{-1/3+\omega_a +\e}t =n^{-2/3+\e+\omega_a +\omega_1}.
\end{equation} 

In sum, combining \eqref{boundA6}, \eqref{boundA3}, \eqref{boundA4}, \eqref{boundA1}, \eqref{boundB2}, \eqref{boundB11} and \eqref{boundB22} and using the hierarchy of parameters \eqref{hierachypara}, we obtain that for $1\le i \le n^{\omega_a}$,
\begin{align}\label{zetahard1}
|\zeta_i(t)| \le  n^{-1/3-2\omega_{\ell}+\e} \quad \text{with high probability},
\end{align}
 for any constant $\e>0$.
Then, combining \eqref{zetai22} and \eqref{zetahard1}, we obtain that for any constant $\e>0$,
\begin{align}\label{zetahard2}
\|\zeta(t)\|_{\infty} \le n^{-1/3-2\omega_{\ell}+\e} \quad \text{with high probability},
\end{align}
 uniformly in all $n^{-C_0}\le t\le t_1$ under the assumption \eqref{rough remain}. Plugging it into \eqref{eq_duhamel1}, we get
\begin{equation*}
\norm{v(t)}_{\infty} \leq tn^{-1/3-2\omega_{\ell}+\e} =n^{-2/3+\e+\omega_1-2\omega_{\ell}},
\end{equation*}
which concludes \eqref{eq_shortmain} under \eqref{rough remain}. Note that the right hand side of \eqref{eq_shortmain} is much smaller than $n^{-2/3}$ on the right-hand side of \eqref{rough remain}. Then using a simple continuity argument we can remove the assumption \eqref{rough remain}. In fact, the continuity argument is deterministic in nature because $v$ satisfies a system of deterministic equations conditioning on the trajectories of $\{\wt y_i(t,\al)\}$ and $\{\wh y_i(t,\al)\}$. In fact, we can pick a high probability event $\Xi$, on which the rigidity \eqref{rigiditywtz} and the local law, Theorem \ref{thm_local2}, hold for all $n^{-C_0}\le t\le t_1$. Then, we can perform the continuity argument on $\Xi$. 
\end{proof}

Before concluding this section, we record the following rigidity estimates. 
\begin{corollary}\label{coro_rigityestimate} 
Let $i \leq n^{3 \omega_{\ell}+\delta}$ for a constant $0<\delta<\omega_{\ell}-\omega_1.$ Then 
we have 
\begin{equation*}
\sup_{0 \leq t \leq 10t_1}\left|  \widehat{y}_i(t,\alpha)- \widehat{\gamma}_i(t,\alpha) \right| \prec i^{-1/3} n^{-2/3} . 
\end{equation*}
\end{corollary}
\begin{proof}
This is an immediate consequence of Lemma \ref{lem_approximationcontrol} and  \eqref{rigiditywhy}.  
\end{proof}

\subsection{Proof of Theorem \ref{thm_firstkey}}


Our goal is to bound $|\widehat{y}_i(t_1,\al=1)-\widehat{y}_i(t_1,\al=0)|$. For this purpose, we shall study the partial derivative $u_i(t,\alpha):=\partial_{\alpha} \widehat{y}_i(t,\alpha).$  With \eqref{wtySDE1}--\eqref{wtySDE3}, we find that $u=(u_i(t,\alpha):1\le i \le p)$ satisfies the PDE
\begin{equation}\label{eq_ui}
\partial_t u=\mathcal{L} u+\zeta^{(0)}.
\end{equation}
Here the operator $\mathcal{L}$ is defined as $\mathcal{L}=\mathcal{B}+\mathcal{V},$ where $\cal B$ is defined by
\begin{equation}\label{defn BB}
(\mathcal{B} u)_i= - \frac{1}{ 2n} \sum_{j}^{\mathcal{A}, (i)}  \frac{u_i - u_j}{[\widehat{y}_i(t,\alpha)-\widehat{y}_j(t,\alpha)]^2} . 
\end{equation}
$\cal V$ is a diagonal operator with $(\mathcal{V} u)_i=\mathcal{V}_i u_i$, where $\cal V_i$'s are defined as
\begin{equation}\label{calV1}\mathcal{V}_i:=- 2c_n\sqrt{E_+(t,0)}\int_{\cal I_i^c(t,0)} \frac{[\sqrt{E_+(t,0)}-\wh{y}_i(t,\alpha)] \rho_t(E_+(t,0)-E,0)}{ [E - {E_+(t,0)}+(\sqrt{E_+(t,0)}-\wh{y}_i(t,\alpha))^2]^2 }\dd E  \end{equation}
for $1 \leq i \leq n^{\omega_a}$,
\begin{equation}\label{calV2} \cal V_i:=-\frac{c_n}{2} \int_{\wh{\mathcal{I}}_i^c(t,\alpha)\cap \mathcal{J}(t,\alpha)} \frac{f_t(\sqrt{E_+(t,\alpha)}-E,\alpha) }{[\widehat{y}_i(t,\alpha)-E]^2} \dd E \end{equation}
for $n^{\omega_a} < i \leq i_*/2$,
and $\mathcal{V}_i=0$ for $i_*/2 < i \leq p.$ With the same discussion as the one below \eqref{rough remain}, we know that the semigroup of $\cal L$ is a contraction on every $\ell^q(\{1,\cdots, p\})$ space. The random forcing term $\zeta^{(0)}$ comes from the $\partial_\al$ derivatives of all the other terms, and we notice that $\zeta^{(0)}_i=0$ when $1\le i\le n^{\omega_a}$. For $i>n^{\omega_a}$, it is easy to check that for some constant $C>0$,   
\begin{equation}\label{roughzeta}
\max_{i > n^{\omega_a}} |\zeta^{(0)}_i| \leq  n^C \quad \text{ with high probability.}
\end{equation}

Next, we define a long range cut-off of $u$. Fix a small constant $\delta_v>0$ and let $v$ be the solution to the following homogeneous equation 
\begin{equation}\label{eq_vi}
\partial_t v=\mathcal{L} v, \quad v_i(n^{-C_0})=u_i(n^{-C_0}) \mathbf{1}_{\{1 \leq i \leq \ell^3 n^{\delta_v}\}}. 
\end{equation}
Then, we have the following proposition, which essentially states that the $u_i$'s with indices far away from the edge have a negligible effect on the solution. 

\begin{proposition}\label{prop_key} 
With high probability, we have 
\begin{equation*}
\sup_{n^{-C_0} \leq t \leq 10t_1} \sup_{1 \leq i \leq \ell^3} |u_i(t,\al)-v_i(t,\al)| \leq n^{-100}. 
\end{equation*}
\end{proposition}

One can see that Proposition \ref{prop_key} is an immediate consequence of the following finite speed of propagation estimate, whose proof is postponed to Section \ref{sec_finitespeedcal}.  

\begin{lemma}\label{lem_finitespeedcal} For any small constant $\delta>0,$ we have that for $a \geq \ell^3 n^{2\delta}$ and $b \leq \ell^3 n^{\delta}$,  \begin{equation*}
\sup_{n^{-C_0}  \leq s \leq t \leq 10t_1} \left[\mathcal{U}_{ab}^{\mathcal{L}}(s,t)+\mathcal{U}_{ba}^{\mathcal{L}}(s,t) \right] \leq n^{-D} \quad \text{with high probability},
\end{equation*}  
for any large constant $D>0$.
\end{lemma}

\begin{remark}
In fact, we have $\mathcal{U}_{ab}^{\mathcal{L}}(s,t) \ge 0$ and $\mathcal{U}_{ba}^{\mathcal{L}}(s,t)\ge 0$ by maximum principle. More precisely, define $v_i(t)= \exp ( - \int_0^t \cal V_i(s)\dd s)u_i(t)$. Then $v=(v_i:1\le i \le p)$ satisfies the equation $\partial_t v= \cal B v.$ If $v_i(s)\ge 0$ for all $i$ at time $s$, we claim that $v_i(t)\ge 0$ for all $i$ at any time $t\ge s$. To see this, at any time $ t' \in [s, t]$, suppose $v_j(t')=\min\{v_i(t'):1\le i \le p\}$ is the smallest entry of $v(t')$. Then with \eqref{defn BB}, we can check that $\partial_t v_j(t')= (\cal B v(t'))_j \ge 0$, i.e. the smallest entry of $v$ will always increase. Hence the entries of $v$ can never be negative at any time $t\ge s$.
\end{remark}

\begin{proof}[Proof of Proposition \ref{prop_key}]
Fix a $n^{-C_0} \le t \le 10 t_1$, by Duhamel's principle we have that
\begin{align*}
u(t,\al)-v(t,\al)= \mathcal{U}^\mathcal{L}(n^{-C_0},t)[u(n^{-C_0},\al)-v(n^{-C_0},\al)] + \int_{n^{-C_0}}^t  \mathcal{U}^\mathcal{L}(s,t)\zeta^{(0)}(s)\dd s.
\end{align*}
Since $u_i(n^{-C_0},\al)-v_i(n^{-C_0},\al)=0$ for $i \le \ell^3 n^{\delta_v}$ and $\zeta_i^{(0)}(s)=0$ for $i\le n^{\omega_a}$, we can conclude the proof using Lemma \ref{lem_finitespeedcal} and  \eqref{roughzeta}.
\end{proof}

Another key ingredient is the following energy estimate. We postpone its proof until we complete the proof of Theorem \ref{thm_firstkey}. Here we have fixed the starting time point to be $n^{-C_0}$, but the same conclusion holds for any other starting time by the semigroup property. 
\begin{proposition}\label{prop_energyestimate} 
For any small constant $\delta_1>0$, consider a vector $w \in \mathbb{R}^p$ with $w_i=0$ for $i \geq \ell^3 n^{\delta_1}.$ Then, for any constants $\epsilon,\eta>0$ and fixed $q\ge 1$, there exists a constant $C_q>0$ independent of $\epsilon$ and $\eta$ such that for all $2n^{-C_0} \leq t \leq 2t_1,$
\begin{equation}\label{eq generalq}
\norm{\mathcal{U}^{\mathcal{L}}(n^{-C_0},t) w}_{\infty} \leq C(q,\eta) \left( \frac{n^{C_q \eta+\epsilon}}{n^{1/3}t} \right)^{3(1-6\eta)/q} \norm{w}_q. 
\end{equation} 
\end{proposition}

With all the above preparations, we are now ready to give the proof of Theorem \ref{thm_firstkey}.

\begin{proof}[Proof of Theorem \ref{thm_firstkey}]
Fix any $1\le i \le \fa$, by \eqref{E+1} and \eqref{E+2} we have that with high probability,
\begin{align}
 \left|[\lambda_{i}(t_1)-E_{\lambda}(t_1)]-[\mu_i(t_1)-E_{\mu}(t_1)]\right| &\le \left| \widetilde{z}_i(t_1,1)-\widetilde{z}_i(t_1,0) \right| + |E_+(t_1,1)-E_{\lambda}(t_1)|+ |E_+(t_1,0)-E_{\mu}(t_1)|  \nonumber\\
&\le \left| \widetilde{z}_i(t_1,1)-\widetilde{z}_i(t_1,0) \right| + n^{-2/3-\tau} \nonumber
\end{align}
 for some small constant $\tau>0$. Recalling \eqref{eq_connection}, we have that 
\begin{align}
\left| \widetilde{z}_i(t_1,1)-\widetilde{z}_i(t_1,0) \right| & \le |\wt y_i(t_1,1)-\wt y_i(t_1,0)|\left( 2\sqrt{E_+(t_1,0)} -y_{i}(t_1,0) \right) \nonumber \\
&\quad +|\wt y_i(t_1,1)| \left|2 \sqrt{E_+(t_1,1)} -y_{i}(t_1,1)- 2\sqrt{E_+(t_1,0)} + y_{i}(t_1,0) \right| \nonumber\\
&\lesssim |\wt y_i(t_1,1)-\wt y_i(t_1,0)|+ \OO_\prec\left(n^{-2/3} t_1\right),\nonumber
\end{align}
where in the second step we used \eqref{eq_edgecomparebound2} and the rigidity estimate \eqref{rigiditywhy}. Together with Lemma \ref{lem_approximationcontrol}, we obtain that with high probability,
\begin{equation}\label{closematch1}
\left|[\lambda_{i}(t_1)-E_{\lambda}(t_1)]-[\mu_i(t_1)-E_{\mu}(t_1)]\right|  \lesssim |\wh y_i(t_1,1)-\wh y_i(t_1,0)|+n^{-2/3-\tau}
\end{equation}
for some small constant $\tau>0$. 
Now, we write that 
\begin{equation*}
\widehat{y}_i(t_1,1)-\widehat{y}_i(t_1,0)=\int_0^1 u_i(t_1,\alpha) \dd \alpha  .
\end{equation*}
Applying Proposition \ref{prop_key} (together with a simple stochastic continuity argument to pass to all $0\le \al\le 1$), we get that 
\begin{equation}\label{closematch2}
\left| \widehat{y}_i(t_1,0)-\widehat{y}_i(t_1,1)\right| \leq n^{-50}+\left|\int_0^1 v_i(t_1,\alpha) \dd \alpha\right|,
\end{equation}
with high probability. By \eqref{eq_rigidityalphaequal1} and (\ref{eq_edgecontrolcontrol}), we have that at $t=0$,
$$|z_j(t=0,0)- z_j(t=0,1)|\prec n^{-2/3-\omega_0} + j^{-1/3}n^{-2/3},\quad 1\le j \le \ell^3n^{\delta_v},$$
for a small enough constant $\delta_v>0$. Moreover, at $t= n^{-C_0}$ the eigenvalues are perturbed at most by $n^{-C_0/2}$, so we can calculate that 
\begin{equation*}
\norm{v(n^{-C_0},\al)}_4 \prec n^{-2/3-\omega_0} (\ell^3n^{\delta_v})^{1/4}+ n^{-2/3}\le 2n^{-2/3},\quad 0\le \al \le 1 .
\end{equation*}
Finally, using Proposition \ref{prop_energyestimate} with $q=4$, we find that 
\begin{equation*}
\left| \int_0^1 v_i(t_1,\alpha) \dd \alpha \right| \prec  n^{-2/3-\omega_1/2 }.
\end{equation*}
Inserting it into \eqref{closematch2} and further into \eqref{closematch1}, we conclude the proof.
\end{proof}


The proof of Proposition \ref{prop_energyestimate} is almost the same as the one for Lemma 3.11 in \cite{edgedbm}, so we only give an outline of it. 
\begin{proof}[Proof of Proposition \ref{prop_energyestimate}]
 The proof relies on Lemma \ref{lem_finitespeedcal} and the estimates in the following lemma. 

\begin{lemma}\label{lem_holder}
 Fix a constant $0<\delta_1<\omega_{\ell}-\omega_1$.
Let $w\in \R^p$ be a vector such that $w_i=0$ for $i \geq \ell^3 n^{\delta_1}.$ For any constants $\eta,\epsilon>0$, there is a constant $C>0$ independent of $\epsilon$ and $\eta$, and a constant $c_\eta>0$ such that the following estimates hold with high probability for all $n^{-C_0} \leq s \leq t \leq 5t_1$:
\begin{equation}\label{eq_holder1}
\norm{\mathcal{U}^{\mathcal{L}}(s,t) w}_2 \leq \left( \frac{n^{C\eta+\epsilon}}{c_\eta n^{1/3}(t-s) }\right)^{\frac32(1-6\eta)} \norm{w}_1,
\end{equation}
and 
\begin{equation}\label{eq_holder2}
\norm{(\mathcal{U}^{\mathcal{L}}(s,t))^\top w}_2 \leq \left( \frac{n^{C\eta+\epsilon}}{c_\eta n^{1/3}(t-s) }\right)^{\frac32(1-6\eta)} \norm{w}_1.
\end{equation}
\end{lemma}
\begin{proof}
The proof is very similar to the ones for \cite[Lemma 3.13]{edgedbm}, \cite[Proposition 10.4]{Bourgade2014EdgeUO} and \cite[Section 10]{Erdos2015}. More precisely, our operator $\cal L$ is almost the same as the operator $\cal L$ in \cite[Lemma 3.13]{edgedbm}, where the only difference is the form of $\cal V$. However, the $\cal V_i$'s in \eqref{calV1} and \eqref{calV2} satisfy exactly the same estimates as the $\cal V_i$'s in \cite{edgedbm}. So we omit the details of the proof.  
\end{proof}

Now, we complete the proof of Proposition \ref{prop_energyestimate}. Fix constants 
$0<\delta_1<\delta_2<\omega_{\ell}-\omega_1.$ We define the indicator function $\mathcal{X}_2(i)=\mathbf{1}_{\{1 \leq i \leq \ell^3 n^{\delta_2}\}}$ and let $\mathcal{X}_2$ be the associated digonal operator. For any $v\in \R^p$ with $\norm{v}_1=1$, we decompose that 
\begin{equation*}
\langle \mathcal{U}^\mathcal{L}w,v \rangle=\langle w,  (\mathcal{U}^\mathcal{L})^\top v\rangle=\langle w, (\mathcal{U}^\mathcal{L})^\top \mathcal{X}_2 v \rangle+\langle  w, (\mathcal{U}^\mathcal{L})^\top (1-\mathcal{X}_2) v \rangle, 
\end{equation*}
where we have abbreviated $\cal U^{\cal L}\equiv \mathcal{U}^\mathcal{L}(n^{-C_0},t)$. For the second term, with Lemma \ref{lem_finitespeedcal}, we obtain that 
\begin{equation*}
\left| \langle  w, (\mathcal{U}^\mathcal{L})^\top (1-\mathcal{X}_2) v \rangle \right| \leq n^{-100} \norm{w}_1\norm{v}_1 \le n^{-99}\|w\|_2\norm{v}_1
\end{equation*}
with high probability. 
For the first term, with Lemma \ref{lem_holder} and Cauchy-Schwarz inequality, we get that for any constant $\eta>0,$
\begin{align*}
\langle w, (\mathcal{U}^\mathcal{L})^\top \mathcal{X}_2 v \rangle \leq \norm{w}_2 \norm{(\mathcal{U}^\mathcal{L})^\top \mathcal{X}_2 v}_2 \leq \norm{w}_2\left( \frac{n^{C\eta+\epsilon}}{c_\eta n^{1/3}(t-n^{-C_0}) }\right)^{\frac32(1-6\eta)}   \norm{v}_1.
\end{align*}
By $\ell^1$--$\ell^\infty$ duality and using $t\ge 2n^{-C_0}$, we find that 
\begin{equation*}
\norm{\mathcal{U}^\mathcal{L} w}_{\infty} \le  C(\eta) \left( \frac{n^{C\eta+\epsilon}}{n^{1/3}t  }\right)^{\frac32(1-6\eta)}   \norm{w}_2 .
\end{equation*}
Consequently, by the semigroup property, we find that
\begin{equation*}
\begin{split}
 \norm{\mathcal{U}^\mathcal{L}(n^{-C_0},t)w}_{\infty} & = \norm{\mathcal{U}^\mathcal{L}(2t/3,t)\mathcal{U}^\mathcal{L}(n^{-C_0},2t/3)w}_{\infty}\\
& \le C(\eta) \left( \frac{n^{C\eta+\epsilon}}{ n^{1/3}t  }\right)^{\frac32(1-6\eta)}   \norm{\mathcal{U}^\mathcal{L}(n^{-C_0},2t/3)w}_{2} \leq  C(\eta) \left( \frac{n^{C\eta+\epsilon}}{ n^{1/3}t  }\right)^{3(1-6\eta)}  \norm{w}_1  ,
\end{split}
\end{equation*}
where we used Lemma \ref{lem_holder} again in the last step. Finally, the estimate \eqref{eq generalq} for general $q$ follows from the standard interpolation argument.  
\end{proof}

\subsection{Proof of Lemma \ref{lem_finitespeedcal}}\label{sec_finitespeedcal}
 
Finally, in this section, we prove the finite speed of propagation estimate, Lemma \ref{lem_finitespeedcal}. 
For simplicity of notations, we shift the time such that the starting time point is $t=0$.
%
We first prove a result for fixed $s$. 

\begin{lemma}\label{lem_speedlemmaone} Fix a small constant $0<\delta<\omega_{\ell}-\omega_1.$ For any $a \geq \ell^3 n^{\delta},$ $b \leq  \ell^3 n^{\delta}/2$ and fixed $0\le s\le 10t_1$, we have that for any large constant $D>0$,
\begin{equation*}
\sup_{t: s \leq t \leq 10t_1} \left[ \mathcal{U}_{ab}^{\mathcal{L}}(s,t)+\mathcal{U}_{ba}^{\mathcal{L}}(s,t) \right] \leq n^{-D},\quad \text{with high probability.}
\end{equation*} 
\end{lemma}

We postpone its proof until we complete the proof of Lemma \ref{lem_finitespeedcal}. We need to use the following lemma in order to extend the result in Lemma \ref{lem_speedlemmaone} to all $0\le s \le t \le 10t_1$ simultaneously.

\begin{lemma}\label{lem_speedlemmatwo} Let $u\in \R^p$ be a solution of $\partial_t u=\mathcal{L} u$ with $u_i(0) \geq 0$ for $1\le i \le p$. Then, for $0 \leq t \leq 10t_1,$ we have 
\begin{equation*}
\frac{1}{2} \sum_{i} u_i(0) \leq \sum_{i} u_i(t) \leq \sum_i u_i(0). 
\end{equation*}
\end{lemma} 
\begin{proof}
Summing over $i$ and using $\sum_i (\cal Bu)_i=0$, we get that 
\begin{equation*}
\partial_t \sum_i u_i =\sum_i \mathcal{V}_i u_i.
\end{equation*}
We now bound \eqref{calV1} and \eqref{calV2}. Using \eqref{rigiditywhy} and Lemma \ref{lem_approximationcontrol}, we have that with high probability,
$$E - {E_+(t,0)}+(\sqrt{E_+(t,0)}-\wh{y}_i(t,\alpha))^2\gtrsim E + n^{-2/3+2\omega_{\ell}},\quad 1\le i \le n^{\omega_a},\ \ E\in \cal I_i^c(t,0).$$
Together with the estimate $\rho_t(E_+(t,0)-E,0)\sim \sqrt{E}$, we get that for $1\le i \le n^{\omega_a}$, 
\begin{align*}
0\le -\mathcal{V}_i \lesssim  \int_{\cal I_i^c(t,0)} \frac{\sqrt{E}}{ |E + n^{-2/3+2\omega_{\ell}}|^2 }\dd E \lesssim n^{1/3-\omega_{\ell}}.\end{align*}
We can get the same bound for \eqref{calV2}. 
Then, applying Gronwall's inequality to
\begin{equation*}
-\left( C n^{1/3-\omega_{\ell}} \right) \sum_i u_i \le \partial_t \sum_i u_i \leq 0,  
\end{equation*}
we can conclude the proof.
\end{proof}


Now, we can complete the proof of Lemma \ref{lem_finitespeedcal}.

\begin{proof}[Proof of Lemma \ref{lem_finitespeedcal}]
Fix any constant $0< \e <\delta$, $a \geq \ell^3 n^{2\delta}$ and $b \leq \ell^3 n^{\delta}$.  By the semigroup property, we have
\begin{equation}\label{finaleqfinite}
\mathcal{U}_{bi}^{\mathcal{L}}(n^{-C_0},t)=\sum_j \mathcal{U}_{bj}^{\mathcal{L}}(s,t) \mathcal{U}_{ji}^{\mathcal{L}}(n^{-C_0},s)  \geq \mathcal{U}_{ba}^{\mathcal{L}}(s,t) \mathcal{U}_{ai}^{\mathcal{L}}(n^{-C_0},s).
\end{equation}
By Lemma \ref{lem_speedlemmatwo}, we find that $\sum_i \mathcal{U}^{\mathcal{L}}_{ai}(n^{-C_0},s) \geq 1/2.$ Moreover, by Lemma \ref{lem_speedlemmaone} we have that $\mathcal{U}_{ai}^\mathcal{L}(n^{-C_0},s) \leq n^{-100}$ for any $i \leq \ell^3 n^{\delta+\epsilon}.$ This implies that there exists an $i_* \geq \ell^3 n^{\delta+\e}$ such that $\mathcal{U}^\mathcal{L}_{ai_*}(n^{-C_0},s) \geq (4n)^{-1}.$ However, by Lemma \ref{lem_speedlemmaone} we have that $\mathcal{U}_{bi_*}^{\mathcal{L}}(0,t) \leq n^{-D}$ for any large constant $D>0.$ Thus picking $i=i_*$ in \eqref{finaleqfinite}, we get that
$\mathcal{U}_{ba}^{\mathcal{L}}(s,t) \leq n^{-D+2}.$
This finishes the proof for the estimate on $\mathcal{U}_{ba}^{\mathcal{L}}(s,t)$. The estimate on $\mathcal{U}_{ab}^{\mathcal{L}}(s,t)$ can be proved in a similar way.   
\end{proof}


It remains to prove Lemma \ref{lem_speedlemmaone}. The strategy was first developed in \cite{Bourgade2017}, and later used in \cite{dysonbulk,edgedbm} to study the symmetric DBM for Wigner type matrices. Our proof is similar to the ones for \cite[Lemma 4.2]{dysonbulk} and \cite[Lemma 4.1]{edgedbm}, so we will not write down all the details. 
\begin{proof}[Proof of Lemma \ref{lem_speedlemmaone}] 
We focus on the case $s=0$ and the general case can be dealt with similarly using a simple time shift. 
Let $\psi$ be a smooth function satisfying the following properties: (i) $\psi(x)=-x$ for $ |x| \leq \ell^2 n^{-2/3+2\delta/3},$ (ii)
$\psi'(x)=0$ for $|x| > 2\ell^2 n^{-2/3+2\delta/3},$ (iii) $\psi$ is decreasing, (iv) $|\psi(x)-\psi(y)| \leq |x-y|$ and $|\psi'(x)| \leq 1$, and (v) $|\psi^{''}(x)| \leq C \ell^{-2} n^{2/3-2 \delta/3}$ for some constant $C>0$. 
Similar to \cite[Lemma 4.1]{edgedbm}, we now consider a solution of 
\begin{equation*}
\partial_t f=\mathcal{L} f,\quad \text{with }\ f_i(0)=\delta_{q_*},
\end{equation*} 
for any $q_* \geq q:=\ell^3 n^{\delta}.$ 
Let $\nu>0$ be a fixed constant and define the functions
\begin{equation*}
\phi_k:=\exp\left[\nu \psi(\widehat{y}_k(t,\alpha)-\widehat{\gamma}_q(t,\alpha))\right],\quad v_k:=\phi_k f_k, \quad F(t):=\sum_k v_k^2.
\end{equation*}
For our proof, we will choose a specific $\nu$ later. 
By Ito's formula, we find that $F$ satisfies the SDE 
\begin{align}
\dd F & =-\sum_{(i,j) \in \mathcal{A}} \mathcal{B}_{ij}(v_i-v_j)^2 \dd t + 2 \sum_i  {\mathcal{V}}_i v_i^2 \dd t \label{eq_dftfirst}\\
& +\sum_{(i,j) \in \mathcal{A}} \mathcal{B}_{ij} v_i v_j \left(\frac{\phi_i}{\phi_j}+\frac{\phi_j}{\phi_i} -2\right) \dd t \label{eq_dftsecond} \\
&+ 2\nu \sum_{i} v_i^2 \psi'(\widehat{y}_i-\wh\gamma_q)\dd (\widehat{y}_i-\wh\gamma_q) \label{eq_dftthird} \\
&+ \sum_{i} v_i^2 \left( \frac{\nu^2}{n}[\psi'(\widehat{y}_i-\wh\gamma_q)]^2+\frac{\nu}{n} \psi^{''}(\widehat{y}_i-\widehat\gamma_q) \right) \dd t \label{eq_dftfourth},
\end{align}
where we denoted 
\begin{equation*}
\mathcal{B}_{ij}=\frac{1}{2n} \frac{1}{(\widehat{y}_i(t,\alpha)-\widehat{y}_j(t,\alpha))^2} .
\end{equation*}  
Now, we choose a proper stopping time. Let $\tau_1$ be the stopping time such that for $t<\tau_1$, Lemmas \ref{lem_rigidty z} and \ref{lem_approximationcontrol} hold true for a sufficiently small constant $0<\epsilon<\delta/100.$ Note that $\tau_1 \ge 10 t_1$ with high probability. Let $\tau_2$ be the first time such that $F \geq 10.$ Then, we define the stopping time 
\begin{equation*}
\tau:=\min\{\tau_1, \tau_2, 10t_1\}.
\end{equation*}  
For the rest of the proof, we only consider times with $t<\tau$. We will show that with a suitable choice of $\nu,$ we actually have $\tau=10 t_1$ with high probability.  

We now deal with each term in \eqref{eq_dftfirst}-\eqref{eq_dftfourth}. First, (\ref{eq_dftfirst}) is a dissipative term, so it only decreases the size of $F(t)$. 
By Corollary \ref{coro_rigityestimate}, we see that $\psi'(\widehat{y}_i-\wh\gamma_q)=0$ when $i>C' \ell^3 n^{\delta}$ for a large enough constant $C'>0$. Moreover, if $i \leq C' \ell^3 n^{\delta}$ and $(i,j) \in \mathcal{A},$ then $j \leq C \ell^3 n^{\delta}$ for some constant $C>0$ depending on $C'$. Thus the nonzero terms in (\ref{eq_dftsecond}) must satisfy that $i,j \leq C \ell^3 n^{\delta}$ for a large enough constant $C>0$. Then, by Corollary \ref{coro_rigityestimate}, for $i ,j \leq C \ell^3 n^{\delta}$ satisfying $(i,j) \in \mathcal{A}$, we have
\begin{equation*}
\left|\widehat{y}_i-\widehat{y}_j \right| \lesssim   \ell^2 n^{-2/3+\delta/3} . 
\end{equation*}
Now, with the Taylor expansion of $e^{-x}+e^x-2$, we get that if 
\begin{equation}\label{eq_conditionnu}
 {\nu \ell^2 n^{-2/3+\delta/3}}  \leq C_1
\end{equation} 
for some constant $C_1>0,$  then 
\begin{equation}\label{eq_controldft2dnterm}
(\ref{eq_dftsecond}) \lesssim \frac{\nu^2}{n}  \sum_{  (i,j) \in \mathcal{A}}(v_i^2+v^2_j) \mathbf{1}_{\{\phi_j \neq \phi_i\}}\dd t \leq \frac{\nu^2 \ell^3 n^{2 \delta/3}}{n} F(t)\dd t.
\end{equation}
The term (\ref{eq_dftfourth}) can be easily bounded as 
\begin{equation}\label{eq_controldftfourth}
(\ref{eq_dftfourth}) \lesssim\left( \frac{\nu^2}{n}+\frac{\nu \ell^{-2}}{ n^{1/3+2\delta/3}} \right) F(t)\dd t. 
\end{equation}
It remains to control (\ref{eq_dftthird}). Since $\psi'(\widehat{y}_i-\zeta_q) \neq 0$ only when $i \leq C \ell^3 n^{\delta} \ll n^{\omega_a}$, thus $\wh y_i$ satisfies the SDE \eqref{wtySDE1}, which gives that
\begin{align}
 \dd  \left[\widehat{y}_i(t,\alpha)-\wh\gamma_q(t,\alpha)\right]= & -\frac{\dd B_i}{ \sqrt{n}} +\frac{1}{ 2n } \sum_{j}^{\mathcal{A}, (i)}  \frac{1}{\widehat{y}_i(t,\alpha)-\widehat{y}_j(t,\alpha)} \dd t  +\left( \frac{\dd \sqrt{E_+(t,0)}}{\dd t} -\frac{\dd\wh\gamma_q(t,\alpha)}{\dd t}\right)\dd t -\frac{n-p}{2nE_+(t,0)}  \dd t \nonumber \\
&- \left[c_n\int_{\cal I_i^c(t,0)} \frac{\sqrt{E_+(t,0)}\rho_t(E_+(t,0)-E,0)}{ E - {E_+(t,0)}+(\sqrt{E_+(t,0)}-\wh{y}_i(t,\alpha))^2 }\dd E\right] \dd t.  \label{eq_lastterm} 
\end{align}
By Burkholder-Davis-Gundy inequality and Markov's inequality, we find that for any constant $\e>0$,
\begin{equation}\label{eq_controldeffouthone}
\sup_{0 \leq t \leq \tau} \nu \left|\int_0^t \sum_i v_i^2 \psi'(\widehat{y}_i-\wh\gamma_q) \frac{\dd B_i}{\sqrt{n}} \right| \leq  n^{\epsilon}\nu \left( \frac{ n^{\omega_1}}{n^{4/3}} \right)^{1/2}
\end{equation} 
with high probability.
Moreover, with the same arguments for (4.17) of \cite{edgedbm}, we obtain that
\begin{equation}\label{eq_controldftfourthree}
\frac{\nu}{n} \sum_{(i,j) \in \mathcal{A}} \frac{v_i^2 \psi'(\widehat{y}_i-\wh\gamma_q)}{\widehat{y}_i-\widehat{y}_j} \leq \sum_{(i,j) \in \mathcal{A}} \frac{\mathcal{B}_{ij}}{100} (v_i-v_j)^2+C \left(\frac{\nu n^{\omega_{\ell}}}{n^{1/3}}+\frac{\nu^2 n^{3 \omega_{\ell}+2 \delta/3}}{n} \right) F(t)
\end{equation}
for large enough constant $C>0$. The main difference from the argument in \cite{edgedbm} is about the term 
\begin{align}
& \frac{\dd \sqrt{E_+(t,0)}}{\dd t} -\frac{\dd\wh\gamma_q(t,\alpha)}{\dd t} -\frac{n-p}{2n\sqrt{E_+(t,0)}}   - c_n \int_{\cal I_i^c(t,0)} \frac{\sqrt{E_+(t,0)} \rho_t(E_+(t,0)-E,0)}{ E - {E_+(t,0)}+(\sqrt{E_+(t,0)}-\wh{y}_i(t,\alpha))^2 }\dd E     \nonumber \\
&= c_n \sqrt{E_+(t,0)} \left[m_t \left(\left(\sqrt{E_+(t,0)}-\wh{y}_i(t,\alpha)\right)^2,0 \right) -m_{t}\left(E_+(t,0),0\right) \right] \nonumber\\
&+c_n\int_{\cal I_i(t,0)} \frac{\sqrt{E_+(t,0)} \rho_t(E_+(t,0)-E,0)}{ E - {E_+(t,0)}+(\sqrt{E_+(t,0)}-\wh{y}_i(t,\alpha))^2 }\dd E  -\frac{\dd\widehat\gamma_q(t,\alpha)}{\dd t}  + \OO(t),  \label{detererr}
\end{align}
where we used \eqref{eq_deriofsing} and $\zeta_t(E_+(t,0),0)=E_+(t,0)+\OO(t)$ in the derivation. 
By the square root behavior of $m_t$ around the right edge, we have that 
$$|m_t ((\sqrt{E_+(t,0)}-\wh{y}_i(t,\alpha))^2,0 ) -m_{t}(E_+(t,0),0) | \lesssim \sqrt{|\wh{y}_i(t,\alpha)|} \lesssim n^{-1/3+\omega_{\ell} + \delta/3}$$
with high probability, where we used \eqref{rigiditywhy} in the last step. For the second term on the right-hand side of \eqref{detererr}, it can be bounded in the same way as \eqref{detereq0} and we can get that
$$\left|\int_{\cal I_i(t,0)} \frac{\sqrt{E_+(t,0)} \rho_t(E_+(t,0)-E,0)}{ E - {E_+(t,0)}+(\sqrt{E_+(t,0)}-\wh{y}_i(t,\alpha))^2 }\dd E\right| \lesssim n^{-1/3+\omega_{\ell} + \delta/3}, $$
with high probability. Finally, we know that $\widehat\gamma_q(t,\alpha)$ satisfies 
$$\int_0^{\widehat\gamma_q(t,\alpha)} \wt \rho(t, E)\dd E= \frac{q}{p},\quad \wt \rho(t, E):= f_t(\sqrt{E_+(t,\al)}-E,\al).$$
Taking the derivative of this equation, we get
$$ \frac{\dd\widehat\gamma_q(t,\alpha)}{\dd t}=\frac{-1}{\wt \rho(t, \wh\gamma_q(t,\al))} \int_0^{\widehat\gamma_q(t,\alpha)} \partial_t \wt \rho(t, E)\dd E.$$
It is trivial to check that $\partial_t \wt \rho(t, E)=\OO(1)$, and we have $\wt \rho(t, \wh\gamma_q)\sim \sqrt{\wh\gamma_q(t,\al)}$ by \eqref{sqrtdensity}. Thus, we obtain from the above equation that 
$$\left| \frac{\dd\widehat\gamma_q(t,\alpha)}{\dd t}\right|\lesssim \sqrt{\wh\gamma_q(t,\al)}\lesssim n^{-1/3+\omega_{\ell} + \delta/3} .$$
Combining the above estimates, we get 
\begin{equation}\label{eq_controllast}
|(\ref{detererr})|=\OO(n^{-1/3+\omega_{\ell} + \delta/3}).
\end{equation}
Now, combining (\ref{eq_controldft2dnterm}), (\ref{eq_controldftfourth}),(\ref{eq_controldeffouthone}), (\ref{eq_controldftfourthree}) and \eqref{eq_controllast}, we find that if $\nu$ satisfies the condition of (\ref{eq_conditionnu}), then with high probability,
\begin{equation*}
\partial_t F(t) \leq C \left(\frac{\nu^2   n^{3\omega_l+2 \delta/3}}{n}+\frac{\nu n^{\omega_{\ell}+\delta/3}}{n^{1/3}} \right)F(t). 
\end{equation*}
Then, by Gronwall's inequality, we get that
$$\sup_{0\le s\le \tau} F(s) \le F(0)+ C \left(\frac{\nu^2  n^{3\omega_l+2 \delta/3 + \omega_1}}{n^{4/3}}+\frac{\nu n^{\omega_{\ell}+\omega_1+\delta/3}}{n^{2/3}} \right)$$
with high probability. Hence, choosing $\nu= n^{2/3-2\omega_l-\delta/3}$, we obtain by continuity that $\tau=10t_1$ with high probability, i.e., 
$$\sup_{0\le s\le 10t_1} F(s) \le 10,\quad \text{ with high probability}.$$ 
Now, notice that if $i \leq \ell^3 n^{\delta}/2,$ we have that 
\begin{equation*}
\nu |\widehat{y}_i(t,\alpha)-\wh\gamma_q(t,\al)| \gtrsim n^{\delta/3},\quad \text{with high probability}.
\end{equation*} 
Then, by the definition of $F(t)$ and Markov's inequality, we obtain that $\mathcal{U}^{\mathcal{L}}_{iq_*}(0,t) \leq n^{-D}$ for any large constant $D>0$ if $i \leq \ell^3 n^{\delta}/2$ and $q_*\ge \ell^3 n^{\delta}$. The proof for $\mathcal{U}^{\mathcal{L}}_{q_* i}$ is the same by setting $\psi\to -\psi.$ 
\end{proof}

\end{document}